\documentclass[reqno,oneside,12pt]{amsart}
\usepackage{
  amsmath,
  amssymb,
  amsthm,
  geometry,
  paralist,
  thmtools,
  tikz,
  tikz-cd,
  verbatim,
  relsize, %mathlarger
  bbm, %blackboard numbers
  dsfont,
  mathtools,
  xifthen, % ifthenelse package
  mathrsfs, % math script
  calligra, %for calligraphy
  enumitem
}

\usepackage{tangle}

\usepackage[T1]{fontenc}
\usepackage[all,cmtip]{xy}
\usepackage{caption}
\usepackage{subcaption}
\usetikzlibrary{angles}

\renewcommand\labelenumi{(\roman{enumi})}
\renewcommand\theenumi\labelenumi

\newcommand{\unit}{\mathds{1}}
\newcommand{\Fun}{\operatorname{Fun}}

\RequirePackage{ifpdf}
\ifpdf
  \IfFileExists{pdfsync.sty}{\RequirePackage{pdfsync}}{}
  \RequirePackage[pdftex,
   colorlinks = true,
   urlcolor = myblue, % \href{...}{...} external (URL)
   citecolor = mygreen, % \cite{...}
   linkcolor = myred, % \ref{...} and \pageref{...}
   bookmarksopen=true,
   unicode, psdextra]{hyperref}
\else
  \RequirePackage[hypertex, unicode, psdextra]{hyperref}
\fi
\hypersetup{
    colorlinks=true,
    linkcolor=blue,
    filecolor=magenta,      
    urlcolor=cyan,
    citecolor = blue
} 

\RequirePackage{ae, aecompl, aeguill} % to have pretty pdf

\usepackage{quiver}

\usepackage{cleveref} %cref 

\usepackage{adjustbox}

\usepackage[maxbibnames=99,
			maxalphanames=10, 
			style=numeric,
            url=false]{biblatex}
\usepackage{xurl}

\usepackage{dynkin-diagrams}

\usepackage[boxsize=.5em]{ytableau}

\usetikzlibrary{arrows.meta} % not arrows, but arrows.meta
\newcommand{\strbarbto}{\mathbin{\tikz[baseline]\draw[arrows={-angle 90},yshift=0.75ex, thick] (0,0) -- (.95em,0);}}

\makeatletter
\providecommand\@dotsep{5}
\makeatother

\newcommand{\<}{\langle}
\renewcommand{\>}{\rangle}
\newcommand{\wt}{\widetilde}

\DeclareMathOperator{\id}{id}

\newcommand{\xra}{\xrightarrow}

\DeclareMathOperator{\colim}{colim}
\DeclareMathOperator{\coker}{coker}

\DeclareMathOperator{\Hom}{Hom}
\DeclareMathOperator{\iHom}{\mathscr{H}\text{\kern -3pt {\calligra\large om}}\,}
\DeclareMathOperator{\rad}{rad}
\DeclareMathOperator{\End}{End}
\DeclareMathOperator{\iEnd}{\mathscr{E}\text{\kern -3pt {\calligra\large nd}}\,}
\DeclareMathOperator{\Ext}{Ext}
\DeclareMathOperator{\iExt}{\mathscr{E}\text{\kern -3pt {\calligra\large xt}}\,\,}

\DeclareMathOperator{\img}{im}

\DeclareMathOperator{\Mat}{Mat}
\DeclareMathOperator{\op}{op}
\DeclareMathOperator{\fp}{fp}
\DeclareMathOperator{\ind}{ind}

\DeclareMathOperator{\coev}{coev}
\DeclareMathOperator{\ev}{ev}

\DeclareMathOperator{\proj}{\mathsf{proj}}
\newcommand{\Cproj}[1]{\ifthenelse{\isempty{#1}}{\cC\text{-}\proj}{#1\text{-}\proj}}
\DeclareMathOperator{\Gproj}{\mathsf{Gproj}}
\DeclareMathOperator{\inj}{\mathsf{inj}}
\newcommand{\Cinj}[1]{\ifthenelse{\isempty{#1}}{\cC\text{-}\inj}{#1\text{-}\inj}}

\DeclareMathOperator{\add}{\mathsf{add}}

\DeclareMathOperator{\pd}{pd}
\newcommand{\ipd}{\cC\!\pd}
\DeclareMathOperator{\Gpd}{Gpd}
\DeclareMathOperator{\dom}{dom}

\DeclareMathOperator{\FPdim}{FPdim}
\DeclareMathOperator{\fin}{fin\!}

\newcommand{\leftdual}[1]{#1^*}
\newcommand{\rightdual}[1]{{}^*\!#1}

\newcommand{\rightop}[1]{{}^\vee\!#1}

\newcommand{\Xdim}[1]{#1.\dim}

\newcommand{\mR}{\mathbb{R}}
\makeatletter
\let\@tl\triangleleft% Save these because we will overwrite them later.
\let\@tr\triangleright% ^
\newcommand*{\@smallTriangle}[2]{\vcenter{\hbox{\scalebox{0.75}{\ensuremath{#1#2}}}}}
\newcommand*{\@medTriangle}[2]{\vcenter{\hbox{\scalebox{0.90}{\ensuremath{#1#2}}}}}
\newcommand*{\lact}{\mathbin{\mathpalette\@smallTriangle\@tr}}
\newcommand*{\ract}{\mathbin{\mathpalette\@smallTriangle\@tl}}
\let\@btl\blacktriangleleft% Save these because we will overwrite them later.
\let\@btr\blacktriangleright% ^
\newcommand*{\blact}{\mathbin{\mathpalette\@medTriangle\@btr}}
\newcommand*{\bract}{\mathbin{\mathpalette\@medTriangle\@btl}}
\renewcommand*{\triangleright}{\lact}
\renewcommand*{\triangleleft}{\ract}
\renewcommand*{\blacktriangleright}{\blact}
\renewcommand*{\blacktriangleleft}{\bract}

\newcommand{\cP}{\mathcal{P}}
\newcommand{\cI}{\mathcal{I}}
\newcommand{\cC}{\mathcal{C}}
\newcommand{\cF}{\mathcal{F}}
\newcommand{\cM}{\mathcal{M}}
\newcommand{\cN}{\mathcal{N}}
\newcommand{\cA}{\mathcal{A}}
\newcommand{\cB}{\mathcal{B}}
\newcommand{\cX}{\mathcal{X}}
\renewcommand{\Vec}{\mathsf{Vec}}
\renewcommand{\vec}{\mathsf{vec}}
\newcommand{\Mod}{\mathsf{Mod}}
\renewcommand{\mod}{\mathsf{mod}}
\newcommand{\rep}{\mathsf{rep}}

\declaretheorem[numberwithin=section]{theorem}
\declaretheorem[sibling=theorem]{lemma}
\declaretheorem[sibling=theorem]{corollary}

\declaretheorem[sibling=theorem]{proposition}

\declaretheorem[sibling=theorem, style=remark]{remark}
\declaretheorem[sibling=theorem, style=definition]{definition}
\declaretheorem[sibling=theorem, style=definition]{example}

\newtheorem*{theorem*}{Theorem}

\crefname{lemma}{Lemma}{Lemma}
  \crefname{corollary}{Corollary}{Corollary}
  \crefname{theorem}{Theorem}{Theorem}
  \crefname{definition}{Definition}{Definition}
   \crefname{proposition}{Proposition}{Proposition}
\crefname{question}{Question}{Question}
 \crefname{section}{Section}{Section} 
   \crefname{construction}{Construction}{Construction}
   \crefname{generalization}{Generalization}{Generalization}
  \crefname{construction}{Construction}{Construction}
  \crefname{notation}{Notation}{Notation}
   \crefname{example}{Example}{Example}
  \crefname{remark}{Remark}{Remark}
  \crefname{fact}{Fact}{Fact}
  \crefname{conjecture}{Conjecture}{Conjecture}
  \crefname{motivation}{Motivation}{Motivation}  
  \crefname{figure}{Figure}{Figure}  
  \crefname{assumption}{Assumption}{Assumption}

\newcommand{\xtwoheadrightarrow}[2][]{%
  \xrightarrow[#1]{#2}\mathrel{\mkern-14mu}\rightarrow
}

\title[Subcategories with homological properties]{A subcategory with remarkable homological properties via actions of tensor categories}
\author{Kevin Coulembier}
\address{KC: School of Mathematics and Statistics, University of Sydney, NSW 2006, Australia}
\email{kevin.coulembier@sydney.edu.au}

\author{Edmund Heng}
\address{EH: School of Mathematics and Statistics, University of Sydney, NSW 2006, Australia}
\email{edmund.heng@sydney.edu.au}

\date{}

\begin{document}

\begin{abstract}
    By equipping the category $\cM$ of modules over a finite dimensional algebra with the action of a finite tensor category $\cC$, we introduce a subcategory $\cP\cM$ satisfying numerous remarkable properties: it is generating and cogenerating, functorially finite, closed under Auslander--Reiten translations (i.e.\ precluster tilting), and it computes the finitistic dimension.
    Through a functor category perspective, $\cP\cM$ is identified with the $\cC$-module functors that are projective, and we also prove a $\cC$-equivariant version of the Auslander(--Solberg) correspondence.
    %We show that $\cP\cM$ has an additive generator for tensor product of representation-finite algebras with certain self-injective Nakayama algebras, hence proving that these algebras have finite finitistic dimension.
    %As application, we use $\cP\cM$ to show that tensor products of representation-finite algebras with certain self-injective Nakayama algebras have finite finitistic dimensions.
    As examples, we show that $\cP\cM$ contains only finitely many indecomposable modules for some tensor product algebras.
    For some family of 1-Iwanaga--Gorenstein algebras introduced by Geiss--Leclerc--Schr\"oer, we show that $\cP\cM$ is exactly the subcategory of $\tau$-locally free modules.
\end{abstract}
\maketitle

\tableofcontents

\section{Introduction}
In the representation theory of finite-dimensional algebras, one is interested in the category $\cM = \mod_A$ of finite-dimensional $A$-modules.
While classifying all modules would be the ultimate goal, most algebras are of wild representation-type and understanding all modules for such algebras is unattainable.
Nonetheless, it is often interesting (and useful) to study special subcategories of $\cM$, in which asking for appropriate ``finite-type'' properties for them leads to new ways of understanding different families of algebras.
Examples of such subcategories include the subcategory of bricks, the subcategory of Gorenstein projectives, higher cluster tilting subcategories and more.

\iffalse
\begin{itemize}
    \item (Bricks) $\cM$ has finitely many bricks. This turns out to be equivalent to a wealth of different structures through $\tau$-tilting theory.
    \item (Cohen--Macaulay/Gorenstein-projectives) $\cM$ has finitely many indecomposable Gorenstein projectives. This has been studied via a generalised Auslander correspondence.
    \item $\cM$ has finitely many indecomposable modules $M$ whose proj.dim$(\tau^kM)<\infty$ for all $k \in \mathbb{Z}$. These has been classified by Geiss--Leclerc--Schroer for a family of $1$-Iwanaga--Gorenstein algebras, which carries Weyl group combinatorics.
    \item Higher Auslander correspondence. This was introduced by Iyama and has since been generalised to many different settings
\end{itemize}
\fi

The goal of this paper is to introduce new subcategories of $\cM$ through the (finite) categorical symmetries of $\cM$.
Precisely, we study $\cM$ equipped with a monoidal, exact functor from a \emph{finite tensor category} $\cC$ (\cref{defn:FTC}) into the category of right exact endofunctors of $\cM$:
\begin{equation} \label{eqn:C-mod-structure}
\begin{split}
    \cC&\to \End^{\mathrm{rex}}(\cM)\cong {}_A\mod_A \\
    X &\mapsto X \lact -.
\end{split}
\end{equation}
Under the framework of \cite{EGNO}, such a structure on $\cM$ makes $\cM$ a (finite) \emph{module category over} $\cC$ (\cref{defn:modulecat}).
Note that due to the existence of left and right duals in finite tensor categories, all endofunctors $X\lact -$ will necessarily have both left and right adjoints, and hence are exact.
However, they need not be invertible.

A large family of such structures can be obtained from actions and coactions of Hopf algebras $H$ on $A$ (see e.g.\ \cite{AM-module-cat-Hopf,KW_Hopf-path,CY-Taft_action_fda,KA_Hopf-Ore-path}\footnote{All of these works study classifications of (co)actions of Hopf algebras on algebras $A$ (typically for $A$ which are well understood). Our paper serves to point out that (co)actions of Hopf algebras -- even on complicated algebras $A$ -- are useful for studying the representation theory of $A$.}), which equips $\cM=\mod_A$ with the structure of a module category over the finite tensor category $\mod_{\leftdual{H}}$ and $\mod_H$ respectively.
In particular, having a finite group $G$ acting on $A$, and hence on $\mod_A$, can be viewed as a module category structure on $\mod_A$ over the (semisimple) finite tensor category $\vec_G = \mod_{\Bbbk^G}$.
However, $\cM$ can also be a module category over more ``exotic'' finite tensor categories that are not realised by Hopf algebras (see \cite{Heng_Coxquiver,EH_fusionquiver, Betz_fusion-action-quiver} for examples related to quivers).

Now given a $\cC$-module category structure on $\cM$, we consider the endofunctors coming from projectives in $\cC$: 
\[
\cP \coloneqq \{P\lact - \mid P \in \proj(\cC)\} \subseteq \End^{\mathrm{rex}}(\cM)
\]
and introduce the following subcategory of $\cM$ that behaves like the ``projection''\footnote{While $P\lact(P\lact M)$ is almost never $P\lact M$, we have that $P'\lact(P\lact M) \cong (P'\otimes P)\lact M$ with $P'\otimes P \in \proj(\cC)$, hence $\cP(\cP\cM) = \cP\cM$.} of $\cM$ by $\cP$:
\begin{equation} \label{eqn:PM}
    \cP\cM \coloneqq \add\{ P\lact M \mid P \in \proj(\cC), M \in \cM\} \subseteq \cM.
\end{equation}
We show that $\cP\cM$ satisfies numerous remarkable properties:
\begin{theorem}\label{thm:intro1}
    Let $\cM\cong \mod_A$ be a $\cC$-module category.
    \begin{enumerate}
        \item (\cref{cor:PM-precluster-tilting}) $\cP\cM\subseteq\cM$ is a (1-)precluster tilting subcategory of $\cM$ in the sense of Iyama--Solberg \cite{IS_preclustertilting}, concretely:
            \begin{enumerate}
                \item $\cP\cM$ is a generator and cogenerator of $\cM$.
                \item $\cP\cM$ is closed under Auslander--Reiten translates.
                \item $\cP\cM \subseteq \cM$ is functorially finite (i.e.\ $\cM$ has both left and right $\cP\cM$-approximations).
            \end{enumerate}
            
        \item \label{item:pd-from-PM} (\cref{prop:finite_pd=finite_ipd}) For any $M \in \cM$ with projective dimension $\pd(M) = d < \infty$, we have that $\pd(P\lact M) = d$ for all non-zero $P \in \proj(\cC)$.
        \item \label{item:Cproj=Gproj} (\cref{prop:Cprojdimprojdim} and \cref{cor:Cproj=Gproj}) Suppose $\pd(M') \leq d <\infty$ for all $M' \in \cP\cM$.
            Then $\cM$ is $d$-Iwanaga--Gorenstein and the Gorenstein projectives are exactly the $M \in \cM$ such that $P\lact M \in \proj(\cM)$ for all $P\in\proj(\cC)$ (such $M$ are called $\cC$-projectives; see \cref{defn:Cproj}).
    \end{enumerate}
\end{theorem}
We emphasise that being a $\cC$-module category is a \emph{structure} on $\cM$ and \emph{not a property}: even over the same $\cC$, $\cM$ could be equipped with many different $\cC$-module category structures.
This is a generalisation of the fact that $\cM$ can be equipped with many different group actions.
Accordingly, $\cP\cM$ \emph{depends} on the $\cC$-module structure.
In some sense, the ``best'' $\cC$-module structure on $\cM$ is the one that leads to the smallest possible\footnote{Note that we always have $\proj(\cM)\subseteq\cP\cM$, and $\cM$ is said to be \emph{exact} if equality is attained (\cref{defn:modulecat}). From this point of view the best $\cC$-module structures are simply the exact ones. However, \cref{thm:intro1} implies that exact module categories have to be $0$-Iwanaga--Gorenstein (i.e.\ quasi-Frobenius), hence exact module structures cannot be realised on an arbitrary abelian category.} $\cP\cM$.
Thus we make the following definition:
\begin{definition}
    We say that a $\cC$-module category $\cM$ is \emph{$\cC$-rep-finite} if $\cP\cM$ has only finitely many indecomposables, i.e.\ there exists $\overline{M} \in \cP\cM$ such that $\add(\overline{M}) = \cP\cM$.
\end{definition}
One should think of $\cC$-rep-finite as the analogue (and generalisation) of representation-finite in the presence of a $\cC$-module structure.
This is (partly) motivated by \cref{thm:linearisingCfunctor} and the following correspondence classifying $\cC$-rep-finite module categories, which we call the \emph{$\cC$-equivariant version of Auslander(--Solberg) correspondence} ($\ipd(-)$ below denotes the \emph{$\cC$-projective dimension}; see \cref{defn:Cprojdim}):
\begin{theorem}[\cref{thm:equivariant-Aus-cor}]
    There is a bijection between the following two families of $\cC$-module categories (up to $\cC$-module equivalences):
    \begin{enumerate}
        \item $\cC$-rep-finite $\cC$-module categories $\cM \cong \mod_A$;
        \item $\cC$-module categories $\cN \cong \mod_B$ satisfying $\ipd(\cN) \leq 2 \leq \dom.\dim(\cN)$.
    \end{enumerate}
\end{theorem}
Note that our correspondence reduces to the classical Auslander correspondence by choosing $\cC=\vec$ (trivial symmetry), where $\cP\cM=\cM$ and $\ipd(\cN) = \pd(\cN)$.

Analogous to the classical setting, our correspondence is obtained through the consideration of \emph{$\cC$-module functor categories} via an enriched Yoneda embedding; see \cref{sec:functor-cat} for more details.
In particular, under the enriched Yoneda embedding, $\cP\cM$ is identified with the subcategory of functors that are projective (whereas $\cM$ is identified with functors that are Gorenstein projective).

Before we discuss applications, let us illustrate our results above with the following simple (yet illuminating) example:
\begin{example}\label{eg:introA2}
    Assume $\mathrm{char}(\Bbbk)\neq 2$.
    Let $\hat{A}$ be the finite dimensional algebra with the following equivalent descriptions:
    \begin{itemize}
        \item $\hat{A} = \begin{psmallmatrix} B_2 & B_2 \\ 0 & B_2\end{psmallmatrix} = B_2 \otimes_\Bbbk \Bbbk A_2$, where $B_2 = \Bbbk[\varepsilon]/\<\varepsilon^2\>$ and $\Bbbk A_2$ is the $A_2$ quiver algebra.
        \item $\hat{A}$ is the quiver algebra
        $\begin{tikzcd}[ampersand replacement = \&]
        1 
        \ar[loop left,"\varepsilon_1"] 
        \ar[r,"\alpha"]
        \&
        2
        \ar[loop right,"\varepsilon_2"]
        \end{tikzcd}
        $
        modulo relations generated by $\varepsilon_i^2 =0$ and $\alpha\varepsilon_1 = \varepsilon_2\alpha$.
    \end{itemize}
    Note that an $\hat{A}$-module is the same as a morphism of $B_2$-modules and the two indecomposable projective $\hat{A}$-modules are $0 \to B_2$ and $B_2 \xrightarrow{=} B_2$.
    
    One can equip $\cM\coloneqq \mod_{\hat{A}}$ with the structure of a module category over $\cC \coloneqq \mod_H$, where 
    \[
    H \coloneqq T_2(-1) = \Bbbk\<x,g\>/\<g^2=1, x^2=0, gx=-xg\>
    \]
    is the Hopf algebra known as the (4-dimensional) Taft algebra; see \cref{eg:speciesA2}.
    For now the important point is that $\cC$ has exactly two indecomposable projectives $P_0$ and $P_1$ with $P_0\lact M \cong P_1\lact M$ for all $M \in \mod_{\hat{A}}$.
    
    We depict in \cref{fig:introA2} the images of $P_i\lact -$ on the indecomposables, where we draw one thick red arrow $M {\color{red}{\strbarbto}} N$ for each indecomposable summand $N$ of $P_i \lact M$ (similar to a ``McKay graph'' associated to $P_i\lact -$).
    The indecomposable modules that are in $\cP\cM$ are therefore the target of these red arrows.
    We see that $\cP\cM \subsetneq \cM$, and this is typically\footnote{For $\cP\cM \neq \cM$, it is necessary that $\cC$ is a non-semisimple tensor category.} the case.

    It is easy to check from \cref{fig:introA2} that $\cP\cM$ is a precluster tilting subcategory.
    One also sees that all objects in $\cP\cM$ have projective dimension $\leq 1$, so it follows from \cref{thm:intro1} that $\hat{A}$ is 1-Iwanaga--Gorenstein.
    %Note that an object in $\cM$ has finite projective dimension $\leq 1$ if and only if it has $B_2^{\oplus m}$ on the vertices of the quiver (locally free).
    %It is also well-known that $\hat{A}$ is 1-Iwanaga--Gorenstein.
    Moreover, the Gorenstein projectives are exactly the modules whose image under $P_i\lact -$ are projective, so there are exactly 3 indecomposable Gorenstein projectives that are not projective: $0\to \Bbbk$, $\Bbbk \xrightarrow{=} \Bbbk$ and $\Bbbk \hookrightarrow B_2$.
    One also observes that $P_i\lact -$ indeed preserves the projective dimension of modules with finite projective dimension.
    Finally, under the $\cC$-equivariant Auslander correspondence, we get a $\cC$-module category $\cN$ that is equivalent to the category of modules over the endomorphism algebra of the additive generator of $\cP\cM$, which in this case is isomorphic to $B_2\otimes_\Bbbk \mathrm{Aus}(A_2)$ with $\mathrm{Aus}(A_2)$ denoting the (classical) Auslander algebra of the $A_2$ quiver.
\end{example}

\begin{figure}
    \centering
    \tikzset{MK/.style = {draw=red, % just to demonstrate, where MK is used
                      line width=#1, -{Straight Barb[length=3pt]}},
         MK/.default=1pt
        }
    \adjustbox{scale=0.8,center}{%
    \begin{tikzcd}
	& {\color{red} \boldsymbol{0\to B_2}} && {\color{red} \boldsymbol{B_2 \to 0}} & \\
	{0 \to \Bbbk} && {B_2 \xrightarrow{\neq}B_2} && {\Bbbk \to 0} \\
	&& {\color{red} \boldsymbol{B_2 \xrightarrow{=} B_2}} \\
	& {B_2 \twoheadrightarrow \Bbbk} && {\Bbbk \hookrightarrow B_2} \\
	{\Bbbk \to 0} && {\Bbbk\xrightarrow{=}\Bbbk} && {0\to\Bbbk}
	\arrow[MK, from=1-2, to=1-2, loop, in=40, out=140, distance=12mm]
	\arrow[MK, from=1-2, to=1-2, loop, in=60, out=120, distance=5mm]
	\arrow[color={rgb,255:red,191;green,191;blue,191}, from=1-2, to=2-3]
	\arrow[color={rgb,255:red,191;green,191;blue,191}, dashed, from=1-4, to=1-2]
	\arrow[MK, from=1-4, to=1-4, loop, in=40, out=140, distance=12mm]
	\arrow[MK, from=1-4, to=1-4, loop, in=60, out=120, distance=5mm]
	\arrow[color={rgb,255:red,191;green,191;blue,191}, from=1-4, to=2-5]
	\arrow[color={rgb,255:red,191;green,191;blue,191}, from=2-1, to=1-2]
	\arrow[shift left, MK, curve={height=-6pt}, from=2-1, to=1-2]
	\arrow[color={rgb,255:red,191;green,191;blue,191}, from=2-1, to=4-2]
	\arrow[shift right, MK, curve={height=6pt}, from=2-3, to=1-2]
	\arrow[color={rgb,255:red,191;green,191;blue,191}, from=2-3, to=1-4]
	\arrow[shift left, MK, curve={height=-6pt}, from=2-3, to=1-4]
	\arrow[color={rgb,255:red,191;green,191;blue,191}, dashed, from=2-3, to=2-1]
	\arrow[MK, from=2-3, to=3-3]
	\arrow[color={rgb,255:red,191;green,191;blue,191}, from=2-3, to=4-4]
	\arrow[shift right, MK, curve={height=6pt}, from=2-5, to=1-4]
	\arrow[color={rgb,255:red,191;green,191;blue,191}, dashed, from=2-5, to=2-3]
	\arrow[MK, from=3-3, to=3-3, loop, in=220, out=320, distance=12mm]
	\arrow[MK, from=3-3, to=3-3, loop, in=240, out=300, distance=5mm]
	\arrow[color={rgb,255:red,191;green,191;blue,191}, from=3-3, to=4-2]
	\arrow[color={rgb,255:red,191;green,191;blue,191}, from=4-4, to=3-3]
	\arrow[shift left=4, MK, curve={height=-24pt}, from=4-2, to=1-4]
	\arrow[color={rgb,255:red,191;green,191;blue,191}, from=4-2, to=2-3]
	\arrow[shift left=2, MK, curve={height=-6pt}, from=4-2, to=3-3]
	\arrow[color={rgb,255:red,191;green,191;blue,191}, from=4-2, to=5-3]
	\arrow[shift right=3, MK, curve={height=24pt}, from=4-4, to=1-2]
	\arrow[color={rgb,255:red,191;green,191;blue,191}, from=4-4, to=2-5]
	\arrow[shift right=2, MK, curve={height=6pt}, from=4-4, to=3-3]
	\arrow[color={rgb,255:red,191;green,191;blue,191}, dashed, from=4-4, to=4-2]
	\arrow[color={rgb,255:red,191;green,191;blue,191}, from=4-4, to=5-5]
	\arrow[color={rgb,255:red,191;green,191;blue,191}, from=5-1, to=4-2]
	\arrow[MK, from=5-3, to=3-3]
	\arrow[color={rgb,255:red,191;green,191;blue,191}, from=5-3, to=4-4]
	\arrow[color={rgb,255:red,191;green,191;blue,191}, dashed, from=5-3, to=5-1]
	\arrow[color={rgb,255:red,191;green,191;blue,191}, dashed, from=5-5, to=5-3]
    \end{tikzcd}
    }
    \caption{AR quiver of $\cM =\mod_{\hat{A}}$ from \cref{eg:introA2} placed on a M\"obius strip; cf.\ \cite[Figure 5]{GLS_quiverI}. The image of $P_i\lact -$ is illustrated in thick red arrows, and the indecomposables in $\cP\cM$ are in red (sinks).}
    \label{fig:introA2}
\end{figure}

\iffalse
The subcategory $\cP\cM$ behaves as the ``orthogonal'' subcategory to the resolving subcategory of \emph{$\cC$-projectives} that play the role of projectives in the theory of tensor categories:
\begin{equation}
    \begin{split}
        \Cproj{\cC}(\cM) \coloneqq \{ M \in \cM \mid P\lact M \in \proj(\cM) \text{ for all } P \in \proj(\cC)\};
        %\Cinj{\cC}(\cM) &\coloneqq \{ M \in \cM \mid P\lact M \in \inj(\cM) \text{ for all } P \in \proj(\cC)\}.
    \end{split}
\end{equation}
Precisely, we get the following (\cref{lem:PPM}):
\begin{align*}
    \Cproj{\cC}(\cM) \cap \cP\cM = \proj(\cM).
    %\Cinj{\cC}(\cM) \cap \cP\cM &= \inj(\cM).
\end{align*}
Moreover, we have the following chain of containments (\cref{prop:Cproj-Gproj}): 
\begin{align*}
    \proj(\cM) &\subseteq \Cproj{\cC}(\cM) \subseteq \Gproj(\cM),
    %\inj(\cM) &\subseteq \Cinj{\cC}(\cM) \subseteq \Ginj(\cM),
\end{align*}
where $\Gproj(\cM)$ denotes the subcategory of Gorenstein projectives.
\fi

One of our main applications of $\cP\cM$ is in relation to computing the finitistic dimension, whose definition we now recall.
For a subcategory $\cA \subseteq \cM$, the \emph{finitistic dimension} of $\cA$ is defined as follows:
    \[
    \fin.\pd(\cA) \coloneqq \sup \{ \pd(M)
        \mid M\in \cA, \ \pd(M)<\infty  
        \}.
    \]
The (small) \emph{finitistic dimension conjecture}\footnote{During final preparations of this paper, a counterexample to this conjecture was announced by an AI company.} then states that $\fin.\pd(\mod_A)$ is always finite.
We refer the reader to \cite{Z-H_fin-dim-conjecture} for a survey on this topic, noting that it is one of the ``strongest'' forms of homological conjectures; see \cite[\S 13.7]{FDatlas} for a summary.

To this end, note that \cref{thm:intro1}\ref{item:pd-from-PM} implies that one can compute the finitistic dimension of $\cM$ from the typically smaller subcategory $\cP\cM$:
\begin{theorem}[\cref{thm:PM-computes-findim}] \label{thm:intro-PM-findim}
    Let $\cM \cong \mod_A$ be a $\cC$-module category\footnote{The general setup of our paper allows for the case where $\tilde{\cM} \cong \Mod_A$, so that $\fin.\pd(\tilde{\cM}) = \fin.\pd(\cP\tilde{\cM})$ also applies. In other words, $\cP\tilde{\cM}$ computes the big finitistic dimension of $A$.}. We have 
    \[
    \fin.\pd(\cM) = \fin.\pd(\cP\cM).
    \]
    In particular, if $\cM$ is $\cC$-rep-finite, then $\cM$ has finite finitistic dimension.
\end{theorem}

Our next result provides a family of interesting examples of $\cM$ that are $\cC$-rep-finite via tensor product algebras (where $\cM$ itself could be representation-infinite).
\begin{theorem}[\cref{thm:uniserial}]
    Let $A$ be a representation-finite algebra and $B$ be a self-injective Nakayama algebra.
    Suppose that $\mod_B$ can be equipped with an exact (\cref{defn:modulecat}) $\cC$-module category structure for some finite tensor category $\cC$.
    Then $\cM\coloneqq \mod_{B\otimes_\Bbbk A}$ is a $\cC$-module category with 
    \[
    \cP\cM \cong \add\{ P\otimes_\Bbbk M \mid P \in \proj(\mod_B), M \in \mod_A\}.
    \]
    In particular, $\cM$ is $\cC$-rep-finite (and $\fin.\pd(\cM) < \infty$).
\end{theorem}
With respect to $\cM =\mod_{\hat{A}} = \mod_{B_2\otimes_\Bbbk \Bbbk A_2}$ from \cref{eg:introA2}, having 3 indecomposables in $\cP\cM$ agrees with the fact that there is exactly one indecomposable projective in $\mod_{B_2}$ and 3 indecomposable modules in $\mod_{\Bbbk A_2}$.
A full account of the family of self-injective Nakayama algebras $B$ for which we know (to the best of our knowledge) how to realise $\mod_B$ as an exact module category is provided in \cref{eg:tensorproduct} -- whether this is possible for all self-injective Nakayama algebras remains open.

Finally, we study in \cref{sec:species} \emph{representations of $\cC$-species}, which are the same as modules over tensor algebras in $\cC$. 
We show that these form $\cC$-module categories that are always 1-Iwanaga--Gorenstein (\cref{prop:species-1IG}).
We then apply this theory to study a family of 1-Iwanaga--Gorenstein algebras $H(C,D,\Omega)$ introduced by Geiss--Leclerc--Schr\"oer in \cite{GLS_quiverI} by identifying $\cM \coloneqq \mod_{H(C,D,\Omega)}$ with the representation category of an appropriate $\cC$-species (under certain assumptions on the symmetriser $D$), so that $\cM$ is now equipped with a $\cC$-module structure.
When $H(C,D,\Omega)$ is associated with a Dynkin diagram, we recover the subcategory of $\tau$-locally free modules as $\cP\cM$ (see \cref{prop:GLS}); this includes \cref{eg:introA2} (type $A_2$ with non-minimal symmetriser).

The study of a larger family of $H(C,D,\Omega)$ is carried out in a companion paper \cite{HP_rigid-locfree}, where we also apply such $\cC$-module structures to show that rigid modules are $\tau$-locally free, answering the question of Geiss--Leclerc--Schr\"oer in \cite{GLS_rigid-schur}.

\subsection*{Notation and convention}
We always fix $\Bbbk$ to be an algebraically closed field (not necessarily of characteristic zero).
In this paper, all categories will be additive $\Bbbk$-linear, and accordingly all functors (bifunctors etc.) will be additive $\Bbbk$-linear.
In an additive category, we use $\rad(-,-)$ to denote its radical ideal (of morphisms): 
\[
\rad(X,Y) = \{f:X \to Y \mid 1_X-gf \text{ is invertible for all } g:Y\to X \}.
\]
Unless explicitly stated otherwise, all abelian categories considered will have enough projectives.
For an abelian category $\cA$, we use $\proj(\cA)$ and $\inj(\cA)$ to denote its full subcategory of projective objects and injective objects respectively.
We use $\vec$ (resp.\ $\Vec$) to denote the category of finite dimensional (resp.\ all) vector spaces.

We say that an abelian category $\cA$ is \emph{finite} if it is Hom-finite (i.e.\ $\Hom$ spaces are finite-dimensional), finite-length, and has only finitely many simple objects up to isomorphisms (and has enough projectives as part of our blanket assumption).
Note that $\cA$ is finite if and only if there exists a finite-dimensional algebra $A$ such that $\cA \cong \mod_A$ -- it follows that $\cA^{\op}$ is also finite, hence $\cA$ also has enough injectives. Note that many of our results apply to abelian categories that are not necessarily finite (for example most of \cref{sec:homological-alg}).

When working with monoidal categories and their module categories, we will suppress associators in our proofs by appealing to Mac Lane's coherence theorem; see e.g.\ \cite[Theorem 2.8.5 and Remark 7.2.4]{EGNO}.

\subsection*{Outline}
\cref{sec:prelim} recalls some of the classical notions from abelian categories that we need.
\cref{sec:tensor-mod-PM} contains the definitions of finite tensor categories and their module categories, followed by the introduction of the main character of this paper: $\cP\cM$.
We also prove some preliminary results on $\cP\cM$, e.g.\ being functorially finite (\cref{lem:functoriallyfinite}).

In \cref{sec:homological-alg} we study homological algebraic notions relative to $\cC$-projectives (\cref{defn:Cproj}), where the relations to Gorenstein projectives are in \cref{sec:Gorenstein-proj}.
The main results \cref{thm:intro1}\ref{item:Cproj=Gproj}\&\ref{item:pd-from-PM} and \cref{thm:intro-PM-findim} are proven in this section.
We also recall the notion of internal Homs and Exts in \cref{sec:internal-hom-ext}.

In \cref{sec:functor-cat} we take on the functor category point of view and study (finitely presented) $\cC$-module functors. In \cref{sec:enriched-Yoneda}, we show that $\cP\cM$ is identified with the $\cC$-module functors that are projective under the enriched Yoneda embedding  using internal Homs (\cref{lem:projFunfp}), which allows us to describe the $\cC$-module functor category using $\cP\cM$ (\cref{thm:linearisingCfunctor}).
We also show that $\cP\cM$ is closed under Auslander--Reiten translates in \cref{sec:PM-AR-translate}, and conclude that $\cP\cM$ is a precluster-tilting subcategory (\cref{cor:PM-precluster-tilting}).
The $\cC$-equivariant Auslander correspondence (and its relation to the Auslander--Solberg correspondence) is shown in \cref{sec:equiv-Aus-cor}.

The last two sections contain our main examples. \cref{sec:example-tensor-alg} studies tensor products of representation-finite algebras with self-injective Nakayama algebras, whereas \cref{sec:species} studies $\cC$-species and 1-Iwanaga--Gorenstein algebras of Geiss--Leclerc--Schr\"oer.

\subsection*{Acknowledgement}
This work is a result of many wonderful suggestions from the people we have interacted with over the years. EH would like to thank all those who told him about the work of Geiss--Leclerc--Schr\"oer \cite{GLS_quiverI} during his talks on a generalisation of Gabriel's theorem \cite{Heng_Coxquiver}. The attempts to understand these algebras have ultimately led to this work. We would like to thank (in no particular order): Bernhard Keller for pointing out the categorical origin of the Auslander--Reiten translation; Matt Pressland for all of the suggested references on Gorenstein algebras; Sondre Kvamme and Julian K\"ulshammer for discussions on pro-species; Ryan Kinser for early discussions on tensor algebras in finite tensor categories; and Matti Stroi\'nski for numerous discussions on related topics. Finally, we are grateful to Cl\'ement Delcamp for the physical motivation behind viewing finite tensor categories as ``generalised symmetries''. The research of KC was supported by ARC grants FT220100125 and DP250100762, and EH was supported by DP230100654.

\subsection*{Statement on AI usage} 
We declare that no generative artificial intelligence were used in the mathematical development of this work, including the formulation of results and the construction of proofs. Generative AI was used only after a complete draft had been prepared, as an additional check of the completed work.

\section{Preliminaries}\label{sec:prelim}
Throughout this section, $\cA$ denotes an abelian category.

\subsection{Resolving subcategories and dimensions}
Let us first recall the notion of resolving subcategories; cf.\ \cite[\S1.1]{Holm_Gorenstein}, \cite{Auslander-Bridger_stable} or \cite{AB_CMapprox}.
\begin{definition}\label{defn:resolvingsubcat}
    Let $\cX \subseteq \cA$ be a full subcategory of $\cA$.
    We call $\cX$ a \emph{resolving subcategory} if the following properties hold:
    \begin{enumerate}
        \item \label{item:extclosed} $\cX$ is closed under extensions;
        \item \label{item:kerclose} $\cX$ is closed under taking kernels of epimorphisms; 
        \item \label{item:summandclosed} $\cX$ is closed under taking direct summands.
    \end{enumerate}
    It is a \emph{projectively resolving subcategory} if it moreover contains $\proj(\cA)$.
    For a resolving subcategory $\cX$, an \emph{$\cX$-resolution} of $M\in\cA$ is an exact sequence of the form (possibly unbounded on the left):
    \[
    \cdots \to E_n \to E_{n-1} \to \cdots \to E_1 \to E_0\to M \to 0,
    \] with $E_i\in\cX$ for all $i$.
    The length of such a resolution is the minimum $n$ such that $E_{j} = 0$ for all $j \geq n+1$; otherwise the length is infinite. 
    The \emph{$\cX$-dimension} of $M$, denoted by $\Xdim{\cX}(M) \in \mathbb{N \cup \{\infty\}}$, is defined as the length that is minimal over all of its $\cX$-resolutions.
    
    For any (full) subcategory $\cB\subseteq \cA$, we define the \emph{$\cX$-dimension of $\cB$} as 
    \[
    \Xdim{\cX}(\cB)\coloneqq \sup\{\Xdim{\cX}(M) \mid M \in \cB \},
    \]
    and the \emph{finitistic $\cX$-dimension of $\cB$} as
    \[
    \fin.\Xdim{\cX}(\cB) \coloneqq \sup\{\Xdim{\cX}(M) \mid M \in \cB: \Xdim{\cX}(M)<\infty \}.
    \]
\end{definition}
The category of projectives $\proj(\cA)$ is the prototypical example of a projectively resolving subcategory, and the corresponding notion of resolution and resolution dimension are the usual ones.
Throughout this paper, we will simply denote 
\begin{equation}\label{eqn:projdimshorthand}
    \begin{split}
        \pd(M) &\coloneqq \Xdim{\proj(\cA)}(M); \\
        \pd(\cB) &\coloneqq \Xdim{\proj(\cA)}(\cB); \\
        \fin.\pd(\cB) &\coloneqq \fin.\Xdim{\proj(\cA)}(\cB).
    \end{split}
\end{equation}

\subsection{Approximations, minimality and almost split sequences}
\begin{definition}\label{defn:approx}
    Let $\cX$ be a full subcategory of $\cA$ and let $M \in \cA$.
    A morphism $g: E \to M$ in $\cA$ with $E \in \cX$ is called a \emph{right $\cX$-approximation of $M$} if for all $E' \in \cX$, the induced map $\Hom_\cA(E',g):\Hom_\cA(E',E) \to \Hom_\cA(E',M)$ is surjective.
    Dually, a morphism $f: M \to E$ in $\cA$ with $E \in \cX$ is called a \emph{left $\cX$-approximation of $M$} if for all $E' \in \cX$, the induced map $\Hom_\cA(f,E'):\Hom_\cA(E,E') \to \Hom_\cA(M,E')$ is surjective. 
\end{definition}

\begin{definition}
    Let $\cX$ be a full subcategory of $\cA$.
    We say that $\cX$ is \emph{right (resp.\ left) functorially finite} if every object $M\in\cA$ has a right (resp.\ left) $\cX$-approximation.
    We say that $\cX$ is \emph{functorially finite} if $\cX$ is both right and left functorially finite.
\end{definition}

\begin{definition}
    Let $f: M \to N$ be a morphism in $\cA$.
    We say that $f$ is \emph{right minimal} if all $\varphi: M \to M$ such that $f \circ \varphi = f$ are automorphisms.
    Dually, we say that $f$ is \emph{left minimal} if all $\psi:N\to N$ such that $\psi\circ f = f$ are automorphisms.
\end{definition}

\begin{definition}
    Let $f: M \to N$ be a morphism in $\cA$.
    We say $f$ is \emph{right almost split} if $f$ is not a split epimorphism and for every $f':M' \to N$ that is not a split epimorphism, there exists $g:M' \to M$ such that $f\circ g = f'$.
    Dually, we say $f$ is \emph{left almost split} if $f$ is not a split monomorphism and for every $f':M \to N'$ that is not a split monomorphism, there exists $g:N \to N'$ such that $g\circ f = f'$.
\end{definition}

\begin{definition}
    A short exact sequence $0 \to M \xra{f} E \xra{g} N \to 0$ in $\cA$ is called an \emph{almost split sequence} if $f$ is left minimal and left almost split, and $g$ is right minimal and right almost split.
\end{definition}
Note that being in an almost split sequence $0\to M \to E \to N \to 0$ implies that $M$ and $N$ are both indecomposable objects in $\cA$.
If $\cA$ is a finite abelian category, then almost split sequences exist: any non-projective indecomposable $N$ fits into an almost split sequence $0 \to M \to E \to N \to 0$; similarly for any indecomposable non-injective $M$.
See e.g.\ \cite[Chapter IV.3]{ASS_elerep}.
Moreover, for an exact sequence $0\to M \xrightarrow{f} E \xrightarrow{g} N \to 0$ in a finite abelian category to be almost split, it is equivalent to having $f$ left almost split and $N$ indecomposable; or dually having $g$ right almost split and $M$ indecomposable.
We refer the reader to \cite[Theorem 1.13]{ASS_elerep} for other characterisations of almost split sequences.

\section{Tensor categories, module categories and \texorpdfstring{$\cP\cM$}{PM}} \label{sec:tensor-mod-PM}
\subsection{Tensor categories and their module categories}
\begin{definition}[\protect{\cite[Definition 4.1.1 \& Chapter 6]{EGNO}}] \label{defn:FTC}
    A \emph{finite tensor category} $\cC$ is a finite abelian category $\cC$ equipped with a monoidal structure $-\otimes -: \cC \times \cC \to \cC$ such that the monoidal unit $\unit$ is simple and every object $X \in \cC$ has a left dual $\leftdual{X}$ and right dual $\rightdual{X}$.
    %A \emph{finite tensor category} is a tensor category whose underlying abelian category is finite.
\end{definition}
We note that in a finite tensor category $\cC$, both left and right duals of a projective object remain projective, hence $\proj(\cC) = \inj(\cC)$ \cite[Proposition 6.1.3]{EGNO}.
In other words, $\cC$ is a quasi-Frobenius category, i.e.\ $\cC$ is equivalent to the category of modules over a finite dimensional self-injective algebra.
Moreover, $X_1\otimes X_2 \in \cC$ is projective(-injective) if at least one of the $X_i$ is projective. 
%Indeed, $X_1\otimes-$ has exact left and right adjoints $\leftdual{X_1}\otimes -$ and $\rightdual{X_1}\otimes -$, hence preserves projectives(-injectives); similarly for $-\otimes X_2$.
As such, $-\otimes P$ and $P\otimes -$ for any $P\in \proj(\cC)$ sends exact sequences in $\cC$ to split exact sequences.

A large class of finite tensor categories can be obtained via $\mod_H$ for $H$ a finite dimensional Hopf algebra, where for $X, Y \in \mod_H$, $X \otimes Y$ is again a $H$-module via 
\[
h \cdot (x\otimes y) \coloneqq \sum h_{(1)}x \otimes h_{(2)}y,
\]
where $\Delta(h)=\sum h_{(1)}\otimes h_{(2)}$ and $\Delta$ is the comultiplication of $H$.
The Hopf algebra below will be used in \cref{sec:species}.
\begin{example}[Taft algebras]\label{eg:Taft}
    Fix $n \geq 1$ an integer and suppose $\zeta \in \Bbbk$ is a primitive $n$-th root of unity.
    Define $T_n(\zeta) \coloneqq \Bbbk\<g,x\>/\<g^n = 1, x^n=0, gx=\zeta xg\>$ to be the Hopf algebra with comultiplication and counit defined as follows:
    \[
    \Delta(g)=g\otimes g, \quad \Delta(x)=x\otimes 1 + g\otimes x, \qquad \epsilon(g)=1, \quad \epsilon(x)=0.
    \]
\end{example}

The following definition is a slight generalisation of \cite[Definition 7.3.1 \& 7.5.1]{EGNO} (we do not require local finiteness).
\begin{definition}\label{defn:modulecat}
    A \emph{left module category} $\cM$ over a finite tensor category $\cC$, or a \emph{left $\cC$-module category}, is an abelian category $\cM$ equipped with a bifunctor $-\lact - : \cC \times \cM \to \cM$ that is exact in both variables\footnote{Exactness in the second variable is implied by existence of (tensor) duals in $\cC$; the functor $X \lact -: \cM \to \cM$ has left adjoint $\leftdual{X}\lact -$ and right adjoint $\rightdual{X}\lact -$, hence must be exact.}, satisfying the associativity condition and unitality condition; see \cite[Definition 7.1.1]{EGNO}.
    A \emph{right $\cC$-module category} is defined similarly, with action bifunctor denoted by $- \ract -$.
    Following \cite[Definition 7.5.1]{EGNO}, a $\cC$-module category $\cM$ is said to be \emph{exact} if $P \lact M \in \proj(\cM)$ for all $P \in \proj(\cC)$.

    A \emph{$\cC$-module functor} is a functor $F: \cM_1 \to \cM_2$ between $\cC$-module categories equipped with a family of natural isomorphisms:
    \[
    s_{X,M}:F(X \lact M) \xrightarrow{\cong} X\lact F(M), \quad X \in \cC, M \in \cM_1,
    \]
    satisfying associativity (pentagon) and unital constraints; see \cite[Definition 7.2.1]{EGNO}.
\end{definition}
\begin{remark}
    This definition of a $\cC$-module category is equivalent to the one given in the introduction; see \cite[Proposition 7.3.3]{EGNO}. 
\end{remark}

Note that $\cC$, viewed as a (left or right) $\cC$-module category, is always exact, and that any exact $\cC$-module category is quasi-Frobenius (\cite[Remark 7.6.5]{EGNO}).
We note that being exact is a \emph{property} of a $\cC$-module category (and is not to be confused with exact \emph{structures}).

We also mention here a few salient properties of the functors $X\lact -$ (same can be said for $- \ract X$):
\begin{enumerate}
    \item $X\lact -$ has exact left and right adjoints given by $\leftdual{X}\lact -$ and $\rightdual{X}\lact -$ respectively. Hence $X\lact -$ preserves $\proj(\cM)$ and $\inj(\cM)$.
    \item For $X \neq 0 \in \cC$, $X\lact -$ reflects isomorphisms, as follows from the following commutative diagram induced by the epic counit $\leftdual{X} \otimes X \twoheadrightarrow \unit$ of the left dual and the monic unit $\unit \hookrightarrow \rightdual{X}\otimes X$ of the right dual:
    \[
    \begin{tikzcd}
        \leftdual{X} \otimes X \lact M \ar[r,two heads] \ar[d,"1\otimes 1 \lact f"] & M \ar[r, hook] \ar[d,"f"] & \rightdual{X} \otimes X \lact M \ar[d,"1\otimes 1\lact f"]\\
        \leftdual{X} \otimes X \lact N \ar[r, two heads] & N \ar[r, hook] & \rightdual{X} \otimes X \lact N\\
    \end{tikzcd}
    \]
\end{enumerate}

For $H$ a finite dimensional Hopf algebra, $\cM=\vec$ is always an exact module category over $\mod_H$ via the forgetful functor.
More generally, finite module categories over $\mod_H$ can be obtained as $\mod_A$ for a finite dimensional $H$-comodule algebra $A$ (i.e.\ $A$ is an algebra and the $H$-comodule structure $\lambda:A \to H \otimes A$ of $A$ is an algebra map).
Indeed, given $M \in \mod_A$ and $X \in \mod_H$, $X\otimes M$ is again an $A$-module via
\[
a\cdot (x\otimes m) \coloneqq \sum a_{(0)}\cdot x \otimes a_{(1)}\cdot m,
\]
where $\lambda(a) = \sum a_{(0)}\otimes a_{(1)}$.
The following provide examples of exact module categories related to the Taft algebras.
\begin{example}\label{eg:exactmod-Taft}
    Let $T_n(\zeta)$ be the Taft algebra from \cref{eg:Taft} and consider the algebras
    \[
    B_n(d) \coloneqq \Bbbk\<h,\varepsilon\>/\<h^d=1, \varepsilon^n=0, h\varepsilon = \zeta^m\varepsilon h \>.
    \]
    for $n=dm$.
    Note that $B_n(n)\cong T_n(\zeta)$ by definition and $B_n(d)$ are all self-injective Nakayama algebras: each $B_n(d)$ is isomorphic to $\Bbbk \wt{A}^c_{d-1}/R^n$ with $\wt{A}^c_{d-1}$ the cyclic quiver with $d$ vertices and $R$ the arrow ideal.
    We can equip $B_n(d)$ with the structure of a $T_n(\zeta)$-comodule algebra via
    \[
    \lambda(h)= g^m\otimes h, \quad \lambda(\varepsilon) = x\otimes 1 + g\otimes\varepsilon.
    \]
    Under this structure, $\mod_{B_n(d)}$ is an exact module category over $\mod_{T_n(\zeta)}$ -- when $\mathrm{char}(\Bbbk)=0$, these are all of those which are non-semisimple (and indecomposable), and we refer the reader to \cite[\S8.1]{Mombelli_uqsl2} and \cite[Theorem 4.10]{EO_FTC} for the full classification of exact finite module categories over $\mod_{T_n(\zeta)}$.
\end{example}
More examples of module categories will appear in \cref{sec:example-tensor-alg} and \ref{sec:species}.

\begin{definition}[\protect{\cite[Remark 7.1.5]{EGNO}}] \label{defn:inducedrightmod}
    Let $\cM$ be a left $\cC$-module category.
    We define $\rightop{\cM}$ to be the \emph{right} $\cC$-module category whose underlying category is the opposite category $\cM^{\op}$ of $\cM$, with $\cC$-action bifunctor $- \ract -: \cM^{\op} \times \cC \to \cM^{\op}$ defined by
    \[
    M \ract X \coloneqq \rightdual{X} \lact M.
    \]
\end{definition}

\subsection{The subcategory \texorpdfstring{$\cP\cM$}{PM} and some preliminary results}
\begin{definition}\label{defn:PM}
    Let $\cM$ be a $\cC$-module category. We define $\cP\cM$ to be the full subcategory of $\cM$ consisting of direct summands of objects $P\lact M$ over all $P\in\proj(\cC)$ and $M\in\cM$.
    We say that $\cM$ is \emph{$\cC$-rep-finite} if $\cP\cM$ has an additive generator, i.e.\ there exists $M \in \cM$ such that $\cP\cM=\add(M)$. When $\cM$ is finite as an abelian category, this is equivalent to $\cP\cM$ having only finitely many indecomposable objects up to isomorphism. 
\end{definition}

\begin{remark}\label{rmk:PMforexact}
    Note that $\proj(\cM) \subseteq \cP\cM$, and we have $\cP\cM=\proj(\cM)$ if and only if $\cM$ is exact as a $\cC$-module category (cf.\ \cref{defn:modulecat}).
    However, $\cP\cM$ need not be resolving in general. 
    This can be seen from $\cP\cM$ in \cref{fig:introA2}, which is not closed under taking extensions.
\end{remark}

\begin{lemma}\label{lem:functoriallyfinite}
    Let $K \in \cM$ and $\pi: Q \twoheadrightarrow \unit$ and $\iota: \unit \hookrightarrow \leftdual{Q}$ be the projective cover and injective envelope of the monoidal unit $\unit \in \cC$ respectively.
    Then 
    \[
    \pi\lact K : Q\lact K \twoheadrightarrow K; \quad \text{and} \quad \iota \lact K: K \hookrightarrow \leftdual{Q} \lact K
    \]
    are right and left $\cP\cM$-approximations of $K$ respectively (not necessarily minimal).
    As such, the subcategory $\cP\cM \subseteq \cM$ is functorially finite.
\end{lemma}
\begin{proof}
    We will only show that $\pi \lact K: Q\lact K \twoheadrightarrow K$ is a right $\cP\cM$-approximation of $K$; the left version follows from a dual argument.
    Since every object $X'$ in $\cP\cM$ is a summand of an object of the form $Q \lact X$, we shall take $X' = Q\lact X$ and consider morphisms $g: Q\lact X \to K$ without loss of generality.
    We shall show that there exists $h: Q\lact X \to Q\lact K$ such that
    \[
    (\pi \lact K) \circ h = g
    \]
    ($h$ need not be of the form $Q \lact h'$ for some $h': X \to K$).
    %This is equivalent to showing that the map \[ (Q\lact \pi)\circ - :\Hom_\cM(Q\lact X, Q\lact K) \to \Hom_\cM(Q\lact X, K) \] is surjective.
    
    Consider the right dual $\rightdual{Q}$ of $Q$, so that we have unit and counit morphisms: 
    \[
    \cup: \unit \to \rightdual{Q}\otimes Q; \quad \cap: Q\otimes \rightdual{Q} \to \unit
    \]
    satisfying the duality conditions.
    Applying $\rightdual{Q}\otimes-$ to $\pi$, we see that $\rightdual{Q}\otimes \pi: \rightdual{Q} \otimes Q \to \rightdual{Q}$ is an epimorphism onto a projective in $\cC$, hence has a splitting morphism $s: \rightdual{Q} \to \rightdual{Q}\otimes Q$.
    We define $h: Q \lact X \to Q \lact K$ to be the morphism given by the following composition:
    \[
    h \coloneqq (\cap \lact (Q\lact K))\circ ((Q\otimes s)\lact K)\circ ((Q\otimes \rightdual{Q})\lact g) \circ (Q\lact (\cup \lact X)).
    \]
    A direct calculation shows that $(\pi \lact K)\circ h = g$; a calculation using string diagrams is provided below (for $f\circ g$, $f$ is stacked on top of $g$).
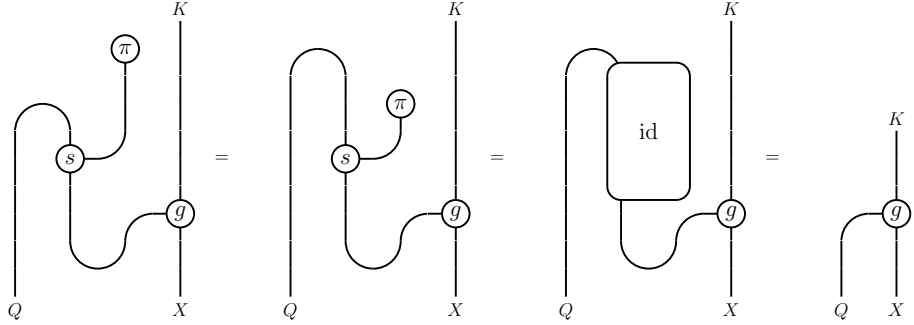
    \begin{center}
\resizebox{0.9\textwidth}{!}{%
\begin{tangle}{(17,5)}
	\tgBorderA{(2,0)}{white}{white}{white}{white}
	\tgBorder{(2,0)}{0}{0}{1}{0}
	\tgBorderA{(3,0)}{white}{white}{white}{white}
	\tgBorder{(3,0)}{1}{0}{1}{0}
	\tgBorderC{(5,0)}{3}{white}{white}
	\tgBorderC{(6,0)}{2}{white}{white}
	\tgBlank{(7,0)}{white}
	\tgBorderA{(8,0)}{white}{white}{white}{white}
	\tgBorder{(8,0)}{1}{0}{1}{0}
	\tgBorderC{(10,0)}{3}{white}{white}
	\tgBorderC{(11,0)}{2}{white}{white}
	\tgBlank{(12,0)}{white}
	\tgBorderA{(13,0)}{white}{white}{white}{white}
	\tgBorder{(13,0)}{1}{0}{1}{0}
	\tgBorderC{(0,1)}{3}{white}{white}
	\tgBorderC{(1,1)}{2}{white}{white}
	\tgBorderA{(2,1)}{white}{white}{white}{white}
	\tgBorder{(2,1)}{1}{0}{1}{0}
	\tgBorderA{(3,1)}{white}{white}{white}{white}
	\tgBorder{(3,1)}{1}{0}{1}{0}
	\tgBorderA{(5,1)}{white}{white}{white}{white}
	\tgBorder{(5,1)}{1}{0}{1}{0}
	\tgBorderA{(6,1)}{white}{white}{white}{white}
	\tgBorder{(6,1)}{1}{0}{1}{0}
	\tgBorderA{(7,1)}{white}{white}{white}{white}
	\tgBorder{(7,1)}{0}{0}{1}{0}
	\tgBorderA{(8,1)}{white}{white}{white}{white}
	\tgBorder{(8,1)}{1}{0}{1}{0}
	\tgBorderA{(10,1)}{white}{white}{white}{white}
	\tgBorder{(10,1)}{1}{0}{1}{0}
	\tgBorderA{(11,1)}{white}{white}{white}{white}
	\tgBorder{(11,1)}{1}{0}{1}{0}
	\tgBorderA{(12,1)}{white}{white}{white}{white}
	\tgBorder{(12,1)}{0}{0}{1}{0}
	\tgBorderA{(13,1)}{white}{white}{white}{white}
	\tgBorder{(13,1)}{1}{0}{1}{0}
	\tgBorderA{(0,2)}{white}{white}{white}{white}
	\tgBorder{(0,2)}{1}{0}{1}{0}
	\tgBorderA{(1,2)}{white}{white}{white}{white}
	\tgBorder{(1,2)}{1}{1}{1}{0}
	\tgBorderC{(2,2)}{1}{white}{white}
	\tgBorderA{(3,2)}{white}{white}{white}{white}
	\tgBorder{(3,2)}{1}{0}{1}{0}
	\tgBorderA{(5,2)}{white}{white}{white}{white}
	\tgBorder{(5,2)}{1}{0}{1}{0}
	\tgBorderA{(6,2)}{white}{white}{white}{white}
	\tgBorder{(6,2)}{1}{1}{1}{0}
	\tgBorderC{(7,2)}{1}{white}{white}
	\tgBorderA{(8,2)}{white}{white}{white}{white}
	\tgBorder{(8,2)}{1}{0}{1}{0}
	\tgBorderA{(10,2)}{white}{white}{white}{white}
	\tgBorder{(10,2)}{1}{0}{1}{0}
	\tgBorderA{(11,2)}{white}{white}{white}{white}
	\tgBorder{(11,2)}{1}{1}{1}{0}
	\tgBorderC{(12,2)}{1}{white}{white}
	\tgBorderA{(13,2)}{white}{white}{white}{white}
	\tgBorder{(13,2)}{1}{0}{1}{0}
	\tgBorderA{(16,2)}{white}{white}{white}{white}
	\tgBorder{(16,2)}{1}{0}{1}{0}
	\tgBorderA{(0,3)}{white}{white}{white}{white}
	\tgBorder{(0,3)}{1}{0}{1}{0}
	\tgBorderA{(1,3)}{white}{white}{white}{white}
	\tgBorder{(1,3)}{1}{0}{1}{0}
	\tgBorderC{(2,3)}{3}{white}{white}
	\tgBorderA{(3,3)}{white}{white}{white}{white}
	\tgBorder{(3,3)}{1}{0}{1}{1}
	\tgBorderA{(5,3)}{white}{white}{white}{white}
	\tgBorder{(5,3)}{1}{0}{1}{0}
	\tgBorderA{(6,3)}{white}{white}{white}{white}
	\tgBorder{(6,3)}{1}{0}{1}{0}
	\tgBorderC{(7,3)}{3}{white}{white}
	\tgBorderA{(8,3)}{white}{white}{white}{white}
	\tgBorder{(8,3)}{1}{0}{1}{1}
	\tgBorderA{(10,3)}{white}{white}{white}{white}
	\tgBorder{(10,3)}{1}{0}{1}{0}
	\tgBorderA{(11,3)}{white}{white}{white}{white}
	\tgBorder{(11,3)}{1}{0}{1}{0}
	\tgBorderC{(12,3)}{3}{white}{white}
	\tgBorderA{(13,3)}{white}{white}{white}{white}
	\tgBorder{(13,3)}{1}{0}{1}{1}
	\tgBorderC{(15,3)}{3}{white}{white}
	\tgBorderA{(16,3)}{white}{white}{white}{white}
	\tgBorder{(16,3)}{1}{0}{1}{1}
	\tgBorderA{(0,4)}{white}{white}{white}{white}
	\tgBorder{(0,4)}{1}{0}{1}{0}
	\tgBorderC{(1,4)}{0}{white}{white}
	\tgBorderC{(2,4)}{1}{white}{white}
	\tgBorderA{(3,4)}{white}{white}{white}{white}
	\tgBorder{(3,4)}{1}{0}{1}{0}
	\tgBorderA{(5,4)}{white}{white}{white}{white}
	\tgBorder{(5,4)}{1}{0}{1}{0}
	\tgBorderC{(6,4)}{0}{white}{white}
	\tgBorderC{(7,4)}{1}{white}{white}
	\tgBorderA{(8,4)}{white}{white}{white}{white}
	\tgBorder{(8,4)}{1}{0}{1}{0}
	\tgBorderA{(10,4)}{white}{white}{white}{white}
	\tgBorder{(10,4)}{1}{0}{1}{0}
	\tgBorderC{(11,4)}{0}{white}{white}
	\tgBorderC{(12,4)}{1}{white}{white}
	\tgBorderA{(13,4)}{white}{white}{white}{white}
	\tgBorder{(13,4)}{1}{0}{1}{0}
	\tgBorderA{(15,4)}{white}{white}{white}{white}
	\tgBorder{(15,4)}{1}{0}{1}{0}
	\tgBorderA{(16,4)}{white}{white}{white}{white}
	\tgBorder{(16,4)}{1}{0}{1}{0}
	\tgCell{(3,3)}{g}
	\tgCell{(1,2)}{s}
	\tgCell{(2,0)}{\pi}
	\tgCell{(6,2)}{s}
	\tgCell{(7,1)}{\pi}
	\tgCell{(8,3)}{g}
	\tgCell[(1,2)]{(11.5,1.5)}{\id}
	\tgCell{(13,3)}{g}
	\tgCell{(16,3)}{g}
	\tgAxisLabel{(3.5,0)}{south}{K}
	\tgAxisLabel{(8.5,0)}{south}{K}
	\tgAxisLabel{(13.5,0)}{south}{K}
	\tgAxisLabel{(16.5,2)}{south}{K}
	\tgAxisLabel{(4,2.5)}{west}{=}
	\tgAxisLabel{(9,2.5)}{west}{=}
	\tgAxisLabel{(14,2.5)}{west}{=}
	\tgAxisLabel{(0.5,5)}{north}{Q}
	\tgAxisLabel{(3.5,5)}{north}{X}
	\tgAxisLabel{(5.5,5)}{north}{Q}
	\tgAxisLabel{(8.5,5)}{north}{X}
	\tgAxisLabel{(10.5,5)}{north}{Q}
	\tgAxisLabel{(13.5,5)}{north}{X}
	\tgAxisLabel{(15.5,5)}{north}{Q}
	\tgAxisLabel{(16.5,5)}{north}{X}
\end{tangle}
}
    \end{center}
\end{proof}

\begin{corollary}\label{cor:weakcokernels}
    The additive category $\cP\cM$ contains weak kernels and weak cokernels.
\end{corollary}
\begin{proof}
    By \cref{lem:functoriallyfinite}, $\cP\cM$ is functorially finite, so that every object $M \in \cM$ has both right and left $\cP\cM$-approximation.
    The lemma then follows from the general fact that a right functorially finite subcategory of an abelian category contains weak kernels (namely right approximations of the actual kernels taken in $\cM$), and dually, that a left functorially finite subcategory contains weak cokernels.
    \end{proof}
    %We spell this out for the case of weak kernels, where the existence of weak cokernels follows from a dual argument.
    %Let $f: M \to N$ be a morphism between $M,N \in \cP\cM$ and take the kernel $K \coloneqq \ker(f) \xrightarrow{\iota} M$ in $\cM$.
    %We claim that the right $\cP\cM$-approximation $\pi_K: \tilde{K} \to K$ is then a weak kernel of $f$ in $\cP\cM$.
    %Indeed, let $X \in \cP\cM$ and $g: X \to M$ such that $fg=0$. 
    %Now $fg=0$ implies that $g$ factors through $\iota$ through some $g': X \to K$, so that $\iota \circ g' = g$.
    %Since $\pi_K:\tilde{K} \to K$ is a right $\cP\cM$-approximation, there exists $h: X \to \tilde{K}$ such that $\pi_K \circ h = g'$, which shows that $\tilde{K}$ is a weak kernel of $f$ as desired.
    %The fact that $\Fun^{\fp}(\cP\cM,\vec)$ is abelian follows from the general fact that $\Fun^{\fp}(\cA,\vec)$ for a ($\Bbbk$-linear) additive category $\cA$ is an abelian category if and only if $\cA$ contains weak kernels; see e.g.\ \cite[Lemma 2.1.6]{Krause_HomThRep}.  It follows immediately from the definition that $\Fun^{\fp}(\cP\cM,\vec)$ is Hom-finite and has enough projectives.

\begin{lemma}\label{lem:PPM}
    Suppose further that $\cM$ is a finite $\cC$-module category. For any $M\in \cM$ and $P\in\proj\cC$, we can write $P\lact M=M_1\oplus M_2$ where $M_1$ is projective and $M_2$ has no projective summands. Then, for every $Q\in\proj\cC$, the object $Q\lact M_2$ has no projective summands.
\end{lemma}

\begin{proof}
    We consider the Frobenius--Perron dimension $\FPdim$ from \cite[\S 3 and \S 4]{EGNO}, which gives a ring homomorphism
    \[\FPdim:K_0[\cC]\to\mR.\]
    The Frobenius-Perron dimension of an object equals that of its duals.
    Moreover, if we label the simple objects in $\cC$ as $L_i,i\in I$, with projective cover $P_i$, then 
    \[d_i:=\FPdim(L_i)\;>\;0.\]
    If we denote the multiplicity of $P_l$ in $P_i\otimes P_j$ by $c_{ij}^l$, then one verifies quickly that
    \begin{equation}\label{eq:standid}\sum_{j\in I} c_{ij}^ld_j\;=\;d_l\FPdim P_i.\end{equation}

We label the simple objects in $\cM$ by $V_j,j\in J$, and their projective cover by $Q_j$.
    Now, using \cite[Proposition~3.4.4]{EGNO}, there are $\mu_j\in\mR_{>0}$ such that, for any $X\in\cC$ in $K_0[\cM]$
    \[\sum_{j\in J}\mu_j[X\lact V_j]\;=\;\FPdim(X)\sum_{l\in J}\mu_l[V_l].\]
    We can rewrite this equation (using $\FPdim X=\FPdim X^\ast$) as
    \[\sum_j\mu_j(X\lact Q_l:Q_j)\;=\;\mu_l \FPdim X,\]
    where $(Q:Q_j)$ stands for the multiplicity of $Q_j$ as a direct summand in a projective $Q$. For a general $Q=\oplus Q_s^{\oplus b_s}\in\proj\cM$, we can thus define
    \[\FPdim(Q):=\sum_s b_s\mu_s,\]
    and the above implies
    \[\FPdim(X\lact Q)\;=\; \FPdim(X)\cdot\FPdim(Q)\]
    for all $X\in\cC$ and $Q\in\proj\cM$.

    Now, we fix $M\in\cM$ and for every $i\in I$, we define $A_i$ to be $\FPdim$ of the projective part of $P_i\lact M$. Then by acting with $P_i\otimes P_j$ on $M$ we find
    \[A_j\FPdim(P_i)\;\le\; \sum_l c_{ij}^lA_l,\]
    and the statement in the lemma is precisely the claim that this is actually an equality. We can establish this, by adding up the inequalities for every $j\in I$, weighted with $d_j>0$, which gives
    \[\FPdim P_i\sum_j d_j A_j\le \sum_j d_j\sum_l c_{ij}^l A_l=\sum_ld_lA_l\FPdim P_i ,\]
    where the inequality is the weighted sum of inequalities from above and the equality comes from \eqref{eq:standid}. Consequently, the summed inequality (and hence every individual inequality) was indeed an equality.
\end{proof}

The following is an immediate corollary; see \cref{defn:Cproj} for the definition of $\cC$-projective.
\begin{corollary}
    Let $\cM$ be a finite $\cC$-module category. The only $\cC$-projective objects in $\cP\cM$ are the projective objects, i.e.\ $\Cproj{}(\cM) \cap \cP\cM = \proj(\cM)$.
\end{corollary}

\begin{remark}
    \begin{enumerate}
        \item The proof of Lemma~\ref{lem:PPM} can be simplified by realising $\cM$ as the category of modules (in $\cC$) over an algebra $A\in\cC$, in which case we can simply use the Frobenius--Perron dimension of the underlying object in $\cC$ for objects in $\cM$.
        \item The proof of Lemma~\ref{lem:PPM} is heavily inspired by ideas from \cite{MattiIMRN}. In fact, it can be used to recover the main application of \cite{MattiIMRN} in \cite{CSZ}. Indeed, the lemma implies that if $\cM$ is not exact, then there is a non-zero $M\in\cM$ such that $P\lact M$ does not contain any non-zero projective summands, for any $P\in\proj(\cC)$. By taking a projective presentation of $M$ we thus obtain a morphism in $\proj(\cM)$ that generates a proper $\proj(\cC)$-stable ideal in $\proj(\cM)$.
    \end{enumerate}
\end{remark}

\section{Internal homological algebra}\label{sec:homological-alg}
Throughout this section, we fix $\cC$ to be a finite tensor category and $\cM$ to be a  $\cC$-module category -- all results in this section will be stated and proven in the setting of left $\cC$-module categories, but the same results (and proofs) hold for right $\cC$-module categories.

\subsection{\texorpdfstring{$\cC$}{C}-projective objects and resolving subcategories}
We recall the following notion from the theory of tensor categories:
\begin{definition}\label{defn:Cproj}
    Let $\cM$ be a  $\cC$-module category. 
    We say that $E \in \cM$ is a \emph{$\cC$-projective} if $P\lact E \in \proj(\cM)$ is projective for some non-zero projective $P \in \proj(\cC)$.
    We define $\cC$-injectives in a similar fashion.
    The full subcategory of $\cC$-projectives (resp.\ $\cC$-injectives) is denoted by $\Cproj{}(\cM)$ (resp.\ $\Cinj{}(\cM))$.
\end{definition}
Note that $\Cproj{}(\cM)\supseteq \proj(\cM)$ and $\Cinj{}(\cM)\supseteq \inj(\cM)$ since $X\lact -$ preserves projectives and injectives.
\begin{remark}
    The notion of $\cC$-projective here differs from the similarly named notion of $\cB$-projective with respect to an exact subcategory $\cB \subseteq \cM$ defined in \cite{FDatlas}.
\end{remark}

The following well-known lemma provides another equivalent condition of being $\cC$-projective.
The analogous statement for $\cC$-injectives can be obtained through a dual argument. 
\begin{lemma}\label{lem:Cproj}
    The following are equivalent on $E\in\cM$:
    \begin{enumerate}
        \item \label{item:Cproj1} $E$ is $\cC$-projective.
        \item \label{item:Cprojforall1} $Q\lact E$ is projective for all $Q \in \proj(\cC)$.
    \end{enumerate}
\end{lemma}
\begin{proof}
    The implication $\ref{item:Cprojforall1} \implies \ref{item:Cproj1}$ is immediate, so we are left to show
    $\ref{item:Cproj1} \implies \ref{item:Cprojforall1}$.
    Let $E \in \Cproj{}(\cM)$.
    Suppose $P\neq 0 \in \proj(\cC)$ is such that $P\lact E \in \proj(\cM)$, and let $Q \in \proj(\cC)$ be arbitrary.
    Since $P$ has a left dual, we get a non-zero counit map $\eta: \leftdual{P} \otimes P \twoheadrightarrow \unit$, which has to be an epimorphism since $\unit$ is simple.
    It follows that $Q \otimes \eta:Q\otimes \leftdual{P}\otimes P \to Q \otimes \unit \cong Q$ is an epimorphism, which has to split since $Q$ is projective.
    It follows that $(Q \otimes \eta)\lact E : (Q\otimes \leftdual{P}\otimes P)\lact E \to Q\lact E$ is a split epimorphism, hence $Q\lact E$ is a summand of $(Q\otimes \leftdual{P}\otimes P)\lact E \cong (Q\otimes \leftdual{P})\lact (P\lact E)$, which is projective (recall that $X\lact - $ sends projectives to projectives for any $X \in \cC$).
    This shows that $Q\lact E$ is projective, as required.
\end{proof}

\begin{lemma}\label{lem:CprojclosedunderC}
    Let $X \in \cC$ and $E \in \Cproj{}(\cM)$.
    Then $X\lact E \in \Cproj{}(\cM)$.
    In other words, $\Cproj{}(\cM)$ is closed under the action of $\cC$.
\end{lemma}
\begin{proof}
    It suffices to show that $P\lact (X\lact E) \cong (P \otimes X) \lact E \in \proj(\cM)$ for some $P\neq 0 \in \proj(\cC)$.
    But since $P\otimes X$ is a projective in $\cC$ and $E$ is a $\cC$-projective by assumption, this follows from \cref{lem:Cproj}.
\end{proof}

\begin{proposition}\label{prop:Cprojresolving}
    $\Cproj{}(\cM)$ is a projectively resolving subcategory (\cref{defn:resolvingsubcat}) of $\cM$.
    Moreover, $\proj(\cM)$ is a cogenerator of $\Cproj{}(\cM)$, i.e.\ for all $E \in \Cproj{}(\cM)$, there exists an exact sequence $0 \to E \to G \to E' \to 0$ with $G \in \proj(\cM)$ and $E' \in \Cproj{}(\cM)$.
\end{proposition}
\begin{proof}
    Let $0 \to M' \to M \to M'' \to 0$ be an exact sequence  in $\cM$.
    Properties \ref{item:extclosed} to \ref{item:summandclosed} in \cref{defn:resolvingsubcat} all follow immediately from applying the exact endofunctor $P\lact -$ to the exact sequence for some $P \neq 0 \in \proj(\cC)$.
    The fact that $\proj(\cM) \subseteq \Cproj{}(\cM)$ follows immediately from the fact that $X \lact -$ has both left and right adjoints for any $X \in \cC$.
    
    Finally, to see that $\proj(\cM)$ is a cogenerator for $\Cproj{}(\cM)$, take the unit morphism $\mu:\unit \to P\otimes \leftdual{P}$ that defines the left dual $\leftdual{P}$ of $P\neq 0 \in \proj(\cC)$ and consider the short exact sequence $(\mu$ is a monomorphism since it is non-zero and $\unit$ is simple):
    \[
    0 \to \unit \xrightarrow{\mu} P\otimes \leftdual{P} \to \coker(\mu) \to 0.
    \]
    Given any $E \in \Cproj{}(\cM)$, we obtain the desired cogeneration property by applying the exact functor $-\lact E$ to this exact sequence (noting that $P\otimes \leftdual{P} \in \proj(\cC)$, and $\coker(\mu) \lact E \in \Cproj{}(\cM)$ by \cref{lem:CprojclosedunderC}).
\end{proof}

Specialising to $\cX = \Cproj{}(\cM)$ from \cref{defn:resolvingsubcat}, we have the following:
\begin{definition}\label{defn:Cprojdim}
    Let $\cM$ be a $\cC$-module category and let $M$ be an object in $\cM$.
    A $\Cproj{}(\cM)$-resolution of $M$ is called a \emph{$\cC$-projective resolution of $M$}.
    The \emph{$\cC$-projective dimension of $M$} is defined as 
    \[
    \ipd(M) \coloneqq \Xdim{\Cproj{}(\cM)}(M).
    \]
    For a (full) subcategory $\cB \subseteq \cM$, we define the \emph{$\cC$-projective dimension of $\cB$} as
    \[
    \ipd(\cB) \coloneqq \Xdim{\Cproj{}(\cM)}(\cB),
    \]
    and the \emph{finitistic $\cC$-projective dimension of $\cB$} as
    \[
    \fin{.}\ipd(\cB) \coloneqq \fin{.}\Xdim{\Cproj{}(\cM)}(\cB).
    \]
\end{definition}
Note that by definition, an object is a $\cC$-projective if and only if its $\cC$-projective dimension is zero. 
In particular, $\ipd(\cM) = 0$ if and only if every object in $\cM$ is a $\cC$-projective.

The following lemma shows that $\cC$-projectives are also ``Ext-orthogonal'' to objects with finite projective dimension:
\begin{lemma}\label{lem:orth}
    Consider $M,N\in\cM$, where $M$ is $\cC$-projective and $N$ has finite projective dimension. Then $\Ext^i_{\cM}(M,N)=0$ for $i>0$.
\end{lemma}
\begin{proof}
    We can prove this by induction on $n=\pd(N)$. If $n=0$, then $N$ is a summand of an object $P\lact N'$ for $P\in\cC$ projective, and the claim follows by adjunction. Assume the claim holds for $n$ and consider the case $\pd(N)=n+1$, then with $K$ the kernel of an epimorphism $Q\twoheadrightarrow N$ for $Q \in \proj(\cM)$, we have an exact sequence
    \[\Ext^i(M,Q)\to \Ext^i(M,N)\to\Ext^{i+1}(M,K).\]
    The left and right term are zero by the base case and the induction hypothesis respectively, and hence so is the middle term.
\end{proof}
\begin{remark}
    The lemma above is also a consequence of the fact that all $\cC$-projectives are Gorenstein projectives, which will be shown later in \cref{prop:Cproj-Gproj}.
\end{remark}

\subsection{Computing \texorpdfstring{$\cC$}{C}-projective dimension and resolution}
As before, $\cM$ is again a $\cC$-module category.

Recall that the (usual) projective dimension of $M\in \cM$ is denoted by $\pd(M)$.
%In this setting, given an object $M$ with $\pd(M) = m$, any of its projective resolutions can be shortened to length $m$ by taking the kernel of the $(m-1)$-th differential, which has to be projective by Schanuel's lemma.
The following proposition relates projective resolutions with $\cC$-projective resolutions, and provides an analogue of Schanuel's lemma for $\cC$-projective resolutions.
\begin{proposition}\label{prop:Cprojdimprojdim}
    Let $M \in \cM$ and $P \neq 0 \in \proj(\cC)$.
    Then 
    \[
    \ipd(M) = \pd(P\lact M).
    \]
    Moreover, if $\ipd(M) < \infty$, any projective resolution of $M$ 
    \begin{equation}\label{eqn:projres}
    \cdots \to E_m \xra{d_{m}} E_{m-1} \xra{d_{m-1}} E_{m-2} \to \cdots \to E_1 \xra{d_1} E_0 \xra{d_0} M \xra{d_{-1}} 0,
    \end{equation}
    induces the following $\cC$-projective resolution of $M$
    \begin{equation}\label{eqn:projtoCprojrest}
    0 \to \ker(d_{n-1}) \to E_{n-1} \xra{d_{n-1}} E_{n-2} \to \cdots \to E_1 \xra{d_1} E_0 \xra{d_0} M \xra{d_{-1}} 0
    \end{equation}
    for any $n \geq \ipd(M)$.
    %In particular, $\ipd(\cM)=\pd(\cP\cM)$.
\end{proposition}
\begin{proof}
    The inequality $\pd(P\lact M) \leq \ipd(M)$ follows immediately from the fact that any $\cC$-projective resolution $E^\bullet \to M$ gives a projective resolution of $P\lact E^\bullet \to P\lact M$ ($P\lact -$ is an exact functor).
    Thus we are left to show the reverse inequality for $P\neq 0 \in \proj(\cC)$.

    If $\pd(P\lact M) = \infty$ then we have nothing left to show; so suppose $m \coloneqq \pd(P\lact M)< \infty$ instead.
    Given a projective resolution of $M$ as in \eqref{eqn:projres}, we apply $P\lact -$ and get (from Schanuel's lemma) the following projective resolution of $P\lact M$ for any $n \geq m$:
    \[
     0 \to \ker(P \lact d_{n-1}) \to P\lact E_{n-1} \xra{P\lact d_{n-1}}  \cdots \to P\lact E_1 \xra{P \lact d_1} P\lact E_0 \xra{P\lact d_0} P\lact M \xra{P\lact d_{-1}} 0.
    \]
    In particular, $\ker(P\lact d_{n-1})$ is projective, and by exactness of $P\lact -$, we get that $P \lact \ker(d_{n-1}) \cong \ker(P \lact d_{n-1})$ is projective, which shows that $\ker(d_{n-1})$ is $\cC$-projective.
    It follows that
    \[
    0 \to \ker(d_{n-1}) \to E_{n-1} \xra{d_{n-1}} \cdots \to E_1 \xra{d_1} E_0 \xra{d_0} M \xra{d_{-1}} 0
    \]
    is a $\cC$-projective resolution of $M$ for any $n \geq m$, giving us \eqref{eqn:projtoCprojrest} as required.
    Our desired inequality $\ipd(M) \leq m = \pd(P\lact M)$ also follows by picking $n=m$.
    %For the final equality $\ipd(\cM)=\pd(\cP\cM)$, the proven equality $\ipd(M) = \pd(P\lact M)$ for all $M \in \cM$ shows that $\ipd(\cM) \leq \pd(\cP\cM)$. On the other hand, any $M' \in \cP\cM$ is a summand of $P\lact M$ for some $P \neq 0 \in \proj(\cC)$ and some $M\in \cM$, hence $\pd(M') \leq \pd(P\lact M')=\ipd(M)$ as well, as required.
\end{proof}

\begin{remark}\label{rmk:pdPMindependent}
Note that \cref{prop:Cprojdimprojdim} implies $\pd(P\lact M) = \pd(P'\lact M)$ for any non-zero $P,P' \in \proj(\cC)$ (since $\ipd(M)$ is independent of $P$).
\end{remark}

\begin{corollary}\label{cor:pdCpdrelations}
    Let $L \neq 0 \in \cC$ and $P\neq 0 \in \proj(\cC)$.
    We have the following chain of inequalities 
    \[
    \pd(P\lact M) = \ipd(M) = \pd((P\otimes L) \lact M) \leq \pd(L\lact M) \leq \pd(M).
    \]
\end{corollary}
\begin{proof}
    The first two equalities follow from \cref{prop:Cprojdimprojdim} (see also \cref{rmk:pdPMindependent}).
    The last two inequalities follow from the fact that any projective resolution $Q^\bullet \to M$ gives a projective resolution of $X\lact Q^\bullet \to X\lact M$ for any $X \in \cC$ ($X\lact -$ is an exact functor).
\end{proof}

\begin{remark}
    Note that the inequality $\pd(L\lact M) \leq \pd(M)$ can be strict even if $L \notin \proj(\cC)$ (which can only happen if $\pd(M)=\infty$; see \cref{prop:finite_pd=finite_ipd}). For instance we can consider a finite tensor category $\cC$, as the regular module category over itself, with non-projective objects $X,Y$ with $X \otimes Y$ projective. An example is given by $\cC=\cC_4$ in \cite[\S 5.4]{BE} where the tensor product of the non-projective simple objects $V_1,V_2$ is a simple projective.
\end{remark}

%%%%%%%%%%%%%%%%%%%%%%%%%%%
%
%%%%%%%%%%%%%%%%%%%%%%%%%%%
\iffalse
\begin{corollary}[Horseshoe lemma]
    Let $0 \to M' \to M \to M'' \to 0$ be a short exact sequence in $\cM$.
    Then $\ipd(M) \leq \max \{ \ipd(M'), \ipd(M'')\}$.
\end{corollary}
\begin{proof}
    This follows immediately from \cref{prop:Cprojdimprojdim} and the usual Horseshoe lemma by applying $P\lact -$ for $P \neq 0 \in \Cproj{}(\cM)$.
\end{proof}
\fi
%%%%%%%%%%%%%%%%%%%%%%%%%%%

\subsection{\texorpdfstring{$\cC$}{C}-projectives and Gorenstein projectives} \label{sec:Gorenstein-proj}
%Readers familiar with Gorenstein homological algebra will notice that many of our results in relation to $\cC$-projectives resemble that of Gorenstein projectives. We clarify their relation in this section.
In the following, a complex $(X_\bullet, d_\bullet)$ in $\proj(\cM)$ is said to be \emph{totally acyclic} if it is exact and $\Hom_\cM(X_\bullet, P)$ is exact for any projective $P \in \proj(\cM)$.
\begin{definition}
    An object $G \in \cM$ is a \emph{Gorenstein projective} if there exists a totally acyclic complex $(P_\bullet, d_\bullet)$ in $ \proj(\cM)$ such that $\ker(P_0 \to P_{-1}) \cong G$.
    The full subcategory of Gorenstein projectives is denoted by $\Gproj(\cM)$.
\end{definition}

Note that $\Gproj(\cM)$ is also a projectively resolving subcategory; see e.g.\ \cite{Holm_Gorenstein} or \cite{AB_CMapprox}.
By specialising $\cX = \Gproj(\cM)$ from \cref{defn:resolvingsubcat}, we have the following:
\begin{definition}
    Let $M$ be an object in $\cM$.
    A $\Gproj(\cM)$-resolution of $M$ is called a \emph{Gorenstein-projective resolution of $M$}.
    The \emph{Gorenstein-projective dimension of $M$} is defined as 
    \[
    \Gpd(M) \coloneqq \Xdim{\Gproj(\cM)}(M).
    \]
    We also define the \emph{Gorenstein-projective dimension of $\cM$} as
    \[
    \Gpd(\cM) = \sup\{ \Gpd(M) \mid M \in \cM\}.
    \]
    We say that $\cM$ is \emph{$d$-Iwanaga--Gorenstein} if $\Gpd(\cM) \leq d$; for a finite abelian category $\cM$, this is equivalent to having all projectives in $\cM$ have injective dimension $\leq d$ and all injectives in $\cM$ have projective dimension $\leq d$  (see e.g.\ \cite[Theorem 2.3.3]{XW-Chen_GorensteinHom}).
\end{definition}

\begin{proposition}\label{prop:Cproj-Gproj}
    Let $\cM$ be a $\cC$-module category and let $X \in \mathcal{M}$. The following are equivalent
    \begin{enumerate}
        \item \label{item:isCproj} $X$ is a $\mathcal{C}$-projective.
        \item \label{item:isGproj-ipd} $X$ is a Gorenstein projective and $\ipd(X)<\infty$. 
    \end{enumerate}
    In particular, $\Cproj{}(\cM) \subseteq \Gproj(\cM)$, so that $\Gpd(\cM) \leq \ipd(\cM)$.
\end{proposition}
\begin{proof}
   We can prove $\ref{item:isCproj} \implies \ref{item:isGproj-ipd}$ as follows. 
   Let $(P_\bullet, d_\bullet)$ denote the totally acyclic complex with $\unit \cong \ker(P_0\xra{d_0} P_{-1})$ obtained by combining a projective resolution and projective(-injective) coresolution of $\unit$ (this is possible since $\cC$ is finite and hence has enough projectives and injectives).
   Given a $\cC$-projective $X$, we can apply the exact functor $-\lact X$ to the totally acyclic complex to get $(P_\bullet\lact X, d_\bullet \lact X)$.
   By definition, each $P_i\lact X \in \proj(\cM)$ and 
   \[
   X \cong \unit\lact X \cong \ker(P_0\lact X\xra{d_0 \lact X} P_{-1}\lact X) \cong \ker(P_0\xra{d_0} P_{-1})\lact X.
   \]
   Moreover, we can observe that $P_\bullet\lact X$ is totally acyclic. Indeed, any projective $R$ in $\cM$ is a direct summand of $Q\lact R$, where $Q\in\cC$ is the projective cover of $\unit$, so it suffices to observe that
   \[\Hom_{\cM}(P_\bullet\lact X,Q\lact R)\cong\Hom_{\cM}((Q^\ast\otimes P_\bullet)\lact X, R)\]
   is exact, since the action of $Q^\ast\otimes -$ on any exact complex in $\cC$ produces a split exact complex.
   
   For $\ref{item:isGproj-ipd} \implies \ref{item:isCproj}$, let $G$ be Gorenstein-projective and suppose that $\ipd(G)=d<\infty$.
   Then we get $\pd(P\lact G) = d$ by \cref{prop:Cprojdimprojdim} for some $P\neq 0 \in \proj(\cC)$.
   Now note that $P\lact G$ is again Gorenstein-projective -- just apply the exact functor $P\lact -$ to the totally acyclic complex associated to $G$.
   Since $P\lact G$ is a Gorenstein projective with finite projective dimension, it has to be projective; this follows from a simple induction argument combined with the fact that $\Ext_{\cM}^1(G,R)=0$ for any $R \in \proj(\cM)$ (see also \cite[Proposition 2.27]{Holm_Gorenstein}).
   %%%%%%%%%%%%%%%%%%%%%%%%%%%%%%%%%%% 
   % Argument for $\pd(G) <\infty \implies \pd(G)=0$:
   %%%%%%%%%%%%%%%%%%%%%%%%%%%%%%%%%%%
   % We induct on $\pd(G)=n$. If $n=0$ this is already what we want, so assume as induction hypothesis that any Gorenstein projective $G'$ has the property that $\pd(G')=n-1 \implies \pd(G')=0$. Let $P_0$ be the projective cover of $G$ and let $G'$ be its kernel. By assumption on $G$, $\pd(G')=n-1$. Moreover, since $G$ is Gorenstein, $G'$ is again Gorenstein, being the syzygy of a projective resolution. Induction hypothesis implies that $G'$ is projective, giving us a short exact sequence of the form $0 \to G' \to P_0 \to G \to 0$. But $\Ext^1_\cM(G,G')=0$ since $G'$ is projective (use the totally acyclic complex defining $G$ as the projective resolution of $G$ to compute $\Ext$). It follows that the short exact sequence splits, proving that $G$ is projective as required.
   It follows that $G \in \Cproj{}(\cM)$, as required.

   From here both the containment $\Cproj{}(\cM) \subseteq \Gproj(\cM)$ and the inequality $\Gpd(\cM) \leq \ipd(\cM)$ are immediate.
\end{proof}

\begin{remark}
    The assumption $\ipd(X)<\infty$ in condition \ref{item:isGproj-ipd} is necessary for \cref{prop:Cproj-Gproj} to hold.
    In other words, the containment $\Cproj{}(\cM) \subseteq \Gproj(\cM)$ can be strict. 
    Indeed, consider $\cC=\vec$ and let $\cM=\mod_A$ for some non-semisimple self-injective algebra $A$. Then $\Cproj{}(\cM) = \proj(\cM) \subsetneq \cM = \Gproj(\cM)$. 
\end{remark}

\begin{corollary}\label{cor:Cproj=Gproj}
    Let $\cM$ be a $\cC$-module category. 
    If $\ipd(\cM) = d < \infty$, then 
    \[
    \Cproj{}(\cM) = \Gproj(\cM).
    \]
    In particular, $\Gpd(\cM)=\ipd(\cM)$ and $\cM$ is $d$-Iwanaga--Gorenstein.
\end{corollary}
\begin{proof}
    This is an immediate consequence of \cref{prop:Cproj-Gproj}.
\end{proof}

\subsection{Computing finitistic dimension} \label{sec:compute-findim}

\begin{proposition}\label{prop:finite_pd=finite_ipd}
    Let $\cM$ be a $\cC$-module category and let $M\in\cM$. Then
    \[
    \pd(M) <\infty \implies \pd(P\lact M) = \ipd(M) = \pd(L \lact M) = \pd(M)
    \] 
    for any $L\neq 0 \in \cC$ and $P \neq 0\in \proj(\cC)$.
    In particular, $\cC$-projective objects are either projective or have infinite projective dimension.
\end{proposition}
\begin{proof}
    By \cref{cor:pdCpdrelations} and \cref{rmk:pdPMindependent}, it suffices to show that $\pd(M) \leq \pd(P \lact M)$ for some non-zero $P \in \proj(\cC)$.
    The case where $M$ is projective is immediate since $P \lact M$ would again be projective. 
    Assume that $\pd M=n\in\mathbb{Z}_{>0}$. Then there is some $N\in\cM$ with $\Ext^n_{\cM}(M,N)\not=0$, and by maximality of $n$ we can replace $N$ by any projective object that surjects onto it, which we can take to be of the form $P\lact N'$ for some (projective) $N'\in\cM$ and non-zero $P\in \proj(\cC)$. In conclusion, we find
    \[0\neq\Ext^n_{\cM}(M,P\lact N')=\Ext^n_{\cM}(P^\ast\lact M,N').\]
    Together with \cref{prop:Cprojdimprojdim}, this implies that $\ipd(M) = \pd (\leftdual{P}\lact M) \geq n$, concluding the proof.
\end{proof}
\begin{remark}\label{rmk:Gpd=ipd=pd}
    Note that we have a chain of inequalities $\Gpd(M) \leq \ipd(M) \leq \pd(M)$ by \cref{prop:Cproj-Gproj} and \cref{cor:pdCpdrelations}.
    Hence in the setting where $\cM \cong \mod_A$ or $\Mod_A$ (i.e.\ the category of all modules) for some $\Bbbk$-algebra $A$, the result above can also be deduced from \cite[Proposition 2.27]{Holm_Gorenstein}, which shows that if $\pd(M) < \infty$, then $\Gpd(M) = \pd(M)$ (see also \cite{Auslander-Bridger_stable}).
\end{remark}

One of the main consequences is that this allows us to compute the finitistic dimension of $\cM$ from $\cP\cM$ (\cref{defn:PM}):
\begin{theorem}\label{thm:PM-computes-findim}
    Let $\cM$ be a $\cC$-module category. Then 
    \[
    \fin.\pd(\cM) = \fin.\pd(\cP\cM).
    \]
    Hence for a finite $\cC$-module category $\cM$, being $\cC$-rep-finite implies $\fin.\pd(\cM) < \infty$.
\end{theorem}
\begin{proof}
    Since $\cP\cM \subseteq \cM$, we get $\fin.\pd(\cM) \geq \fin.\pd(\cP\cM)$ by definition.
    On the other hand, let $M \in \cM$ be such that $\pd(M)<\infty$. 
    Then by \cref{prop:finite_pd=finite_ipd} we get 
    \[
    \pd(P\lact M) = \ipd(M) = \pd(M)
    \]
    for any $P \neq 0 \in \proj(\cC)$.
    Since $P\lact M \in \cP\cM$ by definition, we get that $\fin.\pd(\cM) \leq \fin.\pd(\cP\cM)$, as required.
    The final statement follows immediately from the definition of being $\cC$-rep-finite.
\end{proof}

\begin{remark}
    Following essentially the same proof as in \cite[Theorem 2.28]{Holm_Gorenstein}, one can also show that
    \[
    \fin.\pd(\cM) =\fin.\ipd(\cM) = \fin.\Gpd(\cM).
    \]
\end{remark}

%As an immediate corollary, we get:
%\begin{corollary}\label{cor:finite-fin-dim}
%    Let $A$ be a finite dimensional algebra. If $\cM \coloneqq \mod_A$ is $\cC$-rep-finite over some finite tensor category $\cC$, then the finitistic dimension conjecture holds for $A$, i.e.\ $\fin.\pd(\cM) < \infty$.
%\end{corollary}

%%%%%%%%%%%%%%%%%%%%%%%%%%%%%%%%%%%%%%%%
% Results of Auslander--Buchweitz in the general setting
%%%%%%%%%%%%%%%%%%%%%%%%%%%%%%%%%%%%%%%%
\iffalse
\begin{proposition}\label{prop:ABapproxC-proj}
    Let $\cM$ be a $\cC$-module category.
    The following are equivalent for any $M \in \cM$:
    \begin{enumerate}
        \item $\ipd(M)< \infty$.
        \item \label{item:AB-approx-eqn}There exists a short exact sequence 
        \begin{equation}
            0 \to Y_M \to E_M \to M \to 0
        \end{equation}
        with $\pd(Y_M) <\infty$ and $E_M \in \Cproj{}(\cM)$.
        \item There exists a short exact sequence 
        \begin{equation}
            0 \to M \to Y^M \to E^M \to 0
        \end{equation}
        with $\pd(Y^M)<\infty$ and $E^M\in \Cproj{}(\cM)$.
    \end{enumerate}
\end{proposition}
\begin{proof}
    This is a special case of \cite[Proposition 1.5]{AB_CMapprox}, with resolving subcategory $\cX=\Cproj{}(\cM)$ and $\omega=\proj(\cM)$, noting that $\proj(\cA)$ is a cogenerator of $\Cproj{}(\cM)$ by \cref{prop:Cprojresolving}.
\end{proof}
\fi

We will need the following result in \cref{sec:PM-AR-translate}, which is essentially a special case of \cite[Proposition 1.5]{AB_CMapprox} by Auslander--Buchweitz; see also \cite[Theorem 6.1.3]{Krause_HomThRep}:
\begin{corollary}\label{cor:approx}
    Let $M \in \cM$ be such that $\ipd(M)=d<\infty$. 
    There exists a short exact sequence
        \begin{equation} \label{eqn:rightC-projapprox}
            0 \to Y_M \to E_M \xrightarrow{p} M \to 0
        \end{equation}
    with $\pd(Y_M)=d-1$ (unless $d=0$, in which case $\pd(Y_M)=0$) and $E_M \in \Cproj{}(\cM)$.
    Moreover, $p:E_M \to M$ is a right $\Cproj{}(\cM)$-approximation of $M$.
\end{corollary}
\begin{proof}
    The existence of the short exact sequence \eqref{eqn:rightC-projapprox} with $\pd(Y_M)<\infty$ and $E_M \in \Cproj{}(\cM)$ is a special case of \cite[Proposition 1.5]{AB_CMapprox}, with resolving subcategory $\cX\coloneqq \Cproj{}(\cM)$ and $\omega\coloneqq \proj(\cM)$, noting that $\proj(\cM)$ is a cogenerator of $\Cproj{}(\cM)$ by \cref{prop:Cprojresolving}.
    To show that $\pd(Y_M)=d-1$, first note that any length $m$ $\cC$-projective resolution of $Y_M$ induces a length $m+1$ $\cC$-projective resolution of $M$.
    It then follows from \cref{prop:Cprojdimprojdim} that we must have $\ipd(Y_M) = d-1$.
    Now since $\pd(Y_M) <\infty$, we get $\pd(Y_M)=\ipd(Y_M)=d-1$ by \cref{prop:finite_pd=finite_ipd}, as required.

    Finally, for any $E \in \Cproj{}(\cM)$, we have an exact sequence
    \[\Hom_\cM(E,E_M) \to \Hom_\cM(E,M)\to \Ext^1_{\cM}(E,Y_M).\]
    By Lemma~\ref{lem:orth}, the right term is zero, and hence the left map is surjective, proving that $E_M$ is a right $\Cproj{}(\cM)$-approximation as required.
\end{proof}

We spell out the construction of \eqref{eqn:rightC-projapprox} in the particular case of $\ipd(M) \leq 1$, since it is the case that we will need later (\cref{thm:PM-ARtrans}):
\begin{example}
    Let $M \in \cM$ be such that $\ipd(M) \leq 1$.
    Then by \cref{prop:Cprojdimprojdim} we know that there exists a $\cC$-projective resolution
    \[
    0 \to E \to Q \to M \to 0
    \]
    with $Q \in \proj(\cM)$ and $E \in \cC$-$\proj(\cM)$.
    On the other hand, using the short exact sequence in $\cC$: $0 \to \mathds{1} \to P\otimes P^* \to K \to 0$, we obtain the following short exact sequence in $\cM$:
    \[
    0 \to E \to P\otimes P^*\lact E \to K \lact E\to 0.
    \]
    The pushout of the two short exact sequences in $\cM$ gives
    \[
    \adjustbox{scale=0.8}{
    \begin{tikzcd}
     & 0 \ar[d] & 0 \ar[d] & 0 \ar[d] \\
    0 \ar[r] & E \ar[r] \ar[d] & Q \ar[r] \ar[d,dashed] & M \ar[r] \ar[d] & 0 \\
    0 \ar[r] & P\otimes P^* \lact E \ar[r,dashed] \ar[d] & E_M \ar[r,dashed] \ar[d,dashed] & M \ar[r] \ar[d] & 0 \\
    0 \ar[r] & K \lact E \ar[r] \ar[d] & K \lact E \ar[r] \ar[d] & 0 \ar[r] \ar[d] & 0 \\
     & 0 & 0 & 0 
    \end{tikzcd}
    },
    \]
    with exact rows and columns.
    Since $Q$ is projective -- hence $\cC$-projective -- and $K\lact E$ is $\cC$-projective (since $E$ is $\cC$-projective), it follows that their extension $E_M$ is also $\cC$-projective.
    Thus the middle row 
    \[
    0 \to P\otimes P^*\lact E \to E_M \to M \to 0
    \]
    is a $\cC$-projective resolution of $M$ with $E_M$ a $\cC$-projective and $P\otimes P^* \lact E$ a projective; compare this with the $\cC$-projective resolution of $M$ that we started with (i.e.\ the top row). 
    By \cref{cor:approx}, $E_M \twoheadrightarrow M$ is moreover a right $\cC$-projective approximation.
\end{example}

\subsection{Internal Hom and Ext} \label{sec:internal-hom-ext}
In this section we focus on the case where $\cM$ is moreover a Hom-finite $\cC$-module category, which allows us to define the following; cf.\ \cite[\S7.9]{EGNO}:
\begin{definition}
    Suppose $\cM$ is a Hom-finite $\cC$-module category.
    Let $M$ and $N$ be objects in $\cM$.
    Since $-\lact -$ is exact in both variables, the functor $\Hom_\cM(-\lact M, N): \cC^{\op} \to \vec$ is left exact, and hence is representable: the representing object is denoted by $\iHom_\cM(M,N)$, called the \emph{internal hom} from $M$ to $N$, satisfying
    \begin{equation} \label{eqn:inthom}
    \Hom_\cM(-\lact M,N) \cong \Hom_\cC(-,\iHom_\cM(M,N)).
    \end{equation}
\end{definition}
Note that $\iHom_\cM(-,-)$ preserves limits in both variables; precisely, $\iHom_\cM(M,-)$ preserves limits and the contravariant functor $\iHom_\cM(-,N)$ sends colimits to limits.
%The proof just uses the fact that adjoint functors (coming from duals in $\cC$) preserve (co)limits, and that the usual hom functor preserves limits.
It follows that $\iHom_\cM(-,-)$ is left exact in both variables.
We also have the following isomorphisms that are natural in all inputs $X \in \cC$ and $M,N \in \cM$:
\begin{itemize}
    \item $\iHom_\cM(M, X\lact N) \cong X\otimes \iHom_\cM(M,N)$.
    \item $\iHom_\cM(X \lact M, N) \cong \iHom_\cM(M,N)\otimes \leftdual{X}$.
\end{itemize}
In particular, the functor $\iHom_\cM(M,-): \cM \to \cC$ is a left $\cC$-module functor, and $\iHom_\cM(-,M): \rightop{\cM} \to \cC$ is a right $\cC$-module functor (see \cref{defn:inducedrightmod}).

From \eqref{eqn:inthom} we see that $-\lact M: \cC \to M$ is the left adjoint of $\iHom_\cM(M,-): \cM \to \cC$ for all $M \in \cM$ (as $\cC$-module functors).
We denote by 
\begin{equation}
    \coev_M : \id_\cC \Rightarrow \iHom_\cM(M,-\lact M), \quad \ev_M: \iHom_\cM(M,-)\lact M \Rightarrow \id_\cM
\end{equation} 
the corresponding unit (coevaluation) and counit (evaluation) respectively.
Note that $\coev_M(\unit): \unit \to \iHom_\cM(M,M)$ is therefore the image of the identity $\id_M \in \Hom_\cM(M,M)$ under the isomorphism in \eqref{eqn:inthom}.
Since $\coev_M$ and $\ev_M$ are natural transformations of $\cC$-module functors, the following diagrams are commutative for all $X, E \in \cC$ and $N \in \cM$:
\[
\begin{tikzcd}[column sep=huge]
    X \otimes E \ar[r, "\coev_M(X\otimes E)"] \ar[d, "="] & \iHom_\cM(M, X\otimes E\lact M) \ar[d, "\cong"] \\
    X \otimes E \ar[r, "X \otimes \coev_M(E)"] & X \otimes \iHom_\cM(M, E\lact M)
\end{tikzcd},
\begin{tikzcd}[column sep=large]
    \iHom_\cM(M, X \lact N) \lact M \ar[r, "\ev_M(X\lact N)"] \ar[d, "\cong"] & X\lact N \ar[d,"="] \\
    X \otimes \iHom_\cM(M, N) \lact M \ar[r, "X \otimes \ev_M(N)"] & X \lact N
\end{tikzcd}.
\]
\begin{definition}
    Let $\cM$ be a Hom-finite $\cC$-module category.
    Suppose that $\cM$ has enough injectives (and as always enough projectives).
    We define $\iExt^i_\mathcal{M}(X,-)$ to be the $i$-th derived functor of the left exact functor $\iHom_\mathcal{M}(X,-)$, which we call the \emph{internal Ext functors}.
    Similarly, $\iHom_\mathcal{M}(-,X)$ is a left exact (contravariant) functor and we define $\iExt^i_\mathcal{M}(-,X)$ to be its $i$-th derived functor.
\end{definition}

\begin{lemma}\label{lem:Ext-iExt_AgreeOnProj}
Let $\cM$ be a Hom-finite $\cC$-module category with enough injectives.
\begin{enumerate}
    \item \label{item:intExtfromproj}We have $\Hom_{\cC}(P,\iExt^\bullet_\cM(X,Y)) \cong \Ext^\bullet_\cM(P\lact X,Y)$ for all $P \in \proj(\cC)$ which is natural in $P\in\proj(\cC)$ and both $X, Y\in \cM$.
    \item \label{item:ExtMN} $\iExt^i_\mathcal{M}(M,-)$ applied to $N$ and $\iExt^i_\mathcal{M}(-,N)$ applied to $M$ are isomorphic.
\end{enumerate}
    
\end{lemma}
\begin{proof}
By definition of the internal hom, we have an isomorphism of functors:
    \[
    \Hom_\cM(P\lact X, - ) \cong \Hom_\cC(P, \iHom_\cM(X,-)).
    \]
    Note that $\Ext^\bullet_\cM(P\lact X,-)$ are the derived functors of $\Hom_\cM(P\lact X, - )$.
    On the other hand, for $P \in \proj(\cC)$, $\Hom_\cC(P, -)$ is moreover an exact functor, and so commutes with taking derived functors. This gives us the isomorphism
    \[
    \Ext^\bullet_\cM(P\lact X,-) \cong \Hom_\cC(P, \iExt^\bullet_\cM(X,-))
    \]
    as required for part \ref{item:intExtfromproj}.

    Note that two objects $X,Y\in\cC$ are isomorphic if and only if the functor
    \[\Hom_{\cC}(-,X):\proj(\cC)^{\op}\to\vec\]
    is isomorphic to the corresponding one for $Y$. Hence \ref{item:ExtMN} follows from \ref{item:intExtfromproj}.
\end{proof}

\begin{lemma}\label{lem:CprojHomfinite}
    Let $\cM$ be a Hom-finite $\cC$-module category with enough injectives and $E \in \cM$.
    The following are equivalent:
    \begin{enumerate}
        \item \label{item:Cproj} $E$ is $\cC$-projective.
        \item \label{item:Cprojforall} $Q\lact E$ is projective for all $Q \in \proj(\cC)$.
        \item \label{item:iHomexact} The internal hom functor $\iHom_\cM(E,-)$ is exact.
        \item \label{item:Ext1} $\iExt_\cM^1(E,-)=0$.
        \item \label{item:Exti} For all $i>0$, $\iExt_\cM^i(E,-)=0$.
    \end{enumerate}
\end{lemma}
\begin{proof}
    The equivalence $\ref{item:Cproj} \iff \ref{item:Cprojforall}$ is \cref{lem:Cproj}.
    
    By definition of the internal hom, we have 
    \[
    \Hom_\cC(P,\iHom_\cM(E,-)) \cong \Hom_\cM(P\lact E,-).
    \]
    If $\iHom_\cM(E,-)$ is exact, then the LHS is also exact for any $P\in \proj(\cC)$.
    In turn, exactness of the RHS implies that $P\lact E \in \proj(\cM)$, giving us $\ref{item:iHomexact}\implies \ref{item:Cprojforall}$.

    On the other hand, suppose that $P\lact E \in \proj(\cM)$ for all $P\in \proj(\cC)$.
    Then the RHS is exact for all $P \in \proj(\cC)$.
    Since $\cC$ has enough projectives, exactness of $\Hom_\cC(P,\iHom_\cM(E,-))$ for all $P \in \proj(\cC)$ implies the exactness of $\iHom_\cM(E,-)$.
    This shows $\ref{item:Cprojforall} \implies \ref{item:iHomexact}$.
    
    That $\ref{item:iHomexact}\implies\ref{item:Exti}$ follows from Lemma~\ref{lem:Ext-iExt_AgreeOnProj}\ref{item:ExtMN}. That $\ref{item:Exti}\implies\ref{item:Ext1}$ is trivial. That $\ref{item:Ext1}\implies\ref{item:iHomexact}$ follows from Lemma~\ref{lem:Ext-iExt_AgreeOnProj}, which completes the proof.
\end{proof}

\begin{proposition} \label{prop:ipd=iExt}
Suppose $\cM$ is a Hom-finite $\cC$-module category with enough injectives and let $X \in \cM$. Then 
\begin{equation}
    \ipd(X) = \min \{ d \mid \iExt^{d+1}_\cM(X,-) = 0 \}.
\end{equation}
($\ipd(X) = \infty$ if $\iExt^{d+1}_\cM(X,-) \neq 0$ for all $d$.)
\end{proposition}
\begin{proof}
    This is immediate since $\cC$-projectives are precisely the acyclic objects with respect to $\iHom_\cM(-,Y)$ by \cref{lem:CprojHomfinite}.
\end{proof}
The proposition has the following immediate consequence.
\begin{corollary}
    Let $\cM$ be a Hom-finite $\cC$-module category with enough injectives.
    The following are equivalent, for $d\in\mathbb{N}$:
    \begin{enumerate}
        \item $\iExt^{d+1}_{\cM}=0$ and $\iExt^d_\cM \neq 0$.
        \item The $\cC$-projective dimension of $\cM$ is $d$.
    \end{enumerate}
\end{corollary}

\section{Functor categories and Auslander--Reiten translations}\label{sec:functor-cat}
Throughout this section, we fix $\cC$ to be a finite tensor category and $\cM$ a \emph{finite} left $\cC$-module category.
%Nonetheless, we note that \cref{lem:intYoneda} holds for non-finite $\cM$.

\subsection{Enriched Yoneda lemma and functor categories}\label{sec:enriched-Yoneda}
Let $\ind(\cC)$ denote the ind-completion of $\cC$: this can be constructed as the full subcategory of presheaves $\Fun(\cC^{\op},\Vec)$ that are small filtered colimits of representable functors.
More concretely, $\ind(\cC)$ can be realised as the category of left exact functors from $\cC^{\op} \to \Vec$ (see e.g.\ \cite[\S 8.6]{Kashiwara-Schapira_Sheaves}), hence $\ind(\cC) \cong \Fun(\proj(\cC)^{\op},\Vec)$.
The important point for us is that $\ind(\cC)$ has all small filtered colimits and is a monoidal category which has exact tensor products, i.e.\ $X\otimes -$ and $-\otimes X$ are exact endofunctors (though as a monoidal category, its objects need not have duals). 
In particular, $\ind(\cC)$ is both a left and right module category over $\cC$ given by:
\[
X\lact F \coloneqq F(\leftdual{X}\otimes -); \qquad F\ract X \coloneqq F(-\otimes \rightdual{X})
\]

Recall the definition of $\rightop{\cM}$ from \cref{defn:inducedrightmod}.
\begin{definition}
    We denote by $\Fun_\cC(\cM,\ind(\cC))$ the category of left $\cC$-module functors from $\cM$ to $\ind(\cC)$, with morphisms given by $\cC$-module natural transformations.
    Similarly, $\Fun_\cC(\rightop{\cM},\ind(\cC))$ denotes the category of right $\cC$-module functors from $\rightop{\cM}$ to $\ind(\cC)$.
\end{definition}
Note that both $\Fun_\cC(\cM,\ind(\cC))$ and $\Fun_\cC(\rightop{\cM},\ind(\cC))$ are abelian categories with kernels and cokernels computed point-wise: $\ind(\cC)$ is an abelian category, and the coherence of $\cC$-module functors ensures that kernels and cokernels are again $\cC$-module functors.

Most results in this section will have two parts: one relevant to $\Fun_\cC(\cM,\ind(\cC))$ and the other relevant to $\Fun_\cC(\rightop{\cM},\ind(\cC))$.
Since the proofs for both cases will be similar, to be succinct we only provide proofs for the case relevant to $\Fun_\cC(\cM,\ind(\cC))$, leaving the other case to the reader.

The following is a version of the (weak) enriched Yoneda lemma; see e.g.\  \cite[\S1.9]{Kelly_EnrichedCat}.
We include a proof below in our setting for the convenience of the reader.
Recall that $\ev_M$ and $\coev_M$ denote the counit and unit of the adjunction $-\lact M \dashv \iHom_\cM(M,-)$ respectively, where $\coev_M(\unit):\unit\to\iEnd_{\cM}(M)$ is the image of the identity $\id_M$ under the isomorphism in \eqref{eqn:inthom}.
\begin{lemma}[Weak enriched Yoneda lemma] \label{lem:intYoneda}
\begin{enumerate}
    \item For any left $\cC$-module functor $F \in \Fun_\cC(\cM,\ind(\cC))$ and $M\in\cM$, the evaluation
    \[\Hom(\iHom_{\cM}(M,-),F)\xrightarrow{\eta\mapsto \eta(M)\circ \coev_M(\unit)}\Hom_{\ind(\cC)}(\unit, F(M))\]
    is an isomorphism.
    \item For any right $\cC$-module functor $F \in \Fun_\cC(\rightop{\cM},\ind(\cC))$ and $M\in \cM$, the evaluation
    \[\Hom(\iHom_{\cM}(-,M),F)\xrightarrow{\eta\mapsto \eta(M)\circ \coev_M(\unit)}\Hom_{\ind(\cC)}(\unit, F(M))\]
    is an isomorphism.
    \end{enumerate}
\end{lemma}
\begin{proof}
    The inverse map is given by sending $x\in \Hom_{\ind(\cC)}(\unit, F(M))$ to the natural transformation $\eta$ defined by the following composition for each $N \in \cM$
    \[\eta(N):\iHom_{\cM}(M,N)\xrightarrow{1\otimes x}\iHom_{\cM}(M,N)\otimes F(M)\xrightarrow{\sim}F(\iHom_{\cM}(M,N)\lact M)\to F(N),\]
    where the last two morphisms come from the $\cC$-module functor structure of $F$ and the evaluation of $F$ on $\ev_M(N):\iHom_{\cM}(M,N)\lact M\to N$ respectively.

    One verifies quickly that the two maps are mutually inverse. One direction is immediate. For the second direction we start from a natural transformation $\eta: \iHom(M,-)\Rightarrow F$ of $\cC$-module functors. Then we can consider the commutative diagram
    \[
    \begin{tikzcd}
    \iHom(M,N)\ar[r,"1\otimes \coev_M"] \ar[rdd,swap,bend right=30,"="] & \iHom(M,N)\otimes\iHom(M,M) \ar[r,"1\otimes \eta(M)"] \ar[d,"\sim"] & \iHom(M,N)\otimes F(M) \ar[d,"\sim"] \\
    & \iHom(M,\iHom(M,N)\lact M) \ar[r,"\eta"] \ar[d,"\protect{\iHom(M,\ev_M)}"] & F(\iHom(M,N)\lact M)\ar[d,"F(\ev_M)"] \\
    & \iHom(M,N)\ar[r,"\eta(N)"] & F(N).
    \end{tikzcd}
    \]
    The composition of morphisms from $\iHom(M,N)$ to itself is the identity morphism since $\ev_M$ and $\coev_M$ are counit and unit of an adjunction.
    The upper commutative square is the expression that the original natural transformation $\eta$ is one of module functors. The lower commutative square expresses that this $\eta$ is a natural transformation.
    The outer edge is how we define a natural transformation $\eta$ from a morphism $\unit\to F(M)$, which in turn comes from a natural transformation $\eta$ of $\cC$-module functors.
\end{proof}

\begin{corollary}\label{cor:fix}
\begin{enumerate}
    \item For every $M\in\cP\cM$, the object $\iHom_{\cM}(M,-)$ is projective in $\Fun_\cC(\cM,\ind(\cC))$. Moreover, for every $0\not=F \in \Fun_\cC(\cM,\ind(\cC))$, there exists $M\in\cP\cM$ with a non-zero morphism $\iHom_{\cM}(M,-)\to F$.
    \item For every $M\in\cP\cM$, the object $\iHom_{\cM}(-,M)$ is projective in $\Fun_\cC(\rightop{\cM},\ind(\cC))$. Moreover, for every $0\not=F \in \Fun_\cC(\rightop{\cM},\ind(\cC))$, there exists $M\in\cP\cM$ with a non-zero morphism $\iHom_{\cM}(-,M)\to F$.
\end{enumerate}
\end{corollary}
\begin{proof}
	Since $M\in\cP\cM$ must be a direct summand of $P\lact M'$ for some $P \in \proj(\cC)$ and $M' \in \cM$, and projectives are closed under taking summands, it suffices to show that $\iHom_\cM(P\lact M',-)$ is projective.
    By \cref{lem:intYoneda}:
    \[\Hom(\Hom_{\cM}(P\lact M',-),F)\cong \Hom_{\ind(\cC)}(\leftdual{P},F(M')),\]
    from which it follows that $\Hom(\Hom_{\cM}(P\lact M',-),-)$ is an exact functor.
    Hence $\Hom_{\cM}(P\lact M',-)$ is projective.

   Note that $0\not=F$ simply means that $F(M')\not=0$ for some $M'\in\cM$, which in turn means that
   \[\Hom_{\ind(\cC)}(Q,F(M'))\not=0\]
   for some $Q\in\proj(\cC)$. Consequently for the second statement we can take $M=\rightdual{Q}\lact M'$ and apply \cref{lem:intYoneda}.
\end{proof}

Let $\cA$ be an abelian category.
Recall that an object $X \in \cA$ is said to be \textbf{finitely presented} if $\Hom_\cA(X,-)$ preserves all small filtered colimits in $\cA$.
The following lemma characterises the finitely presented objects in the category $\Fun_\cC(\cM,\ind(\cC))$ of $\cC$-module functors.
\begin{lemma}\label{lem:FPiffFPfunctor}
    For $F \in \Fun_\cC(\cM,\ind(\cC))$, the following are equivalent:
    \begin{enumerate}
        \item \label{item:finprensented} $F$ is finitely presented as an object in $\Fun_\cC(\cM,\ind(\cC))$.
        \item \label{item:C-projpresentfun} There exists an exact sequence in $\Fun_\cC(\cM,\ind(\cC))$:
            \[
            \iHom_\cM(N,-) \to \iHom_\cM(M,-) \to F \to 0
            \]
            with $M,N \in \cM$.
        \item \label{item:projpresentfun} There exists an exact sequence in $\Fun_\cC(\cM,\ind(\cC))$:
            \[
            \iHom_\cM(N,-) \to \iHom_\cM(M,-) \to F \to 0
            \]
            with $M,N \in \cP\cM$.
    \end{enumerate}
    For $F \in \Fun_\cC(\rightop{\cM},\ind(\cC))$, the following are equivalent:
    \begin{enumerate}
        \item $F$ is finitely presented as an object in $\Fun_\cC(\rightop{\cM},\ind(\cC))$.
        \item There exists an exact sequence in $\Fun_\cC(\rightop{\cM},\ind(\cC))$:
            \[
            \iHom_\cM(-,M) \to \iHom_\cM(-,N) \to F \to 0
            \]
            with $M,N \in \cM$.
        \item There exists an exact sequence in $\Fun_\cC(\rightop{\cM},\ind(\cC))$:
            \[
            \iHom_\cM(-,M) \to \iHom_\cM(-,N) \to F \to 0
            \]
            with $M,N \in \cP\cM$.
    \end{enumerate}
\end{lemma}
\begin{proof}
    Denote $\cF \coloneqq \Fun_\cC(\cM,\ind(\cC))$.
    We see that
    \begin{align*}
    \Hom_\cF(\iHom_\cM(M,-), \colim_i F_i) 
        &\cong \Hom_{\ind(\cC)}(\unit, \colim_i F_i(M)) \\
        &\cong \colim_i \Hom_{\ind(\cC)}(\unit, F_i(M)),
    \end{align*}
    where the first isomorphism follows from \cref{lem:intYoneda} and the second follows from the fact that $\unit$ is a finitely presented object in $\ind(\cC)$.
    Thus $\iHom_\cM(M,-)$ is finitely presented for any $M \in \cM$ (which includes $M\in\cP\cM$).
    It then follows that any $F \in\Fun_\cC(\cM,\ind(\cC))$ equipped with an exact sequence
    \[
    \iHom_\cM(N,-) \to \iHom_\cM(M,-) \to F \to 0
    \]
    is finitely presented.
    This shows \ref{item:C-projpresentfun} implies \ref{item:finprensented}.
    Since \ref{item:projpresentfun} implies \ref{item:C-projpresentfun} is immediate, it suffices to show \ref{item:finprensented} implies \ref{item:projpresentfun}.

Let $F\in\Fun_\cC({\cM},\ind(\cC))$ be arbitrary. We consider the set of subobjects of $F$ that are the image of morphisms 
\[\iHom_{\cM}(M,-)\to F,\quad\mbox{for}\quad M\in\cM.\]
We can observe, using additivity of $\cM$ that this set is filtered for the inclusion order. If we denote by $F'\subset F$ the colimit of this system of subobjects, then we can apply \cref{cor:fix} to $F/F'$. If the quotient is non-zero, we find some non-zero morphism from some $\iHom_{\cM}(M,-)$, which must lift to $F$, giving a contradiction with the construction of $F'$. Hence $F'=F$.

Now assume that $F$ is finitely generated (meaning $\Hom(F,-)$ commutes with filtered colimits of monomorphisms). The equality $F=F'$ then shows that $F'$ must be the image of one single (epimorphism) $\iHom_{\cM}(M,-)\to F$. If $F$ is moreover finitely presented, then the kernel of this epimorphism is finitely generated, so that we know by the same argument that it must also be a quotient of some $\iHom_{\cM}(N,-)$, concluding the proof.
\end{proof}

\begin{definition}\label{defn:finitepresentedfunc}
    We use $\Fun^{\fp}_\cC(\cM,\cC)$ to denote the full subcategory of finitely presented functors in $\Fun_\cC(\cM,\ind(\cC))$, i.e.\ functors $F$ satisfying any of the equivalent conditions in \cref{lem:FPiffFPfunctor}.
    Similarly, $\Fun^{\fp}_\cC(\rightop{\cM},\cC)$ denotes the full subcategory of finitely presented functors in $\Fun_\cC(\rightop{\cM},\ind(\cC))$.
\end{definition}
Note that a finitely presented functor $F$ necessarily takes values in $\cC$, which justifies our notation $\Fun^{\fp}_\cC(\cM,\cC)$ (instead of $\Fun^{\fp}_\cC(\cM,\ind(\cC))$).
We can now state the Yoneda embedding provided by internal Homs:
\begin{corollary}\label{cor:intHom-Yon-Emb}
    For $N,M\in\cM$, the assignment \[
	\Hom_{\cM}(N,M) \xrightarrow{f \mapsto \iHom_\cM(f,-)} \Hom(\iHom_{\cM}(M,-),\iHom_{\cM}(N,-))
	\]
	is an isomorphism, so that the (contravariant) internal Hom Yoneda embedding
    \[
    \begin{split}
        \rightop{\cM} &\to \Fun^{\fp}_\cC(\cM,\cC) \subset \Fun_\cC(\cM,\ind(\cC)) \\
        M &\mapsto \iHom_\cM(M,-)
    \end{split}   
    \]
    is a fully faithful functor.
    Similarly, we have a (covariant) fully faithful functor:
    \[
    \begin{split}
        \cM &\to \Fun^{\fp}_\cC(\rightop{\cM},\cC) \subset \Fun_\cC(\rightop{\cM},\ind(\cC)) \\
        M &\mapsto \iHom_\cM(-,M).
    \end{split}
    \]
\end{corollary}
\begin{proof}
    By \cref{lem:intYoneda} and definition of $\coev$, the arrows pointing to the right in
    \[\xymatrix{
\Hom(\iHom_{\cM}(M,-),\iHom_{\cM}(N,-)) \ar[rr]^{\eta\mapsto \eta(M)\circ \coev_M(\unit)} &&\Hom_{\cC}(\unit,\iHom_{\cM}(N,M))\\
&\Hom_\cM(N,M)\ar[ur]_{f\mapsto \iHom_\cM(N,f)\circ\coev_N(\unit)}\ar[lu]^{f \mapsto \iHom_\cM(f,-)}
    }\]
    are isomorphisms. To conclude that the leftward arrow is an isomorphism it thus suffices to show that the triangle is commutative. In other words, we need the equality
    \[\iHom_{\cM}(f,M)\circ\coev_M(\unit)=\iHom_{\cM}(N,f)\circ\coev_N(\unit)\]
    of morphisms $\unit\to\iHom_\cM(N,M)$, for arbitrary $f\in\Hom_{\cM}(N,M)$. To verify this, it suffices to apply $\Hom_\cC(X,-)$, for arbitrary $X\in\cC$, so that after adjunction \eqref{eqn:inthom} the two maps
    \[\Hom_\cC(X,\unit)\rightrightarrows\Hom_\cM(X\lact N,M)\]
    are $x\mapsto (\unit\lact f)\circ(x\lact N)$ and $x\mapsto (x\lact M)\circ (X\lact f)$, which are indeed the same by naturality of $x \lact - : X \lact - \to \unit \lact -$.
\end{proof}
\begin{remark}\label{rmk:C-equivariant-int-Hom-Yon-embedding}
    We will see in \cref{sec:C-action-Fun} that $\Fun^{\fp}_\cC(\cM,\cC)$ (resp.\ $\Fun^{\fp}_\cC(\rightop{\cM},\cC)$) carries a natural right (resp.\ left) $\cC$-module structure, where it will follow immediately that the internal Hom Yoneda embedding above moreover is a $\cC$-module functor.
\end{remark}

\begin{lemma} \label{lem:projFunfp}
    \begin{enumerate}
        \item The projective objects in $\Fun^{\fp}_\cC(\cM,\cC)$ are precisely the functors $\iHom_\cM(M,-)$ for some $M \in \cP\cM$.
        \item The projective objects in $\Fun^{\fp}_\cC(\rightop{\cM},\cC)$ are precisely the functors $\iHom_\cM(-,M)$ for some $M \in \cP\cM$.
    \end{enumerate}
\end{lemma}
\begin{proof}
That that $\iHom_\cM(M',-)$ are projective for all $M' \in \cP\cM$ is observed in \cref{cor:fix}.
	Now let $F \in \Fun^{\fp}_\cC(\cM,\cC)$ be a projective object. 
	Since $F$ is finitely presented, by \cref{lem:FPiffFPfunctor} we may assume that we have some epimorphism $\pi: \iHom_\cM(M,-) \Rightarrow F$ with $M \in \cP\cM$, which splits by the projectivity assumption on $F$.
	With $\iota$ denoting the splitting monomorphism, we get the following idempotent: 
	\[
		\eta\coloneqq \iota \circ \pi: \iHom_\cM(M,-) \Rightarrow \iHom_\cM(M,-).
	\]
	By \cref{lem:intYoneda}, this corresponds to an idempotent  
	\[
		e \in \Hom_\cM(M,M) \cong \Hom_\cC(\unit,\iHom_\cM(M,M)).
	\]
	It follows that $F\cong \iHom_\cM(E,-)$ for $E \coloneqq \mathrm{im}(e)$ a summand of $M$, which is in $\cP\cM$ as required.
\end{proof}

\begin{proposition}
    $\Fun^{\fp}_\cC(\cM,\cC)$ and $\Fun^{\fp}_\cC(\rightop{\cM},\cC)$ are abelian categories which are Hom-finite and have enough projectives.
\end{proposition}
\begin{proof}
It is a general fact that the class of finitely presented objects is closed under taking finite direct sums and cokernels. One can also produce proofs using the equivalent characterisation in \cref{lem:FPiffFPfunctor}, similar to the case of kernels below.

%    Throughout this proof we use the characterisation of finitely presented functors given by \cref{lem:FPiffFPfunctor}.
 %   
%    Then 
%    \[
%    \iHom_\cM(N_1\oplus N_2,-) \to \iHom_\cM(M_1 \oplus M_2,-) \to F_1\oplus F_2 \to 0,
 %   \]
 %   so $F_1\oplus F_2 \in \Fun^{\fp}_\cC(\cM,\cC)$.
  
    Let $F_1,F_2 \in \Fun^{\fp}_\cC(\cM,\cC)$ with presentation
    \[
    \iHom_\cM(N_i,-) \xrightarrow{\iHom_\cM(p_i,-)} \iHom_\cM(M_i,-) \to F_i \to 0,
    \]
    where $M_i,N_i \in \cP\cM$ via the characterisation in \cref{lem:FPiffFPfunctor}\ref{item:projpresentfun}, and application of \cref{cor:intHom-Yon-Emb}.
    Consider a $\cC$-module natural transformation $\phi: F_1 \Rightarrow F_2$. We can extend it to a morphism of exact sequences by projectivity of $\iHom_\cM(M_i,-)$ and $\iHom_\cM(N_i,-)$ (\cref{lem:projFunfp}):
    \[\begin{tikzcd}
        \iHom_\cM(N_1,-) \ar[r]\ar[d,"\protect{\iHom_\cM(f_N,-)}"] & \iHom_\cM(M_1,-) \ar[r]\ar[d,"\protect{\iHom_\cM(f_M,-)}"] & F_1 \ar[r]\ar[d,"\phi"] & 0 \\
       \iHom_\cM(N_2,-) \ar[r] & \iHom_\cM(M_2,-) \ar[r] & F_2 \ar[r] & 0.
    \end{tikzcd}\]

   % It follows that we can equip $\coker(\phi)$ with presentation
  %  \[
  %  \iHom_\cM(N_2\oplus M_1,-) \to \iHom_\cM(M_2,-) \to \coker(\phi) \to 0.
 %   \]

   Consider the following exact sequences with $M_0, N_0 \in \cM$ defined by taking cokernels in $\cM$:
    \[
    M_2 \xrightarrow{\begin{bsmallmatrix}f_M \\ p_2 \end{bsmallmatrix}} M_1\oplus N_2 \xrightarrow{\begin{bsmallmatrix}\pi_1 & \pi_2 \end{bsmallmatrix}} M_0 \to 0; \qquad
    M_1 \xrightarrow{\begin{bsmallmatrix}\pi_1 \\ p_1 \end{bsmallmatrix}} M_0\oplus N_1 \to N_0 \to 0.
    \]
    This gives rise to the following morphism of exact sequences
    \[\begin{tikzcd}
        \iHom_\cM(N_0,-) \ar[r]\ar[d] & \iHom_\cM(M_0,-) \ar[r]\ar[d] & \ker(\phi) \ar[r]\ar[d] & 0 \\
        \iHom_\cM(N_1,-) \ar[r] & \iHom_\cM(M_1,-) \ar[r] & F_1 \ar[r] & 0 
    \end{tikzcd},\]
    hence it follows that $\ker(\phi) \in \Fun_\cC^{\fp}(\cM,\cC)$ by \ref{item:C-projpresentfun} in \cref{lem:FPiffFPfunctor}.
    All in all, we have shown that $\Fun_\cC^{\fp}(\cM,\cC)$ is closed under taking finite direct sums, cokernels and kernels, hence is abelian.
    %It follows from the characterisation that $\Fun^{\fp}_\cC(\cM,\cC)$ is closed under taking direct sums, cokernels and kernels, and hence is abelian; see e.g.\ the proof of \cite[Lemma 2.1.6]{Krause_HomThRep}, replacing Hom with $\iHom_\cM$.
    
    The characterisation also guarantees that $\Fun^{\fp}_\cC(\cM,\cC)$ is Hom-finite, and since $\iHom_\cM(M,-)$ is projective for all $M\in \cP\cM$ by \cref{lem:projFunfp}, it also follows that $\Fun^{\fp}_\cC(\cM,\cC)$ has enough projectives.
\end{proof}

\begin{definition}\label{defn:FunfpPM}
We let $\Fun^{\fp}(\cP\cM,\vec)$ denote the subcategory of finitely presented objects in the category of ($\Bbbk$-linear additive) functors from $\cP\cM$ into $\Vec$.
Equivalently, $F \in \Fun^{\fp}(\cP\cM,\vec)$ if and only if there exists an exact sequence
\[
\Hom_\cM(N,-)|_{\cP\cM} \to \Hom_\cM(M,-)|_{\cP\cM} \to F \to 0
\]
in the category of functors from $\cP\cM$ to $\Vec$, with $M,N \in \cP\cM$; see e.g.\ \cite[Proposition 11.1.9]{Krause_HomThRep}.
Similarly, $\Fun^{\fp}(\cP\cM^{\op},\vec)$ denotes the subcategory of finitely presented objects in the category of ($\Bbbk$-linear additive) contravariant functors from $\cP\cM$ into $\Vec$.
\end{definition}
We note that $\Fun^{\fp}(\cP\cM, \vec)$ and $\Fun^{\fp}(\cP\cM^{\op},\vec)$ are abelian categories since $\cP\cM$ has weak cokernels and weak kernels by \cref{cor:weakcokernels}; see e.g.\ \cite[Lemma 2.1.6]{Krause_HomThRep}.
This is also a consequence of \cref{thm:linearisingCfunctor} below.

\begin{theorem} \label{thm:linearisingCfunctor}
    \begin{enumerate}
        \item The functor
        \[
        \Fun^{\fp}_\cC(\cM,\cC) \xrightarrow{\sim} \Fun^{\fp}(\cP\cM,\vec)
        \]
        sending $F \mapsto \Hom_\cC(\unit, F(-)|_{\cP\cM})$ is an equivalence.
        \item The functor
        \[
        \Fun^{\fp}_\cC(\rightop{\cM},\cC) \xrightarrow{\sim} \Fun^{\fp}(\cP\cM^{\op},\vec)
        \]
        sending $F \mapsto \Hom_\cC(\unit, F(-)|_{\cP\cM})$ is an equivalence.
    \end{enumerate}
\end{theorem}
\begin{proof}
    Let $\cF \coloneqq \Fun^{\fp}_\cC(\cM,\cC)$.
    Then we have an equivalence of abelian categories 
    \[
    \cF \cong \Fun^{\fp}(\proj(\cF)^{\op}, \vec)
    \]
    sending $F \mapsto \Hom_\cF(-,F)|_{\proj(\cF)}$; see \cite[Lemma 2.7]{Krause_defcat} or \cite[Lemma 18.1.13]{Prest_Purity-Spectra-local} for the dual version in terms of injectives.
    Recall that
    \[
    \proj(\cF) = \{\iHom_\cM(M,-) \mid M \in \cP\cM\}
    \]
    by \cref{lem:projFunfp}.
    Combining this with \cref{lem:intYoneda}, we get the equivalence
    \[
    \Fun^{\fp}(\proj(\cF)^{\op}, \vec) \cong \Fun^{\fp}(\cP\cM, \vec)
    \]
    sending $\Hom_\cF(-,F)|_{\proj(\cF)} \mapsto \Hom_\cC(\unit, F(-)|_{\cP\cM})$.
    This shows that $\Fun^{\fp}_\cC(\cM,\cC) \xrightarrow{\sim} \Fun^{\fp}(\cP\cM,\vec)$ as defined is an equivalence.
\end{proof}

\subsection{\texorpdfstring{$\cC$}{C}-projective dimension of functor categories} \label{sec:C-action-Fun}
\begin{definition}
    Define a right $\cC$-module category structure on $\Fun_\cC(\cM,\ind(\cC))$ as follows.
    For each $X \in \cC$ and $F \in \Fun_\cC(\cM,\ind(\cC))$, $F \ract X$ is the functor defined by
    \begin{equation}\label{eqn:C-actionFuntorCat}
        (F \ract X)(M) \coloneqq F(M)\otimes X.
    \end{equation}
    Similarly, $\Fun_\cC(\rightop{\cM},\ind(\cC))$ is defined as a left $\cC$-module category via
    \[
    (X\lact F)(M) \coloneqq X\otimes F(M).
    \]
\end{definition}
Note that $(F \ract -)(M) = F(M) \otimes -$ is exact for all $M$ 
%Note that $X \otimes -$ and $- \otimes X$ is still exact in $\ind(\cC)$ for all $X \in \ind(\cC)$.
and so is $- \ract X = -\otimes X$, hence the resulting bifunctor $-\ract - : \Fun_\cC(\cM,\ind(\cC)) \times \cC \to \Fun_\cC(\cM,\ind(\cC))$ is exact in both variables.
Since 
    \[
    \iHom_\cM(M,-) \ract X =\iHom_\cM(M,-)\otimes X \cong \iHom_\cM(\rightdual{X}\lact M,-), 
    \]
it follows from \cref{lem:FPiffFPfunctor} that $\Fun^{\fp}_\cC(\cM,\cC)$ is closed under the right $\cC$-action defined in \eqref{eqn:C-actionFuntorCat}, hence is a right $\cC$-module (sub)category.
The same can be said for $\Fun_\cC(\rightop{\cM},\ind(\cC))$ and $\Fun_\cC^{\fp}(\rightop{\cM},\cC)$.

The corresponding right $\cC$-module category structure on $\Fun^{\fp}(\cP\cM, \vec)$ under the equivalence in \cref{thm:linearisingCfunctor} can be defined as follows.
The bifunctor $- \ract -: \Fun^{\fp}(\cP\cM, \vec) \times \cC \to \Fun^{\fp}(\cP\cM, \vec)$ is given by
\begin{equation} \label{eqn:C-actionFunPM}
G\ract X \coloneqq G(\rightdual{\rightdual{X}}\lact -)
\end{equation}
As before, $-\ract X$ sends finitely presented functors to finitely presented functors, since 
\[
\Hom_\cM(M,-) \ract X = \Hom_\cM(M, \rightdual{\rightdual{X}}\lact -) \cong \Hom_\cM(\rightdual{X} \lact M, -),
\]
and $\cP\cM$ is closed under the $\cC$-action.
The fact that this is compatible with the equivalence in \cref{thm:linearisingCfunctor} follows from a direct computation, which we leave to the reader.
Similarly, we have a left $\cC$-module category structure on $\Fun^{\fp}(\cP\cM^{\op}, \vec)$ under the equivalence \cref{thm:linearisingCfunctor} given by
\begin{equation}\label{eqn:rightopPMaction}
X\lact G \coloneqq G(\leftdual{X}\lact -)
\end{equation}

\begin{example}
    Let $\cM$ be a finite left $\cC$-module category which is moreover exact (see \cref{defn:modulecat}).
    In this case, we have that $\cP\cM = \proj(\cM) = \inj(\cM)$.
    Then we have the following equivalences of right $\cC$-module categories:
    \[
    \rightop{\cM} \cong \Fun^{\fp}(\inj(\cM),\vec) = \Fun^{\fp}(\cP\cM,\vec) \cong \Fun^{\fp}_\cC(\cM,\cC);
    \]
    note that the first equivalence of $\cC$-module categories is given by sending $M \mapsto \Hom_\cM(M,-)|_{\inj(\cM)}$, which is known to be an equivalence of abelian categories $\cM^{\op} \cong \Fun^{\fp}(\inj(\cM),\vec)$ (see \cite[Lemma 18.1.13]{Prest_Purity-Spectra-local}).
    In particular, when $S=M \in \cM$ is simple, then $\Hom_\cM(S,-)|_{\inj(\cM)}$ is a simple functor. 
    Indeed, the only indecomposable injective(-projective) $I \in \inj(\cM)$ such that $\Hom_\cM(S,I) \neq 0$ is the injective envelope of $S$, which is unique up to isomorphism.

    Similarly, we have the following equivalences of left $\cC$-module categories:
    \[
    \cM \cong \Fun^{\fp}(\proj(\cM)^{\op},\vec) = \Fun^{\fp}(\cP\cM^{\op},\vec) \cong \Fun^{\fp}_\cC(\rightop{\cM},\cC).
    \]
\end{example}

\begin{example}
    Suppose that $\cM$ is $\cC$-rep-finite, so that $\cP\cM$ has only finitely many indecomposable objects.
    Then $\Fun^{\fp}_\cC(\rightop{\cM},\cC) \cong \Fun^{\fp}(\cP\cM^{\op},\vec)$ is a finite $\cC$-module category; $\cP\cM$ having only finitely many indecomposables ensures that there are only finitely many simple functors and every functor has finite length.
    These functor categories will be part of a $\cC$-equivariant Auslander(--Solberg) correspondence that we study in Section \ref{sec:equiv-Aus-cor}.
\end{example}

As in the previous subsection, we will only provide proofs of the results relevant to $\Fun_\cC(\cM,\ind(\cC))$, and omit the proofs for $\Fun_\cC(\rightop{\cM},\ind(\cC))$.

\begin{proposition}\label{prop:C-projrepresentable}
    \begin{enumerate}
        \item The $\cC$-projective objects in $\Fun^{\fp}_{\cC}(\cM,\cC)$ are precisely the functors $\iHom_{\cM}(M,-)$, for $M\in\cM$.
        Equivalently, the $\cC$-projective objects in $\Fun^{\fp}(\cP\cM,\vec)$ are precisely the functors $\Hom_{\cM}(M,-)|_{\cP\cM}$, for $M\in\cM$.
        \item The $\cC$-projective objects in $\Fun^{\fp}_{\cC}(\rightop{\cM},\cC)$ are precisely the functors $\iHom_{\cM}(-,M)$, for $M\in\cM$.
        Equivalently, the $\cC$-projective objects in $\Fun^{\fp}(\cP\cM^{\op},\vec)$ are precisely the functors $\Hom_{\cM}(-,M)|_{\cP\cM}$, for $M\in\cM$.
    \end{enumerate}
\end{proposition}
\begin{proof}
Throughout this proof we consider the model $\Fun^{\fp}(\cP\cM,\vec) \cong \Fun^{\fp}_{\cC}(\cM,\cC)$ given by \cref{thm:linearisingCfunctor}, so that $\iHom_{\cM}(M,-)$ is identified with 
\[
\Hom_{\cM}(M,-) \coloneqq \Hom_{\cM}(M,-)|_{\cP\cM} \cong \Hom_\cC(\unit, \iHom_\cM(M,-)|_{\cP\cM})
\]
under the equivalence, where we drop $|_{\cP\cM}$ from our notation for simplicity.
It is clear that the functors $\Hom_{\cM}(M,-) \ract P \cong \Hom_\cM({\rightdual{P}}\lact M,-)$ are projective for $M\in\cM$ and $P \in \proj(\cC)$ (one may use \cref{lem:projFunfp}, or just apply the classical Yoneda lemma), so the functors $\Hom_{\cM}(M,-)$ are $\cC$-projective in $\Fun^{\fp}(\cP\cM,\vec)$.

    We now focus on the other direction.
    By \cref{defn:FunfpPM} and \cref{cor:intHom-Yon-Emb}, a finitely presented object $F$ in $\Fun^{\fp}(\cP\cM,\vec)$ comes with an exact sequence
    \[\Hom_{\cM}(M,-)\xrightarrow{\Hom_\cM(f,-)}\Hom_{\cM}(N,-)\to F\to 0,\]
    induced by some morphism $f:N\to M$ in $\cP\cM$. Let $K$ denote the kernel of $f$ in $\cM$. Since $F$ is the cokernel of $\Hom_{\cM}(f,-)$, we obtain a natural transformation $\eta$ making the following diagram commutative
	\begin{equation}\label{eqn:inducednat}
	\begin{tikzcd}[column sep=large]
		\Hom_{\cM}(M,-)\ar[r,"\protect{\Hom_\cM(f,-)}"]\ar[d,"="] & \Hom_{\cM}(N,-)\ar[r]\ar[d,"="] & F\ar[r]\ar[d,dashed,"\eta"] & 0 \\
		\Hom_{\cM}(M,-)\ar[r,"\protect{\Hom_\cM(f,-)}"] & \Hom_{\cM}(N,-)\ar[r] & \Hom_{\cM}(K,-).
	\end{tikzcd}
	\end{equation}

    Now assume that $F$ is $\cC$-projective, take a non-zero $P \in \proj(\cC)$ and set $Q \coloneqq \rightdual{P}$. Then the presentation
    \[\Hom_{\cM}(Q\lact M,-)\to\Hom_{\cM}(Q\lact N,-)\to F\ract P\to 0\]
    is now split exact, since every term is projective. This means in particular that 
    \[0\to Q\lact K\to Q\lact N\xrightarrow{Q\lact f} Q\lact M\]
    is a split exact sequence, and so also
     \[\Hom_{\cM}(Q\lact M,-)\to\Hom_{\cM}(Q\lact N,-)\to \Hom_{\cM}(Q\lact K,-)\to 0\]
     is (split) exact. 
	%%%%%%% Expanded explanation 
	%Indeed, note that $X \xrightarrow{\tilde{f}} Y \xrightarrow{\tilde{p}} Z \to 0$ being split exact implies the existence of $\tilde{g}: Y \to X$ such that $\tilde{f}\tilde{g}\tilde{f} = \tilde{f}$ and $\tilde{g}\tilde{f}\tilde{g}=\tilde{g}$. By Yoneda we also get that  $0 \to K \to M \xrightarrow{f} N$ is split exact, because there exists $g:N \to M$ such that $fgf=f$ and $gfg=g$.
	But this means that the action $-\ract P$ turns \eqref{eqn:inducednat} into a morphism of exact sequences:
	\[
	\begin{tikzcd}
		\Hom_{\cM}(Q\lact M,-)\ar[r,]\ar[d,"="] & \Hom_{\cM}(Q\lact N,-)\ar[r]\ar[d,"="] & F \ract P \ar[r]\ar[d,dashed,"\eta \ract P"] & 0 \\
		\Hom_{\cM}(Q\lact M,-)\ar[r] & \Hom_{\cM}(Q\lact N,-)\ar[r] & \Hom_{\cM}(Q\lact K,-) \ar[r] & 0.
	\end{tikzcd}
	\]
It follows that $\eta \ract P$ is an isomorphism, and so $\eta$ must have been an isomorphism to begin with since $- \ract P$ reflects isomorphisms.
\end{proof}

\begin{theorem}\label{thm:functor-cat-Cprojdim}
The $\cC$-projective dimension of the right (resp.\ left) $\cC$-module category $\Fun^{\fp}_{\cC}(\cM,\cC)$ $\cong \Fun^{\fp}(\cP\cM,\vec)$ (resp.\ $\Fun^{\fp}_{\cC}(\rightop{\cM},\cC)$ $\cong \Fun^{\fp}(\cP\cM^{\op},\vec)$) is at most $2$.
\end{theorem}
\begin{proof}
    By \cref{lem:FPiffFPfunctor}, $F \in \Fun^{\fp}_{\cC}(\cM,\cC)$ fits into the following exact sequence:
    \begin{equation}\label{eqn:projpresentation}
        \iHom_\cM(N,-) \xra{\iHom_\cM(f,-)} \iHom_\cM(M,-) \to F \to 0
    \end{equation}
    induced by $f:M\to N$ for $M,N \in \cP\cM$.
    Let $L \coloneqq \coker(f) \in \cM$. Left exactness of $\iHom_\cM(-,X)$ gives the exact sequence 
    \[
    0 \to \iHom_\cM(L,-) \to \iHom_\cM(N,-) \xra{\iHom_\cM(f,-)} \iHom_\cM(M,-),
    \]
    which combines with \eqref{eqn:projpresentation} to a $\cC$-projective resolution of $F$ of length $2$ ($\iHom_\cM(L,-)$ is a $\cC$-projective by \cref{prop:C-projrepresentable}).
\end{proof}

\begin{remark}
    Combining with \cref{cor:Cproj=Gproj}, we also have that the Gorenstein projectives of $\Fun^{\fp}_{\cC}(\cM,\cC)$ are precisely the functors $\iHom_{\cM}(M,-)$ for $M\in \cM$, and that the Gorenstein projective dimension of $\Fun^{\fp}_{\cC}(\cM,\cC)$ is at most 2.
    The same can be said for $\Fun_\cC^{\fp}(\rightop{\cM},\cC)$.
\end{remark}

\subsection{Almost split sequences and \texorpdfstring{$\cP\cM$}{PM}} \label{sec:PM-AR-translate}
\begin{lemma} \label{lem:ases-resolution}
    Let $\cM$ be a finite left $\cC$-module category.
    Suppose $0\to M\xra{f} E\xrightarrow{g} N\to 0$ is an almost split sequence in $\cM$.
    \begin{enumerate}
        \item \label{item:simplefunctor} Consider the corresponding exact sequence in $\Fun^{\fp}_{\cC}(\cM,\cC)$:
        \begin{equation}\label{eqn:simpleprojres}
            0\to\iHom_{\cM}(N,-)\to\iHom_{\cM}(E,-)\xra{\iHom_\cM(f,-)} \iHom_{\cM}(M,-)\to S_M\to0
        \end{equation}
        with $S_M$ defined as the cokernel of $\iHom_\cM(f,-)$.
        If $M \in \cP\cM$, then $S_M$ is a simple object in $\Fun^{\fp}_{\cC}(\cM,\cC)$; otherwise $S_M = 0$.
        \item Consider the corresponding exact sequence in $\Fun^{\fp}_{\cC}(\rightop{\cM},\cC)$:
        \begin{equation}
            0\to\iHom_{\cM}(-,M)\to\iHom_{\cM}(-,E)\xra{\iHom_\cM(-,g)} \iHom_{\cM}(-,N)\to S_N\to0
        \end{equation}
        with $S_N$ defined as the cokernel of $\iHom_\cM(-,g)$.
        If $N \in \cP\cM$, then $S_N$ is a simple object in $\Fun^{\fp}_{\cC}(\rightop{\cM},\cC)$; otherwise $S_N = 0$.
    \end{enumerate}
\end{lemma}
\begin{proof}
    We only prove \ref{item:simplefunctor}; the other follows similarly. 
    Note that the exactness of \eqref{eqn:simpleprojres} follows from left exactness of the internal hom functor and the fact that $S_M$ is defined as the cokernel.
    With $S'_M$ denoting the image of $S_M$ under the equivalence with $\Fun^{\fp}(\cP\cM,\vec)$ in \cref{thm:linearisingCfunctor}, this exact sequence turns into the following exact sequence for every $X\in\cP\cM$:
    \[0\to \Hom_\cM(N,X)\to \Hom_\cM(E,X)\to\Hom_\cM(M,X)\to S'_M(X)\to 0.\]
    Since we started from an almost split sequence, $M$ is indecomposable and we have $S'_M(X)=0$ for all indecomposable $X\not\cong M$ and otherwise $S'_M(M)=\Bbbk$.
    Hence $S'_M \in \Fun^{\fp}(\cP\cM,\vec)$ is a simple object if $M \in \cP\cM$, and otherwise $S'_M=0$.
\end{proof}

\begin{theorem} \label{thm:PM-ARtrans}
    Let $\cM$ be a finite left $\cC$-module category, and let $0\to M \xrightarrow{f} E \xrightarrow{g} N \to 0$ be an almost split sequence in $\cM$.
    Then $M \in \cP\cM$ if and only if $N \in \cP\cM$. 
    In particular, $\cP\cM$ is closed under Auslander--Reiten translations $\tau^\pm$.
\end{theorem}
\begin{proof}
    We will only prove that $M \in \cP\cM$ implies $N \in \cP\cM$; the other implication follows by working with $\Fun^{\fp}(\rightop{\cM},\cC)\cong \Fun^{\fp}(\cP\cM^{\op},\vec)$ instead.

    Let $M \xrightarrow{f} E$ be left minimal and left almost split, with $M \in \cP\cM$.
    Consider the image factorisation of $\iHom_\cM(f,-): \iHom_\cM(E,-) \to \iHom_\cM(M,-)$ in $\Fun_\cC^{\fp}(\cM,\cC)$:
    \[
    \iHom_\cM(E,-) \xtwoheadrightarrow{\pi_f} \mathcal{R} \hookrightarrow  \iHom_\cM(M,-).
    \]
    Note that by \cref{prop:C-projrepresentable}, $\iHom_\cM(E,-)$ is $\cC$-projective in $\Fun^{\fp}_\cC(\cM,\cC)$.
    We claim that $\iHom_\cM(E,-) \xtwoheadrightarrow{\pi_f} \mathcal{R}$ is in fact a right minimal, right $\cC$-projective approximation of $\mathcal{R}$ in $\Fun_\cC^{\fp}(\cM,\cC)$.
    We begin by showing that for any $E' \in \cM$, any $\varphi: \iHom_\cM(E',-) \to \mathcal{R}$ factors through $\pi_f$:
    \[
    \begin{tikzcd}[column sep = large]
        & \iHom_\cM(E',-) \ar[d, "\varphi"] \ar[dl,dashed, swap, "\exists"] \\
        \iHom_\cM(E,-) \ar[r, two heads, "\pi_f"] & \mathcal{R} \ar[r, hook] & \iHom_\cM(M,-).
    \end{tikzcd}
    \]
    By the (contravariant) internal Hom Yoneda embedding in \cref{cor:intHom-Yon-Emb}, this is equivalent to the existence of the following dashed morphism:
    \[
    \begin{tikzcd}[column sep = large, arrows=<-]
        & E' \ar[dr, "y"] \ar[dl,dashed, swap, "\exists"] \\
        E \ar[rr, "f"] & & M,
    \end{tikzcd}
    \]
    where $y$ is such that $\iHom_\cM(y,-)$ is the composition $\iHom_\cM(E',-) \xrightarrow{\varphi}\mathcal{R}\hookrightarrow \iHom_\cM(M,-)$.
    Since $f$ is left almost split, we just have to show that $y:M \to E'$ is not a split monomorphism.
    Since $M$ is indecomposable (implied by almost split sequence), it suffices to show that $\Hom_\cM(y,M): \Hom_\cM(E',M) \to \Hom_\cM(M,M)$ has image contained in $\rad_\cM(M,M)$ (see e.g.\ \cite[IV Lemma 1.9]{ASS_elerep}).
    But note that under the identification $\Fun^{\fp}_\cC(\cM,\cC) \xrightarrow{\cong} \Fun^{\fp}(\cP\cM,\vec)$ from \cref{thm:linearisingCfunctor}, $\iHom_\cM(y,-)$ is sent to $\Hom_\cM(y,-)|_{\cP\cM}$ and $\mathcal{R}$ is sent to $\rad_\cM(M,-)|_{\cP\cM}$.
    Since $M \in \cP\cM$ by assumption, we can evaluate at $M$ to get $\img(\Hom_\cM(y,M)) \subseteq \rad_\cM(M,M)$ as required.
    From here, right minimality of $\pi_f:\iHom_\cM(E,-) \to \mathcal{R}$ is simply a consequence of left minimality of $f$ and \cref{cor:intHom-Yon-Emb}.

    Since $M \in \cP\cM$ by assumption, recall from \cref{lem:ases-resolution} and \cref{prop:C-projrepresentable} that we have the following $\cC$-projective resolution of $S_M \neq 0$ in $\Fun^{\fp}_\cC(\cM,\cC)$:
    \[
    0 \to \iHom_\cM(N,-) \to \iHom_\cM(E,-) \to \iHom_\cM(M,-) \to S_M \to 0.
    \]
    In particular, the following sequence is exact
    \[
    0 \to \iHom_\cM(N,-) \xrightarrow{\iHom_\cM(g,-)} \iHom_\cM(E,-) \xrightarrow{\pi_f} \mathcal{R} \to 0,
    \]
    and $\mathcal{R}$ has $\cC$-projective dimension $\leq 1$ in $\Fun^{\fp}_\cC(\cM,\cC)$. 
    %We claim that the $\pi_f$ being a right minimal, right $\cC$-approximation implies that $\iHom_\cM(N,-)$ is in fact a projective in $\Fun^{\fp}_\cC(\cM,\cC)$, which implies that $N \in \cP\cM$.

    Consider the objects $Y_\mathcal{R}, E_\mathcal{R} \in \cM$ realising the following exact sequence from \cref{cor:approx} applied to $\mathcal{R}\in\Fun^{\fp}_\cC(\cM,\cC)$:
    \[
    0\to \iHom_\cM(Y_\mathcal{R},-) \to \iHom_\cM(E_\mathcal{R},-) \to \mathcal{R} \to 0,
    \]
    so that $\iHom_\cM(E_\mathcal{R},-) \to \mathcal{R}$ is a right $\cC$-projective-approximation of $\mathcal{R}$ and $\iHom_\cM(Y_\mathcal{R},-)$ is projective in $\Fun^{\fp}(\cP\cM,\vec)$.
    In particular,  $Y_\mathcal{R} \in \cP\cM$ by \cref{lem:projFunfp}.
    
    Using the fact that $\iHom_\cM(E_\mathcal{R},-) \to \mathcal{R}$ and $\iHom_\cM(E,-) \xrightarrow{\pi_f} \mathcal{R}$ are both right $\cC$-projective-approximations, we obtain the following maps of exact sequences:
    \[
    \begin{tikzcd}[column sep=small]
        0 \ar[r] & \iHom_\cM(N,-) \ar[r] \ar[d, dashed] & \iHom_\cM(E,-) \ar[d, dashed] \ar[r, "\pi_f"] & \mathcal{R} \ar[r] \ar[d,"="] & 0 \\
        0 \ar[r] & \iHom_\cM(Y_\mathcal{R},-) \ar[r] \ar[d, dashed] & \iHom_\cM(E_\mathcal{R},-) \ar[d, dashed] \ar[r] & \mathcal{R} \ar[r] \ar[d,"="] & 0 \\
        0 \ar[r] & \iHom_\cM(N,-) \ar[r] & \iHom_\cM(E,-) \ar[r, "\pi_f"] & \mathcal{R} \ar[r] & 0.
    \end{tikzcd}
    \]
    By right minimality of $\iHom_\cM(E,-) \xrightarrow{\pi_f} \mathcal{R}$, the composition of the middle column results in an isomorphism.
    It follows that the composition of the first column is also an isomorphism.
    This, together with the internal Hom Yoneda embedding from \cref{cor:intHom-Yon-Emb}, implies that $N$ is a summand of $Y_\mathcal{R} \in \cP\cM$, and hence $N$ is also in $\cP\cM$, as required.
\end{proof}

Finally, we conclude that $\cP\cM$ satisfies what is known as a (1-)\emph{precluster tilting subcategory} of $\cM$ introduced by Iyama--Solberg in \cite{IS_preclustertilting}:
\begin{definition}
    Let $\cM$ be a finite abelian category.
    A full subcategory $\cA\subseteq \cM$ is called a (1-)\emph{precluster tilting subcategory} if it satisfies the following properties:
       \begin{enumerate}
        \item $\cA$ is a generator and cogenerator of $\cM$.
        \item $\cA$ is closed under Auslander--Reiten translates.
        \item $\cA \subseteq \cM$ is functorially finite.
    \end{enumerate}
\end{definition}
\begin{corollary}\label{cor:PM-precluster-tilting}
     Let $\cM$ be a finite $\cC$-module category. Then $\cP\cM$ is a precluster tilting subcategory of $\cM$.
\end{corollary}
\begin{proof}
    Since $\cP\cM$ contains both the projectives and injectives of $\cM$, it is therefore a generator and cogenerator of $\cM$.
    Closure under the Auslander--Reiten translate is shown in \cref{thm:PM-ARtrans}.
    Functorial finiteness is shown in \cref{lem:functoriallyfinite}.
\end{proof}

\subsection{An equivariant Auslander(--Solberg) correspondence} \label{sec:equiv-Aus-cor}
In this subsection, we prove a $\cC$-equivariant version of the Auslander--Solberg correspondence for finite $\cC$-module categories.
Let us first recall some relevant definitions.
\begin{definition}
Let $\cN$ be a finite abelian category and let $N \in \cN$.
The \emph{dominant dimension} $\dom.\dim(N)$ of $N$ is defined as the supremum over all integer $d$ such that in the minimal injective coresolution of $N$
\[
0 \to N \to I_0 \to I_1 \to \cdots \to I_{d-1} \to I_d \to \cdots,
\]
we have $I_j \in \proj(\cN)\cap\inj(\cN)$ for all $j < d$.
The \emph{dominant dimension of $\cN$} is then defined as\footnote{Let $A$ be the finite dimensional algebra realising $\cN$, so that $\cN \cong \mod_A$. Then $\dom.\dim(\cN)$ is more commonly defined as $\dom.\dim(A)$, which agrees with our more categorical definition.}
\[
\dom.\dim(\cN) \coloneqq \min \{\dom.\dim(P) \mid P \in \proj(\cN)\}.
\]
\end{definition}
Note that for $M$ a projective-injective, $\dom.\dim(M) = \infty$, so the dominant dimension is an inverse measure of how far a module is from being a projective-injective: the smaller number, the further it is.
In particular, $\dom.\dim(\cN) = \infty$ if $\cN$ is a quasi-Frobenius category.

We have the following correspondence due to Auslander--Solberg in \cite{AS_correspondence}; we shall state it using the modern reformulation of Iyama--Solberg in \cite[Theorem 4.5]{IS_preclustertilting}:
\begin{theorem}[\cite{AS_correspondence}]\label{thm:AS-correspondence}
There is a one-to-one correspondence between the following two classes of finite abelian categories (up to equivalences):
    \begin{enumerate}
        \item \label{item:precluster} the class of finite abelian categories $\cA$ with a choice of a (1-)precluster tilting subcategory $\cX\subseteq \cA$ that has an additive generator $X \in \cX$, i.e.\ $\cX = \add(X)$; and
        \item \label{item:min_Aus-Gor}the class of finite abelian categories $\cB$ which is (1-)minimal Auslander--Gorenstein, i.e.\ $\Gpd(\cB) \leq 2 \leq \dom.\dim(\cB)$. 
    \end{enumerate}
    The bijection from \ref{item:precluster} to \ref{item:min_Aus-Gor} is given by $(\cA,\cX) \mapsto \cB \coloneqq \Fun^{\fp}(\cX^{\op},\vec)$, with inverse from \ref{item:min_Aus-Gor} to \ref{item:precluster} given by $\cB \mapsto (\cA\coloneqq \Gproj(\cB), \cX\coloneqq \proj(\cB))$.
\end{theorem}
We note that the correspondence above recovers the classical Auslander correspondence by insisting that $\cX=\cA$.

Here we prove instead a $\cC$-equivariant Auslander correspondence that recovers the classical Auslander correspondence by setting $\cC=\vec$; the relationship between all three correspondences is discussed in \cref{rmk:correspondences-relations}.
\begin{theorem}[$\cC$-equivariant Auslander correspondence] \label{thm:equivariant-Aus-cor}
    There is a one-to-one correspondence between the following two classes of finite $\cC$-module categories (up to $\cC$-module equivalences):
    \begin{enumerate}
        \item \label{item:C-finite} the class of finite $\cC$-module categories $\cM$ that are $\cC$-rep-finite; and
        \item \label{item:C-functor-cat} the class of finite $\cC$-module categories $\cN$ that satisfy $\ipd(\cN) \leq 2 \leq \dom.\dim(\cN)$.
    \end{enumerate}
    The bijection from \ref{item:C-finite} to \ref{item:C-functor-cat} is given by $\cM \mapsto \cN \coloneqq \Fun^{\fp}_\cC(\rightop{\cM},\cC)$, with inverse from \ref{item:C-functor-cat} to \ref{item:C-finite} given by $\cN \mapsto \cM\coloneqq \Cproj{}(\cN)$.
\end{theorem}
\begin{proof}
    We begin by showing that the proposed bijections are well-defined maps between the classes of finite $\cC$-module categories.
    Let $\cM$ be a finite $\cC$-module category that is moreover $\cC$-rep-finite.
    Then $\cN\coloneqq \Fun^{\fp}_\cC(\rightop{\cM},\cC) \cong \Fun^{\fp}(\cP\cM^{\op},\vec)$ is indeed a finite abelian category (since $\cP\cM$ has an additive generator by definition of $\cC$-rep-finite).
    It follows from \cref{thm:functor-cat-Cprojdim} that $\ipd(\cN) \leq 2$.
    Moreover, one can show directly that $\dom.\dim(\cN) \geq 2$; this also follows from the well-definedness of the bijection in \cref{thm:AS-correspondence}, noting that $\cP\cM$ is precluster-tilting by \cref{cor:PM-precluster-tilting}.

    Now let $\cN$ be a finite $\cC$-module category satisfying $\ipd(\cN) \leq 2 \leq \dom.\dim(\cN)$.
    By \cref{cor:Cproj=Gproj}, we get that $\Gpd(\cN)=\ipd(\cN) \leq 2$ and $\Cproj{}(\cN) = \Gproj(\cN)$. 
    Thus $\cN$ is minimal Auslander--Gorenstein.
    In particular, we already know from \cref{thm:AS-correspondence} that $\Gproj(\cN)$ is a finite abelian category.
    Precisely, this is because $\Gproj(\cN)$ can be described as follows (see e.g.\ \cite[Lemma 2.1]{Marczinzik_minAusGor}):
    \begin{align*}
    \Gproj(\cN) &= \{ N \in \cN \mid \dom.\dim(N) \geq 2\} \\
        &= \{N \in \cN \mid \exists 0 \to N \to I_0 \to I_1 \text{ exact}: I_0,I_1 \in \proj(\cN)\cap \inj(\cN)\},
    \end{align*}
    which shows that $\Gproj(\cN)$ is a finite abelian category\footnote{Note that the exact sequences in $\Gproj(\cN)$ need not be induced from the exact sequences in $\cN$.} via the following identification (see e.g.\ \cite[II, Proposition 2.5]{ARS_repArtin}):
    \[
    \begin{split}
        \Gproj(\cN) &\xrightarrow{\cong} \cF \coloneqq \Fun^{\fp}(\proj(\cN)\cap\inj(\cN),\vec)^{\op} \\
        N &\mapsto \Hom_\cN(N,-)|_{\proj(\cN)\cap\inj(\cN)}.
    \end{split}
    \]
    Let us show that $\Cproj{}(\cN) = \Gproj(\cN)$ is moreover a $\cC$-module category via the $\cC$-action restricted from $\cN$.
    This will be achieved by showing that it defines a $\cC$-module category structure on $\cF$ through the equivalence above, which is defined by
    \[
    (X \lact F)(-) = F(\rightdual{X}\lact -), \qquad X \in \cC, \quad F\in \cF.
    \]
    (Note that $\rightdual{X}\lact I \in \proj(\cN)\cap\inj(\cN)$ since $\rightdual{X}\lact -$ preserves both projectives and injectives.)
    To this end, let $0\to X_1 \to X_2 \to X_3 \to 0$ be an exact sequence in $\cC$.
    Then the following sequence is exact in $\cF$ for all $F \in \cF$:
    \[
    0 \to F(\rightdual{X_1}\lact -) \to F(\rightdual{X_2}\lact -) \to F(\rightdual{X_3}\lact -) \to 0
    \]
    since for all $I \in \proj(\cN)\cap\inj(\cN)$, $-\lact I$ is exact and $0\to \rightdual{X_3}\lact I \to \rightdual{X_2}\lact I \to \rightdual{X_1}\lact I \to 0$ splits.
    (Exactness of $X\lact - $ on $\cF$ is immediate.)
    Finally, it follows that $\Cproj{}(\cN)=\Gproj(\cN) \cong \cF$ is $\cC$-rep-finite, since $\cP(\Cproj{}(\cN)) = \proj(\cN)$, which must have an additive generator ($\cN$, being a finite abelian category, has only finitely many indecomposable projectives).
    This concludes the proof that the maps are well-defined.

    We now prove that the proposed maps define bijections between the two classes.
    Starting with a $\cC$-rep-finite module category $\cM$, consider the composition
    \[
    \cM \mapsto\Fun^{\fp}_\cC(\rightop{\cM},\cC) \mapsto \Cproj{}(\Fun^{\fp}_\cC(\rightop{\cM},\cC)).
    \]
    By \cref{prop:C-projrepresentable}, the $\iHom_\cM(-,M) \in \Fun^{\fp}_\cC(\rightop{\cM},\cC)$ are exactly the $\cC$-projectives in $\Fun^{\fp}_\cC(\rightop{\cM},\cC)$, and so \cref{cor:intHom-Yon-Emb} gives us an equivalence of abelian categories $\cM \cong \Cproj{}(\Fun^{\fp}_\cC(\rightop{\cM},\cC))$.
    This equivalence is moreover a $\cC$-module equivalence since (see also \cref{rmk:C-equivariant-int-Hom-Yon-embedding})
    \[
    \iHom_\cM(-,X\lact M) \cong X\otimes \iHom_\cM(-,M) = X\lact \iHom_\cM(-,M).
    \]
    On the other hand, suppose $\cN$ satisfies $\ipd(\cN) \leq 2 \leq \dom.\dim(\cN)$ and consider the composition
    \[
    \cN \mapsto \cM\coloneqq \Cproj{}(\cN) \mapsto \Fun^{\fp}_\cC(\rightop{\cM},\cC).
    \]
    Note that it follows by definition of $\cM$ that $\cP\cM = \proj(\cN)$.
    Thus we get that
    \[
    \Fun^{\fp}_\cC(\rightop{\cM},\cC) \cong \Fun^{\fp}(\cP\cM^{\op},\vec) = \Fun^{\fp}(\proj(\cN)^{\op},\vec) \cong \cN;
    \]
    the first $\cC$-module equivalence comes from \cref{thm:linearisingCfunctor} and the final equivalence $\cN \cong \Fun^{\fp}(\proj(\cN)^{\op},\vec)$ as abelian categories is classical (see \cite[Lemma 2.7]{Krause_defcat} or \cite[Lemma 18.1.13]{Prest_Purity-Spectra-local}), sending $N \in \cN \mapsto \Hom_\cN(-,N)|_{\proj(\cN)}$.
    This equivalence does indeed respect the $\cC$-module structures, since
    \[
    \Hom_\cN(-,X\lact N)|_{\proj(\cN)} \cong \Hom_\cN(\leftdual{X}\lact -,N)|_{\proj(\cN)} = X\lact \Hom_\cN(-,N)|_{\proj(\cN)}
    \]
    (see \eqref{eqn:rightopPMaction} for last equality).
    This concludes our proof.
\end{proof}
\begin{remark}[Relations between 3 correspondences]\label{rmk:correspondences-relations} 
While \cref{thm:AS-correspondence} and \cref{thm:equivariant-Aus-cor} both generalise the classical Auslander correspondence, and are given by essentially the same bijections, we note that neither of them is a direct consequence of the other as the class of ``categories with structures'' considered are different.
In particular, we can have two \emph{distinct} $\cC$-module categories $\cM$ and $\cM'$ in \ref{item:C-finite} of \cref{thm:equivariant-Aus-cor}, which nonetheless have the \emph{same underlying finite abelian category} and have $\cP\cM=\cP\cM'$.
In other words, the only difference is in the $\cC$-module category structure.
(One obtains many such examples from semisimple $\cC$.)
In light of this, it would be interesting to find higher analogues of \cref{thm:equivariant-Aus-cor}, as was done in \cite{Iyama_higher-Aus-cor} and \cite{IS_preclustertilting} for the Auslander correspondence and the Auslander--Solberg correspondence respectively.
\end{remark}

\begin{remark}
    We emphasise that a finite abelian category $\cM$ can be $\cC$-rep-finite with respect to different $\cC$-module category structures (even over the same $\cC$).
    For example, in \cref{eg:introA2} we considered $\cC=\mod_{T_2(-1)}$ and studied $\cM=\mod_{\hat{A}}$ equipped with a non-trivial $\cC$-module structure such that $\cP\cM \subsetneq \cM$.
    But $\mod_{\hat{A}}$ can also be equipped with the ``trivial'' $\cC$-module structure obtained via the forgetful functor $\mod_H \to \vec$ (so that $X \lact M \cong M^{\dim_\Bbbk(X)}$), which is still $\cC$-rep-finite since $\mod_{\hat{A}}$ is itself representation-finite.
    The corresponding $\cC$-module categories in \ref{item:C-functor-cat} of \cref{thm:equivariant-Aus-cor} that one obtains are therefore also different.
    In this case even their underlying abelian categories are different: the trivial one simply gives $\mod_{\mathrm{Aus}(\hat{A})^{\op}}$ for $\mathrm{Aus}(\hat{A})$ the (classical) Auslander algebra of $\hat{A}$, whereas the non-trivial one gives $\mod_{\End_\cM(\overline{M})^{\op}}$ with $\overline{M}$ the direct sum of the 3 indecomposables in $\cP\cM$ (\cref{fig:introA2}).
    %(One can show that $\End_\cM(\overline{M}) \cong B_2\otimes_\Bbbk \mathrm{Aus}(A_2)$ is a $2$-Iwanaga--Gorenstein algebra with infinite global projective dimension.)
\end{remark}

\section{\texorpdfstring{$\cC$}{C}-rep-finite module categories from tensor products}\label{sec:example-tensor-alg}
Throughout this section, $\cC$ is a finite tensor category and all $\cC$-module categories will be finite.
\subsection{A principle}
\label{sec:MNA}
Consider a finite tensor category $\cC$ and a finite $\cC$-module category $\cN$. Let $A$ be a finite dimensional $\Bbbk$-algebra. There is a canonical $\cC$-module structure on the Deligne product
\[\cM:=\cN\boxtimes\mod_A.\]
Objects in $\cM$ are given by pairs of an object $Y\in\cN$ and an algebra morphism $A\to\End_{\cN}(Y)$. The action of $X\in\cC$ on such an object results in the object corresponding to $X\lact Y\in\cN$ together with the composite algebra morphism
\begin{equation}\label{actNA}A\to\End_{\cN}(Y)\xrightarrow{X\lact-}\End_{\cN}(X\lact Y).\end{equation}
Note that if we identify $\cN \cong \mod_B$, then $\cM \cong \mod_{B\otimes_{\Bbbk} A}$ (as abelian categories).

We denote by $\proj(\cN)\stackrel{\bullet}{\otimes} \mod_A$ the additive envelope of the ordinary tensor product of $\Bbbk$-linear categories $\proj(\cN){\otimes} \mod_A$ of $\proj(\cN)$ and $\mod_A$. We have an obvious fully faithful $\Bbbk$-linear functor
\begin{equation}\label{eq:PMtensor}\proj(\cN)\otimes \mod_A\to \cP\cM,\end{equation}
that sends $(P,V)$, for indecomposable $P\in \proj(\cN)$ and $V \in \mod_A$ to the object 
\[A\to \End_{\Bbbk}(V)\to \End_{\cN}(P\otimes_{\Bbbk}V).\]
Here the first morphism is action of $A$ on $V$ and the second arrow is given, after taking a basis $e_1,\ldots, e_n$ of $V$, by $\Mat_n(\Bbbk)\to \Mat_n(\End_{\cN}(P))$ which sends every matrix entry $a\in\Bbbk$ to $a\cdot \id_P$.

We call a finite abelian category {\bf uniserial} if every indecomposable module has a unique composition series. Note that a uniserial finite abelian category necessarily has finitely many indecomposable objects (up to isomorphism).

\begin{theorem}\label{thm:uniserial}
     Assume that $\cN$ is exact and uniserial. Then the functor
     \[\proj(\cN)\stackrel{\bullet}{\otimes} \mod_A\;\to\;\cP\cM\]
is an equivalence of categories. In particular, $\cM = \cN\boxtimes \mod_A$ is $\cC$-rep-finite if and only if $A$ is representation finite.
\end{theorem}
\begin{proof}
By assumption, $\cN$ has finitely many indecomposable objects, which are all uniserial. Fix a projective $Q\in\cC$. For every simple object $S_i\in\cN$, we have a projective object $P_i:=Q\lact S_i\in\cN$. 

Now we will choose `fixed decompositions' of objects $Q\lact N\in\proj(\cN)$, for all indecomposable $N\in\cN$. For this we apply induction on the length (=Loewy length) of $N$. For the simple objects, we simply write $Q\lact S_i=P_i$ (even though this need not be indecomposable). For any indecomposable object $N$ of length two, we choose a short exact sequence
\[0\to S_i\to N\to S_j\to 0\]
and then a splitting of the short exact sequence, obtained via $Q\lact-$,
\[0\to P_i\to Q\lact N\to P_j\to 0,\]
giving our desired decomposition $Q\lact N=P_i\oplus P_j$. Note that for the splitting, if we make an arbitrary choice, the freedom in picking an alternative decomposition corresponds to picking a morphism $P_j\to P_i$. Now for a length three indecomposable $N\in\cN$, we choose short exact sequences
\[0\to S_a \to N\to N'\to 0\quad\mbox{and}\quad 0\to N''\to N\to S_c\to 0. \]
Now we choose a decomposition $Q\lact N=P_a\oplus P_b\oplus P_c$, where $S_b$ is the third simple constituent of $N$, that satisfies two conditions. Firstly, the inclusion $Q\lact N''\hookrightarrow Q\lact N$ identifies the chosen decomposition of $N''$ from the previous step with $P_a\oplus P_b$, and secondly the epimorphism $Q\lact N\twoheadrightarrow Q\lact N'$ corresponds to the projection onto $P_b\oplus P_c$. This can be done: assuming the first condition, the choice of the third summand $P_c$ corresponds to picking morphisms $P_c\to P_a$ and $P_c\to P_b$. The first one can be chosen arbitrarily, while the second one must and can be chosen to comply with our choice for $N'$.

This procedure can be repeated for higher lengths, by choosing short exact sequences corresponding to the radical and socle.

The end result is that, for an arbitrary $N\in\cN$, we have a choice of decomposition of $Q\lact N$ into (not necessarily indecomposable) projectives in $\cN$, such that for any morphism $N_1\to N_2$ (which we can view as a linear combination of compositions of the elementary projections and inclusions that we used in our construction of the decompositions), the induced morphism $Q\lact N_1\to Q\lact N_2$ takes the form of a matrix with entries that are scalar multiples of the identity morphisms of the projective objects $P_a$ into which we decompose. It now follows from the construction in \eqref{actNA}, that the action of $Q\lact-$ on an object $A\to\End_{\cN}(N)$ gives an object in the essential image of the functor in the theorem. In other words, every object of the form $Q\lact M$ lies in the essential image of the fully faithful functor
\[\proj(\cN)\otimes \mod_A\;\to\;\cP\cM,\]
so that taking the idempotent completion on both sides yields an equivalence.
\end{proof}

\begin{remark}
    Since $\cN$ in Theorem~\ref{thm:uniserial} is exact (as a module category) and uniserial, as an abelian category it must be of the form $\mod_B$ for a self-injective Nakayama algebra $B$. Examples of exact module categories with description of the corresponding self-injective Nakayama algebras are given in \cref{eg:exactmod-Taft}. Of course, most exact module categories will not be uniserial; for example the category of modules over small quantum groups (which is exact over itself) is not uniserial.
\end{remark}

%%%%%%%%%%%%%%%%%%%%%%%%%%%%%%%%%%%%%%%%%%%%%%%%%%%%%%%%
% Old corollary about tensor product algebras
%%%%%%%%%%%%%%%%%%%%%%%%%%%%%%%%%%%%%%%%%%%%%%%%%%%%%%%%
\iffalse
\begin{corollary}\label{cor:tensor-Nakayama-rep-fin}
    Let $A$ be a representation-finite algebra and $B$ a self-injective Nakayama algebra.
    Suppose further that $\mod_B$ is an exact module category over some finite tensor category $\cC$.
    Then $\mod_{B \otimes_\Bbbk A}$ has finite finitistic dimension.
\end{corollary}
\begin{proof}
    Set $\cN \coloneqq \mod_B$ and suppose $\cN$ is an exact module category over $\cC$. 
    By \cref{thm:uniserial}, we get that $\cM \coloneqq  \cN\boxtimes\mod_A \cong \mod_{B\otimes_\Bbbk A}$ is $\cC$-module category which is $\cC$-rep-finite.
    It follows from \cref{thm:PM-computes-findim} that $\cM$ has finite finitistic dimension.
\end{proof}
\fi

\begin{example}\label{eg:tensorproduct}
    Suppose $\mathrm{char}(\Bbbk)=0$. 
    Let $\wt{A}^c_{\ell-1}$ be the cyclic quiver with $\ell$ vertices and $R$ the arrow ideal.
    It is known that $Z(\ell,k)\coloneqq \Bbbk\wt{A}^c_{\ell-1}/R^k$ is a self-injective Nakayama algebra for all positive integers $\ell$ and $k\geq 2$, and all connected, non-semisimple self-injective Nakayama algebras are of this form (up to Morita equivalence) \cite[Proposition 3.8]{ASS_elerep}.
    When $\ell$ is a divisor of $k$, $Z(\ell,k)$ is the algebra $B_k(\ell)$ defined in \cref{eg:exactmod-Taft}, where we saw that $\mod_{B_k(\ell)}$ is an exact module category over $\cC\coloneqq \mod_{T_k(\zeta)}$ for $T_k(\zeta)$ the Taft algebra from \cref{eg:Taft}.
    On the other hand, when $k$ is a divisor of $\ell$, then $Z(\ell,k)$ can be equipped with the structure of a Hopf algebra; in fact, these are exactly all of the connected self-injective Nakayama algebras that can be equipped with the structure of a Hopf algebra \cite[Theorem 5.1]{CHYZ_monomial-Hopf}.
    Then for any representation-finite algebra $A$ and $Z(\ell,k)$ such that $\ell \mid k$ or $k \mid \ell$, we get from \cref{thm:uniserial} that $\mod_{Z(\ell,k) \otimes_\Bbbk A}$ is $\cC$-rep-finite.
    Note that $Z(\ell,k)\otimes_\Bbbk A$ can be representation-infinite; this is already the case for $Z(1,2) \otimes_\Bbbk Z(1,2) \cong \Bbbk[x]/\<x^2\>\otimes_\Bbbk\Bbbk[x]/\<x^2\>$.
\end{example}

\begin{remark}
    Let $\cN \cong \mod_B$ and $A$ be as in \cref{thm:uniserial}.
    While the theorem also implies that the tensor product algebra $B \otimes_\Bbbk A$ has finite finitistic dimension when $A$ is representation-finite, it seems to be well-known among the experts that the finitistic dimension is additive under tensor product.
    This is commonly attributed to an unpublished result of Rickard; see also \cite[Theorem 16]{ERZ_dimension-tensor-product}.
    In particular, under our assumption the finitistic dimension of $B \otimes_\Bbbk A$ is simply the finitistic dimension of $A$.
\end{remark}

\begin{remark}
    In light of \cref{eg:tensorproduct}, it is interesting to find out whether the remaining cases of $\mod_{Z(\ell,k)}$ (i.e.\ $\ell \nmid k$ and $k \nmid \ell$) can be equipped with the structure of an \emph{exact} $\cC$-module category for some finite tensor category $\cC$. This would show that the tensor product of any self-injective Nakayama algebra with a representation-finite algebra is $\cC$-rep-finite.
\end{remark}

\subsection{The general case}

We do not know whether Theorem~\ref{thm:uniserial} holds without the uniserial assumption. We briefly discuss what can be said in general.

We call an element $r$ of a ring $R$ (von Neumann) {\bf regular} if there exists $s\in R$  
with $rsr=r$. We will freely use that non-zero regular elements never belong to the Jacobson radical of $R$ (because $rs$ is an idempotent in the notation above).

\begin{remark}
\begin{enumerate}
    \item A ring is von Neumann regular if every element is regular. 
    \item A morphism $f:U\to V$ in an additive category is usually called {\bf split} if it is regular inside $\End(U\oplus V)$. If the category is idempotent closed, this is equivalent to there being decompositions $U=U'\oplus U''$ and $V=V'\oplus V''$ such that $f$ is given by the composite $U\twoheadrightarrow U' \to V'\hookrightarrow V$, for an isomorphism $U'\to V'$.

\end{enumerate}

\end{remark}

Consider again the set-up in Section~\ref{sec:MNA}. That is we construct the $\cC$-module category $\cM\coloneqq \cN\boxtimes\mod_A$ with $\cN$ an exact $\cC$-module category.
\begin{lemma}\label{lem:PMDTP} 
    Every object in $\cP\cM$ is given by an algebra morphism
    \[A\to\End_{\cN}(Q),\]
    where $Q\in \proj(\cN)$ and the morphism takes values in the subset of regular elements of $\End_{\cN}(Q)$.
\end{lemma}
\begin{proof}
    By assumption $\cN$ is exact, so that for any $P\in\proj(\cC)$ and $X\in\cN$, it follows that $P\lact X$ is projective. By considering the four-term exact sequence corresponding to an arbitrary morphism $f$ in $\cN$, it thus follows that $P\lact f$ is a split morphism between projective objects.
    
    Consequently, the action of $P$ on an object $A\to\End_{\cN}(Y)$ in $\cM$ produces an object $A\to\End_{\cN}(P\lact Y)$, where the morphism takes values in regular morphisms.

    Hence every object in $\cP\cM$ is a direct summand of an object represented by a morphism $\rho:A\to\End_\cN(Q)$ taking values in regular elements. For such a direct summand we have decomposition $Q=Q_1\oplus Q_2$ in $\cN$, such that $\rho$ takes values in the subalgebra $\End_\cN(Q_1)\times\End_\cN(Q_2)$. It follows from the definition that an element $(x,y)$ in the latter algebra is regular if and only if it is regular as an element in $\End_\cN(Q)$. This concludes the proof.
\end{proof}

\begin{remark}
    The condition in Lemma~\ref{lem:PMDTP} is very strong, for instance for any element of $A$, not only its image must be regular, but any polynomial in it must be regular. It implies in particular that the kernel of $A\to\End_{\cN}(Q)$ equals the kernel of the composite
    \[A\to\End_{\cN}(Q)\to\End_{\cN}(Q)/J.\]
    If $A$ is semisimple, then the Wedderburn--Malcev Theorem implies that the morphism is conjugate to one that corresponds to an object in $\proj(\cN)\stackrel{\bullet}{\otimes}\mod_A$. However, in general, the fact that an algebra morphism takes values in the set of regular elements does not mean it must be conjugate to a morphism taking values in a chosen (via Wedderburn's Theorem) semisimple complement of the Jacobson radical. This implies that the hypothetical generalisation of Theorem~\ref{thm:uniserial} does not follow from Lemma~\ref{lem:PMDTP}.
\end{remark}

\section{Species and 1-Iwanaga--Gorenstein algebras}\label{sec:species}
Throughout this section, $\cC$ is a finite tensor category and all $\cC$-module categories will be finite.
\subsection{Species and tensor algebras over finite tensor categories}

\begin{definition}
    Let $\cC$ be a finite tensor category. A \emph{$\cC$-species} $(Q, \cM_\bullet, \cF_\bullet)$ consists of the following datum:
    \begin{enumerate}
        \item A finite, acyclic quiver $Q = (Q_0, Q_1)$;
        \item An exact finite $\cC$-module category $\cM_i$ for each vertex $i \in Q_0$;
        \item A $\cC$-module functor\footnote{Since $\cM_i$ are all exact finite $\cC$-module categories, all $\cC$-module functors $\cF_\alpha$ are automatically exact \cite[Proposition 7.6.9 \& 7.9.7]{EGNO} and have (exact) left and right adjoints \cite[Lemma 7.11.4]{EGNO}.} $\cF_\alpha:\cM_{s(\alpha)}\to\cM_{t(\alpha)}$ for each arrow $\alpha:s(\alpha)\to t(\alpha) \in Q_1$.
    \end{enumerate}
    A \emph{representation} $V_\bullet$ of the $\cC$-species $(Q, \cM_\bullet, \cF_\bullet)$ consists of
    \begin{enumerate}
        \item an object $V_i \in \cM_i$ for each vertex $i \in Q_0$; and
        \item a morphism $V_\alpha:\cF_\alpha(V_{s(\alpha)}) \to V_{t(\alpha)}$ in $\cM_{t(\alpha)}$ for each arrow $\alpha \in Q_1$.
    \end{enumerate}
    A \emph{morphism} $f_\bullet:V_\bullet \to W_\bullet$ of representations consists of a morphism $f_i:V_i \to W_i$ for each vertex $i \in Q_0$, such that the following diagram commutes for each arrow $\alpha \in Q_1$:
    \[
    \begin{tikzcd}
        \cF_{\alpha}(V_{s(\alpha)}) \ar[r, "V_{\alpha}"]\ar[d,"\cF_{\alpha}(f_{s(\alpha)})"] & V_{t(\alpha)}\ar[d,"f_{t(\alpha)}"] \\
        \cF_{\alpha}(W_{s(\alpha)}) \ar[r, "W_{\alpha}"] & W_{t(\alpha)}
    \end{tikzcd}
    \]
    The category of representations of the $\cC$-species $(Q, \cM_\bullet, \cF_\bullet)$ is denoted by $\rep(Q, \cM_\bullet, \cF_\bullet)$.
\end{definition}
In the special case where we fix $\cM = \cM_i$ for all $i$ and $\cF_\alpha$ are all identity functors, then $\rep(Q,\cM_\bullet, \cF_\bullet)$ is the same as the category of representations of $Q$ valued in $\cM$ (see also \cref{eg:speciesA2}).

More generally, a $\cC$-species consists of an underlying \textbf{pro-species} as defined in \cite{Kulshammer_prospecies}, equipped with a `$\cC$-equivariant structure'. 
Precisely, each $\cM_i$, being exact, is equivalent to $\mod_{A_i}$ for some self-injective algebra $A_i$ (i.e.\ 0-Iwanaga--Gorenstein).
Each functor $\cF_\alpha$, being a $\cC$-module functor between exact $\cC$-module categories, must have exact left and right adjoints, hence can be realised as a bimodule which is projective from both sides.
We leave it to the reader to check the following, which essentially follows from the respective definitions of representations:
\begin{lemma}
    A representation of a $\cC$-species is the same as a representation of the underlying pro-species; the respective notions of morphisms between representations also agree. 
    In particular $\rep(Q,\cM_\bullet,\cF_\bullet)$ is a finite abelian category, with kernels and cokernels computed vertex-wise.
\end{lemma}
However, the category of representations of a $\cC$-species is naturally equipped with the extra structure of a $\cC$-module category:
\begin{proposition}\label{prop:species-1IG}
    The category $\rep(Q, \cM_\bullet, \cF_\bullet)$ of representations of a $\cC$-species is a finite $\cC$-module category with \[\ipd(\rep(Q, \cM_\bullet, \cF_\bullet)) \leq 1.\]
    In particular, Gorenstein projectives in $\rep(Q, \cM_\bullet, \cF_\bullet)$ are exactly the $\cC$-projectives and $\rep(Q, \cM_\bullet, \cF_\bullet)$ is $1$-Iwanaga--Gorenstein.
\end{proposition}
\begin{proof}
    The $\cC$-module structure is given as follows: for each $X \in \cC$ and $V_\bullet \in \rep(Q, \cM_\bullet, \cF_\bullet)$, we define
    \begin{enumerate}
        \item $(X\lact V_\bullet)_i \coloneqq X\lact (V_i)$; and
        \item $(X\lact V_\bullet)_\alpha \coloneqq \cF_\alpha(X\lact V_{s(\alpha)}) \xrightarrow{\cong} X\lact \cF_\alpha(V_{s(\alpha)}) \xrightarrow{X\lact V_{\alpha}} X\lact V_{t(\alpha)}$,
    \end{enumerate}
    where the first isomorphism is given by the natural isomorphism that is part of the $\cC$-module functor structure of $\cF$.
    For each $f_\bullet: V_\bullet \to W_\bullet$, the morphism $X\lact f_\bullet$ is simply defined vertex-wise: $(X\lact f_\bullet)_i = X\lact (f_i)$, where it follows that this is a morphism of representations using again the natural isomorphism defining the $\cC$-module functor structure of $\cF$ and functoriality of $X\lact -$ on each $\cM_i$.
    It is straightforward to verify that $-\lact V_\bullet$ is indeed an exact functor ($X\lact -$ being exact is immediate due to the existence of left and right duals in $\cC$).

    Let $V \in \rep(Q,\cM_\bullet,\cF_\bullet)$ be an arbitrary representation.
    Now under this $\cC$-module structure, for all $P\neq 0 \in \proj(\cC)$, $W \coloneqq P\lact V$ satisfies $W_i \in \proj(\cM_i)$ for all $i \in Q_0$, due to the exact $\cC$-module category structure of $\cM_i$.
    By \cite[Corollary 3.3]{Kulshammer_prospecies}, $\pd(W) \leq 1$, and it follows from \cref{prop:Cprojdimprojdim} that $\ipd(V) = \pd(W) \leq 1$.
    Since $V$ was arbitrary, we get $\ipd(\rep(Q,\cM_\bullet,\cF_\bullet)) \leq 1$ as required.
    The final statement is simply a consequence of \cref{cor:Cproj=Gproj}.
\end{proof}

\begin{example}\label{eg:speciesA2}
    Let $\cC=\mod_{T_2(-1)}$ be the finite tensor category associated to the Taft algebra from \cref{eg:Taft}, and recall that $\mod_{B_2(1)}$ is an exact $\cC$-module category from \cref{eg:exactmod-Taft} ($B_2(1) \cong  \Bbbk[\varepsilon]/\<\varepsilon^2\>$).
    Consider the $\cC$-species with $Q = \bullet \longrightarrow \bullet$ ($A_2$ quiver), $\cM_1=\cM_2 = \mod_{B_2(1)}$ and $\cF_\alpha =\id$ is the identity functor.
    Then one sees that $\mod_{\hat{A}}$ from \cref{eg:introA2} is equivalent to $\rep(Q,\cM_\bullet, \cF_\bullet)$, hence is a $1$-Iwanaga--Gorenstein $\cC$-module category by \cref{prop:species-1IG}.
\end{example}

\begin{remark}
    Our notion of a $\cC$-species is a ``categorical version'' of a (quasi-)species in $\cC$ defined in \cite[\S3]{HS_species} in terms of algebras and bimodule objects \emph{in $\cC$} (not finite dimensional algebras) -- each $\cM_i$ corresponds to the category of modules in $\cC$ over the semisimple\footnote{A \emph{simple algebra} in $\cC$ means it has no non-trivial two-sided ideal, and a \emph{semisimple algebra} is by definition a product of simple algebras. The category of modules $\mod_A(\cC)$ over an algebra $A$ in $\cC$ is a $\cC$-module category, and it is exact if and only if the algebra $A$ is semisimple \cite[Theorem 7.1]{CSZ}. In particular, $\mod_A(\cC)$ for semisimple $A$ is always quasi-Frobenius, but not necessarily semisimple.} algebra $B_i$ in $\cC$ and each $\cF_\alpha$ corresponds to the bimodule $E_{ij}$ in $\cC$.
    In fact, it is easy to see that $\rep(Q,\cM_\bullet,\cF_\bullet)$ is equivalent to the category of modules over the associated path (tensor) algebra in $\cC$ as defined in \cite[Definition 3.3]{HS_species} and also studied in \cite[Definition 3.2]{EKW_tensoralg}.
    From \cref{prop:species-1IG} it follows that the category of modules over these path algebras associated to an acyclic $\cC$-species is always $1$-Iwanaga--Gorenstein.
\end{remark}

\begin{remark}\label{rmk:pro-species}
    Since the corresponding monomorphism (sub)category (\cite[Definition 3.21]{GKKP_monocat}) of the category of representations of a pro-species is the same as the subcategory of Gorenstein projectives \cite[Proposition 3.8]{Kulshammer_prospecies}, \cref{prop:species-1IG} above provides a different characterisation of the monomorphism category in terms of $\cC$-projectives in this setting.
\end{remark}

%We thus see that $\cC$-species can be used to study categories of modules over $1$-Iwanaga--Gorenstein algebras which carries the structure of a $\cC$-module category.
\subsection{Algebras of Geiss--Leclerc--Schr\"oer}
As examples, we consider the 1-Iwanaga--Gorenstein algebras $H(C,D,\Omega)$ introduced by Geiss--Leclerc--Schr\"oer in \cite{GLS_quiverI}. We show that if the symmetriser $D$ contains at most two values $1$ and $d$, $\mod_{H(C,D,\Omega)}$ can be studied via appropriately defined $\cC$-species.
In what follows, we shall further assume that $\Bbbk$ is a field containing primitive $d$-th roots of unity.

Let $C = (c_{ij}) \in M_n(\mathbb{Z})$ be a symmetrisable generalised Cartan matrix and $\Omega$ an acyclic orientation on the (labelled) graph corresponding to $C$.
We define $Q^\circ\coloneqq Q^\circ(C,\Omega)$ to be the quiver with 
\begin{itemize}
    \item set of vertices $Q_0 = \{1,\ldots,n\}$; and
    \item for each $j \to i \in \Omega$, we have arrows $\alpha_{ij}^{(g)}:j \to i \in Q^\circ_1$ for each $1 \leq g \leq g_{ij}$, where $g_{ij} \coloneqq |\gcd(c_{ij},c_{ji})| = g_{ji}$.
\end{itemize}
We further define $Q\coloneqq Q(C,\Omega)$ to be the quiver obtained from $Q^\circ$ by adding a loop $\varepsilon_i$ to each vertex $i$.

Let $D=\text{diag}(c_1,\ldots,c_n)$ be a symmetriser of $C$ (which is a diagonal matrix).
In this paper, we will restrict to the case where $c_i \in \{1,d\}$ for $d \geq 1$.
Under this assumption, the algebra $H(C,D,\Omega)$ in \cite{GLS_quiverI} has the following simplified definition: 
\begin{definition}[\protect{\cite[\S 1.4]{GLS_quiverI}}]
Suppose the entries $c_i$ of $D$ are either $1$ or $d\geq 1$. 
We define $H(C,D,\Omega) \coloneqq \Bbbk Q/ I$ with $I$ generated by the following relations
\begin{enumerate}
    \item[(H1)] for each $i \in Q_0$, $\varepsilon_i^{c_i} = 0$;
    \item[(H2')] for each $i,j \in Q_0$ such that $c_i=d=c_j$, $\varepsilon_i\alpha^{(g)}_{ij} = \alpha^{(g)}_{ij}\varepsilon_j$ for each $1\leq g\leq g_{ij}$.
\end{enumerate}
\end{definition}

Set $B_\ell \coloneqq \Bbbk[\varepsilon]/\<\varepsilon^\ell\>$.
From the description above it is easy to see that a $H(C,D,\Omega)$-module $V$ is equivalent to giving a $B_{c_i}$-module $V_i$ for each $i \in Q_0^\circ$, and for each arrow $\alpha=\alpha_{ij}^{(g)}:j \to i$ a linear map $V_\alpha:V_j \to V_i$ that is moreover a morphism of $B_d$-modules when $c_i=d=c_j$.
In other words, modules over $H(C,D,\Omega)$ are simply representations over a pro-species \cite{Kulshammer_prospecies}.

Recall that $\mod_{B_d}$ and $\mod_{B_1} = \vec$ are both exact module categories over $\cC_d \coloneqq \mod_{T_d(\zeta)}$ with $T_d(\zeta)$ the Taft algebra (see \cref{eg:Taft} and \ref{eg:exactmod-Taft}).
Hence we can define a $\cC_d$-module structure on $\mod_{H(C,D,\Omega)}$ by upgrading the above pro-species to the following $\cC_d$-species. 
We define $(Q^\circ,\cM_\bullet,\cF_\bullet)$ with $\cM_i = \mod_{B_{c_i}}$, and for each $\alpha=\alpha_{ij}^{(g)}:j \to i$, $\cF_\alpha: \cM_j \to \cM_i$ is the $\cC_d$-module functor that is the identity functor when $c_i=c_j$, the forgetful functor when $c_j=d$ and $c_i=1$ and the extension-of-scalars functor $B_d\otimes_\Bbbk -$ when $c_j=1$ and $c_i=d$.
Note that the $\cC_d$-module functor $B_d\otimes_\Bbbk -$ has a right adjoint given by the forgetful functor, so that
\[
\Hom_{B_d}(B_d \otimes_\Bbbk V_j, V_i) \cong \Hom_\Bbbk(V_j, V_i).
\]
Thus it follows from the description above that $\mod_{H(C,D,\Omega)}$ is equivalent, as abelian categories, to $\rep(Q^\circ, \cM_\bullet, \cF_\bullet)$.
By transferring the $\cC$-module structure, $\mod_{H(C,D,\Omega)}$ is therefore a $\cC_d$-module category.

For the following, we recall that a $H(C,D,\Omega)$-module $V$ is \textbf{locally free} if $V_i$ is a free $\Bbbk[\varepsilon_i]/\<\varepsilon_i^{c_i}\>$-module for all $i \in Q_0$; it is moreover \textbf{$\tau$-locally free} if $\tau^k(V)$ is locally free for all $k \in \mathbb{Z}$. 
\begin{proposition}\label{prop:GLS}
    Let $H(C,D,\Omega)$ ($c_i\in\{1,d\}$) be the algebra defined above and view 
    \[
    \cM \coloneqq \mod_{H(C,D,\Omega)} \cong \rep(Q^\circ, \cM_\bullet, \cF_\bullet)
    \]
    as a module category over $\cC_d = \mod_{T_d(\zeta)}$ as in the preceding discussion.
    Suppose that $C$ is the (generalised) Cartan matrix associated to a Dynkin diagram.
    Then the following are equivalent for indecomposable $V \in \cM$:
    \begin{enumerate}
        \item \label{item:preproj} $V$ is preprojective;
        \item \label{item:preinj} $V$ is preinjective;
        \item \label{item:taulf} $V$ is $\tau$-locally free; 
        \item \label{item:lfrigid} $V$ is locally free and $\Ext_\cM^1(V,V)=0$;
        \item \label{item:PM} $V \in \cP\cM$.
    \end{enumerate}
    Moreover, the indecomposables in $\cP\cM$ are in bijection with the positive roots of the corresponding Dynkin diagram, hence $\cM$ is $\cC$-rep-finite.
\end{proposition}
\begin{proof}
    The equivalence between \ref{item:preproj}, \ref{item:preinj}, \ref{item:taulf} and \ref{item:lfrigid} is given in \cite[Theorem 1.3(ii)]{GLS_quiverI}.
    
    $\ref{item:preproj} \implies \ref{item:PM}$: If $V$ is preprojective, then $V \in \cP\cM$ since $\proj(\cM)\subseteq \cP\cM$ and $\cP\cM$ is closed under Auslander--Reiten translations by \cref{thm:PM-ARtrans}.
    
    $\ref{item:PM} \implies \ref{item:taulf}$: On the other hand, any $V \in \cP\cM$ is locally free, since it is a summand of an object of the form $P\lact W$ for some non-zero $P \in \proj(\cC_d)$, and $(P\lact W)_i = P\lact W_i$ is necessarily projective in $\mod_{\Bbbk[\varepsilon_i]/\<\varepsilon_i^{c_i}\>}$ (by exactness), which has to be free since the unique indecomposable projective is $\Bbbk[\varepsilon_i]/\<\varepsilon_i^{c_i}\>$.
    It follows that $V \in \cP\cM$ is $\tau$-locally free since $\cP\cM$ is closed under Auslander--Reiten translations by \cref{thm:PM-ARtrans}.

    The final statement follows from \cite[Theorem 1.3(i)]{GLS_quiverI}.
\end{proof}

In light of the result above, it is an interesting problem to classify $\cC$-species whose representation categories are $\cC$-rep-finite. When $\cC$ is semisimple, classification for the special case of $\cC$-species known as $\cC$-quivers (i.e.\ $\cM_i=\cC$) is done in \cite[Theorem A]{EH_fusionquiver}, where the classification is controlled by Coxeter--Dynkin diagrams which classify finite Coxeter groups (though in this case $\cC$-rep-finite is the same as representation-finite since $\cC$ is semisimple).

\printbibliography
\end{document}